\documentclass[12pt]{amsart}

\usepackage{verbatim,amssymb,latexsym,amscd}
\usepackage{bbold}
\usepackage[all]{xy}
\usepackage[colorlinks,linkcolor=blue,citecolor=blue,urlcolor=red]{hyperref}
\usepackage[T1]{fontenc}

\newcommand{\ie}{{\it i.e. }}
\newcommand{\cf}{{\it cf. }}
\newcommand{\eg}{{\it e.g. }}
\newcommand{\loccit}{{\it loc. cit. }}
\newcommand{\resp}{{\it resp. }}
\newcommand{\un}{\mathbb{1}}
\newcommand{\A}{\mathbf{A}}

\newcommand{\C}{\mathbf{C}}

\renewcommand{\L}{\mathbb{L}}
\newcommand{\T}{\mathbb{T}}
\newcommand{\N}{\mathbf{N}}
\renewcommand{\P}{\mathbf{P}}
\newcommand{\Q}{\mathbf{Q}}

\newcommand{\V}{\mathcal{V}}
\newcommand{\Z}{\mathbf{Z}}
\newcommand{\sA}{\mathcal{A}}
\newcommand{\sB}{\mathcal{B}}
\newcommand{\sC}{\mathcal{C}}
\newcommand{\sD}{\mathcal{D}}
\newcommand{\sE}{\mathcal{E}}

\newcommand{\sH}{\mathcal{H}}
\newcommand{\sI}{\mathcal{I}}

\newcommand{\sM}{\mathcal{M}}
\newcommand{\sN}{\mathcal{N}}
\newcommand{\sO}{\mathcal{O}}
\newcommand{\sR}{\mathcal{R}}
\newcommand{\sS}{\mathcal{S}}

\newcommand{\sV}{\mathcal{V}}
\newcommand{\sW}{\mathcal{W}}

\newcommand{\uH}{\underline{H}}

\newcommand{\Spec}{\operatorname{Spec}}

\newcommand{\sAb}{\operatorname{\mathcal{A}{\sf b}}}

\newcommand{\id}{\operatorname{id}}
\newcommand{\Id}{\operatorname{Id}}
\newcommand{\Tr}{\operatorname{Tr}}
\newcommand{\Ker}{\operatorname{Ker}}

\newcommand{\IM}{\operatorname{Im}}
\newcommand{\ind}{{\operatorname{ind}}}

\newcommand{\add}{{\operatorname{add}}}
\newcommand{\adm}{{\operatorname{adm}}}

\newcommand{\tr}{{\operatorname{tr}}}

\newcommand{\Hom}{\operatorname{Hom}}
\newcommand{\End}{\operatorname{End}}
\newcommand{\Mod}{\text{\rm Mod--}}

\newcommand{\Pic}{\operatorname{Pic}}

\newcommand{\cl}{{c\ell}}

\newcommand{\Add}{{\operatorname{\bf Add}}}
\newcommand{\Ex}{{\operatorname{\bf Ex}}}

\newcommand{\car}{\operatorname{char}}

\newcommand{\rat}{{\operatorname{rat}}}
\newcommand{\alg}{{\operatorname{alg}}}
\renewcommand{\hom}{{\operatorname{hom}}}
\newcommand{\hun}{{\operatorname{hun}}}
\newcommand{\hum}{{\operatorname{hum}}}
\newcommand{\num}{{\operatorname{num}}}
\newcommand{\mult}{{\operatorname{mult}}}

\newcommand{\rig}{{\operatorname{rig}}}

\newcommand{\tnil}{{\operatorname{tnil}}}
\newcommand{\Cat}{\operatorname{\bf Cat}}
\newcommand{\Corr}{\operatorname{Corr}}

\newcommand{\Sm}{\operatorname{Sm}}

\newcommand{\ab}{{\operatorname{ab}}}
\newcommand{\proj}{{\operatorname{proj}}}

\renewcommand{\Vec}{\operatorname{Vec}}
\newcommand{\Weil}{\operatorname{\bf Weil}}
\newcommand{\weil}{\sW}

\newcommand{\Ad}{\operatorname{Ad}}
\newcommand{\Wl}{\operatorname{Wl}}

\newcommand{\op}{{\text{\rm op}}}
\newcommand{\eff}{{\text{\rm eff}}}
\newcommand{\sat}{{\text{\rm sat}}}

\newcommand{\by}{\xrightarrow}

\newcommand{\iso}{\by{\sim}}

\newcommand{\inj}{\hookrightarrow}

\newcommand{\surj}{\rightarrow\!\!\!\!\!\rightarrow}

\renewcommand{\lim}{\varprojlim}

\newcommand{\gr}{{\operatorname{gr}}}

\renewcommand{\qed}{\hfill $\Box$\medskip}

\newcommand{\fake}[1]{\overset{\boldsymbol{\cdot}}{#1}}

\renewcommand{\phi}{\varphi}
\renewcommand{\epsilon}{\varepsilon}

\newcounter{spec}
\newenvironment{thlist}{\begin{list}{\rm{(\roman{spec})}}%
{\usecounter{spec}\labelwidth=20pt\itemindent=0pt\labelsep=10pt}}%
{\end{list}}%

\numberwithin{equation}{section}

\newtheorem{Thm}{Theorem}
\newtheorem{Prop}{Proposition}
\newtheorem{thm}{Theorem}[subsection]
\newtheorem{lemma}[thm]{Lemma}
\newtheorem{prop}[thm]{Proposition}
\newtheorem{cor}[thm]{Corollary}
\newtheorem{cons}[thm]{Construction}

\theoremstyle{definition}

\newtheorem{defn}[thm]{Definition}
\newtheorem{nota}[thm]{Notation}
\newtheorem{rk}[thm]{Remark}
\newtheorem{rks}[thm]{Remarks}

\newtheorem{ex}[thm]{Example}
\newtheorem{exs}[thm]{Examples}

\newtheorem{warning}[thm]{Warning}

\begin{document}
\title{Universal Weil cohomology}
\author{Luca Barbieri-Viale}
\address{Dipartimento di Matematica ``F. Enriques'', Universit{\`a} degli Studi di Milano\\  Via C. Saldini, 50\\ I-20133 Milano\\ Italy}
\email{luca.barbieri-viale@unimi.it}
\author{Bruno Kahn}
\address{CNRS, Sorbonne Université and Université Paris Cité, IMJ-PRG\\ Case 247\\4 place
Jussieu\\75252 Paris Cedex 05\\France}
\email{bruno.kahn@imj-prg.fr}
\begin{abstract}
We construct a new Weil cohomology for smooth projective varieties over a field, universal among Weil cohomologies with values in rigid additive tensor categories.  A similar universal problem for Weil cohomologies with values in rigid abelian tensor categories also has a solution. We give a variant for Weil cohomologies satisfying more axioms, like Weak and Hard Lefschetz. As a consequence, we get a different construction of André's category of motives  for motivated correspondences and show that it has a universal property. 

This theory extends over suitable bases.
\end{abstract}
\maketitle
\tableofcontents
\newpage 

\section{Introduction}
The main result of this article is:

\begin{Thm}\label{tmain} Over any field $k$, there exists a universal Weil cohomology.
\end{Thm}

We wish the story were so simple; actually we have a commutative square of universal Weil cohomologies:
\begin{equation}\label{eqmany}
\begin{CD}
(\sW,W)@>>> (\sW_\ab,W_\ab)\\
@VVV @VVV\\
(\sW^+,W^+)@>>> (\sW_\ab^+,W_\ab^+)
\end{CD}
\end{equation}
(see \S\S \ref{s1.1a} and \ref{intro1.2} below for the notation).
Here, the top left is universal for Weil cohomologies verifying a standard list of axioms, with values in rigid $\Q$-linear symmetric monoidal categories; same for the top right, but replacing additive by abelian. The bottom row is similar, except that we impose extra axioms, the most important being a weak and a strong Lefschetz property.

To muddy the water a little more, but to add flexibility to our construction, we work as in \cite{A} with respect to a given class of smooth projective varieties verifying certain stability conditions: this class was not displayed in \eqref{eqmany} for simplicity. 

The history of this line of investigation is well-known: it goes back to Grothendieck and is excellently summarised by Serre in \cite{serre-motifs}. His survey shows two things: firstly, how the issue of the relationships between the Weil cohomologies known at the time is the genesis of Grothendieck's theory of motives. Secondly, that he was not so much interested in the universal problem we solve here as in the Tannakian aspect, yielding motivic Galois groups (see \cite[VI.A.4.2]{saa} and the two quotations from \emph{Récoltes et Semailles} in \cite{serre-motifs}). As Serre writes, Grothendieck did expect the category $\sM_\num(k)$ of motives modulo numerical equivalence to be abelian, semi-simple and initial with respect to all (vector space-valued) Weil cohomologies, and he knew that this would follow from the standard conjectures he formulated in  \cite{gro-standard} (see \cite{kdix}, \cite[VI, Appendix]{saa}). The semi-simplicity of $\sM_\num(k)$ has now been proven by Jannsen \cite{jannsen3} independently of the standard conjectures, but its initiality in the above sense remains dependent on them (and essentially equivalent to them). 

One might therefore expect that the categories of \eqref{eqmany} resolve this issue by mapping naturally to $\sM_\num(k)$. Surprisingly, we shall see  that it is not the case unless\dots\ one assumes some conjectures! 

Let us now describe the contents of the article.

\subsection{Construction}\label{s1.1a} In Section \ref{s4}, we lay the ground to formulate Theorem \ref{tmain}. Starting from an admissible class $\sV$ of smooth projective $k$-varieties  (Definition \ref{d3.1}), we give in Definition \ref{d1.1} our  axioms for a Weil cohomology $H$ on $\sV$ with values in a  $\Q$-linear symmetric monoidal category $\sC$. Shorthand notation: $(\sC,H)$ (see Notation \ref{not1}). There are two reformulations: one in terms of Chow motives (Proposition \ref{p1.1}), and one in terms of Chow correspondences (Proposition \ref{p1.3}). It is the latter which is used to construct the pair $(\sW,W)$ of \eqref{eqmany} in Theorem \ref{thm2}, by generators and relations. For generators, the main input is the construction of Theorem \ref{thm1m} due to Levine. We apply it to the $\otimes$-category $\Corr\times \N$, where $\Corr$ is the category of Chow correspondences: this is the key new idea in this work. Then we get to $(\sW,W)$ step by step. To pass from $\sW$ to $\sW_\ab$, we just apply the $2$-functor $T$ from \cite{BVK} (Corollary \ref{cor:thm2}). 

\begin{Prop}\label{P1}\

\noindent a)  (Proposition \ref{prop:alg}) If $k$ is separably closed and $\sV$ contains curves, $W$ factors through algebraic equivalence.\\
b)  (Remark \ref{rk:tweil}) Without any condition on $k$ and $\sV$, $W_\ab$ factors through Voevodsky's smash-nilpotence equivalence. 
\end{Prop}

Assuming that $\sV$ is stable under taking hyperplane sections, we introduce in  Definition \ref{d5.2b} the extra properties of Weil cohomologies we alluded to above, and prove the analogue of Theorem \ref{thm2} and Corollary \ref{cor:thm2} (the bottom row of \eqref{eqmany}) in Theorem \ref{thm2m}. Besides the Lefschetz properties (Definition  \ref{d5.2a}), we incorporate two others: normalised character (Definition \ref{d3.4}) and Albanese invariance (Definition \ref{defalbinv}). The latter is explicitly considered in Kleiman's first article on the standard conjectures \cite{kdix} and not much elsewhere; the former seems to have been overlooked in the literature (except in \cite[Def. 3.41]{zetaL}), while it is natural and necessary. All are verified by the classical Weil cohomologies of Definition \ref{d3.2}. For want of a better terminology, we call the Weil cohomologies having these properties \emph{tight}.

\subsection{Varying the Weil cohomology} \label{intro1.2} Given  $(\sC,H)$, we can ``push-forward'' $H$ through any additive $\otimes$-functor $F:\sC\to \sD$, getting another Weil cohomology $(\sD,F_*H)$ (see \eqref{eq.fu}). The inverse process is more familiar when $F$ is faithful, and is usually called ``enrichment''  (think of Hodge versus Betti cohomology, $\ell$-adic representations versus $\ell$-adic cohomology as such). We show in Theorem \ref{thm2bis} that any $H$ admits an \emph{initial} enrichment $\weil_H$, which is a quotient of $\sW$; if $H$ is abelian-valued, it has similarly an ``ab-initial'' enrichment $\weil_H^\ab$ which is a localisation of $\sW_\ab$ (\emph{ibid.}). The same holds in the tight context (Theorem \ref{thm2tris}). 

The situation is especially interesting when $H$ is \emph{classical} (Definition \ref{d3.2}). We then have a picture analogous to \eqref{eqmany} (see also \eqref{eq+abnat}):
\begin{equation}\label{eqnottoomany}
\begin{CD}
(\sW_H)^\natural@>\iota^\natural >> \sW_H^\ab\\
@V\epsilon^\natural VV @V\epsilon^\ab VV\\
(\sW_H^+)^\natural @>\iota^{+,\natural}>> \sW_H^{\ab,+}\\
@V\rho VV\\
\sM^A_H
\end{CD}
\end{equation}
where $(\ )^\natural$ means pseudo-abelian completion and $\sM^A_H$ is André's category of motives for motivated cycles attached to $H$ \cite{A,A2}. 

\begin{Thm}[Theorems \ref{t6.4},  \ref{thm2tris}, \ref{p5.5}, \ref{p8.3}]\label{T2} 
The $\otimes$-functors $\epsilon^\ab$ and $\rho$ are $\otimes$-equivalences, and $\iota^{+,\natural}$ is a $\otimes$-equivalence if and only if $\sM^A_H$ is abelian. This holds in characteristic $0$, and then $\weil_H^\ab$ is semi-simple.
\end{Thm}

That $\rho$ is an equivalence gives a completely different construction of $\sM^A_H$, and provides it with a universal property. The statement in characteristic $0$ is true because of \cite[Th. 0.4]{A}. 

Note that $\weil_H^\ab$ has the same universal property as the one in \cite[Th. 1.7.13]{hms} and \cite[Th. 2.20]{BVHP}, so we recover that construction in a different way. In particular, Theorem \ref{T2} provides an extension of \cite[Prop. 10.2.1]{hms} from characteristic $0$ to any characteristic.

\subsection{Conjectures on algebraic cycles} We now enter the realm of the standard, and less standard, conjectures. How do they interact with the present constructions? Since we allow ourselves to vary the class $\sV$ and since some of these conjectures are true in certain cases, and as in \cite[beg. §2]{kdix}, we adopt here the terminology ``conditions'' instead.

The first and most obvious condition to study is Condition C, algebraicity of the Künneth projectors. Another one is Condition D: that homological equivalence (for the Weil cohomology under study) agrees with numerical equivalence.\footnote{This ``standard conjecture'' is sometimes attributed to Grothendieck. In fact, it is not mentioned in \cite{gro-standard} and goes back at least to Tate in \cite[p. 97]{woodshole}, for $\ell$-adic cohomology.} For classical Weil cohomologies, it is known that D implies C;  the proof is really $D\Rightarrow B\Rightarrow C$, where $B$ is a standard conjecture ``of Lefschetz type''. 

We don't know any proof of the implication $D \Rightarrow C$ for a general Weil cohomology. To get it, we didn't see another way than to pass through a generalisation of the standard conjecture B. (Anyone who knows a direct proof should contact us immediately.) In order to express this ``standard condition'', we need a Hard Lefschetz property: besides Theorem \ref{T2}, this is the main reason to introduce and study tight Weil cohomologies.

The ``yoga'' of the standard conjectures, \ie the interplay between conjectures B, C and D (and a conjecture A), as well as the ``Hodge positivity'' conjecture, is studied in detail by Kleiman in \cite{kdix} and \cite{kst} for a Weil cohomology with values in vector spaces over a field (``traditional'' in the sense of Definition \ref{d3.2}). The good news is that most of this formalism goes through in our generalised context, basically without change. This is done in \S\S \ref{s8.5} and \ref{s8.6}; the only places where we cannot use Kleiman's arguments is where he employs the Cayley-Hamilton theorem, for the proofs of D $\Rightarrow$ B and of the independence of B from the choice of a polarisation. In the first case (Theorem \ref{t8.1} (2)) we replace it by an argument due to Smirnov \cite{sm}, but in the second case  (Theorem \ref{t8.1} (2)) we have to restrict to enrichments of traditional Weil cohomologies. For the Hodge positivity, see \S \ref{s8.7}. In \S \ref{s8.8} we also study ``fullness conditions'', generalising the Hodge and Tate conjectures in the style of \cite[Ch. 7]{andremotifs}, and their interplay with the standard conditions. 

The reader may feel uncomfortable with the idea of playing with all these conjectures for Weil cohomologies which are more general than those considered traditionally; at least for abelian-valued Weil cohomologies, this is justified by Voevodsky's conjecture \cite[Conj. 4.2]{voe} in view of Proposition \ref{P1} b). These conjectures clarify the rather complex picture of this paper: we urge the reader to look at \S \ref{s9.1} for details.

\enlargethispage*{40pt}

\subsection{Abelian varieties} This is a case where many conjectures are known and where, therefore, results from the previous subsection apply partially. We develop this in \S \ref{s7.4}. In particular, the category  $\sW$ in this case is (up to idempotent completion) of the form $\sM_\sim$ for a very explicit adequate equivalence $\sim$ (Theorem \ref{t9.1}). 

\subsection{Variation over a base} The present theory extends to smooth projective schemes over suitable bases without change: this is briefly explained in \S \ref{s10}. This gives it a flexibility which may be useful in future applications.

\subsection{Graded Weil cohomologies} In the appendix, we show that Theorem \ref{tmain} implies a similar statement for a suitable version of  Saavedra's $\Z$-graded cohomology theories from \cite[VI.A.1.1]{saa}.

\subsection{What is not done here} We list here some lines of investigation that we haven't attempted to follow in this paper.

\subsubsection{} Links with the Tannakian picture.

\subsubsection{} Relationship with the new cohomology theory proposed by Ayoub in \cite{AyW},   
and with Scholze's conjecture  \cite{scholze} on the existence of a generalised Weil cohomology over the algebraic closure of a finite field with values in the semisimple $\Q$-linear tensor category of representations of Kottwitz gerbes ‘‘that practically behaves like a universal cohomology theory''.

\subsubsection{} A triangulated or $\infty$-theoretical version covering the mixed Weil cohomologies of Cisinski-D\'eglise \cite{CisDeg}, see also \cite{AyW2}. This would allow in particular to use integral coefficients, which would be artificial in the present paper since a Künneth isomorphism for cohomology groups only holds up to torsion.

\subsection{Acknowledgements} Our debt to Grothendieck goes without saying. We are also indebted to Nori's construction of an abelian category of mixed motives (see \cite{hms}): it was our initial motivation for this work and we intend to pursue a search for a universal version of it, in the present style. Finally we are indebted to André's work on motives, especially \cite{A}, and not only for Theorem \ref{T2}: we borrowed several of his ideas, like varying the class of smooth projective ``models'' with which we work in Definition \ref{d3.1}, and streamlining Jannsen's construction of the categories of pure motives (see Footnote \ref{fn5}).

Each author gratefully acknowledges the support of the other's institution for several back-and-forth visits since August 2021, having led to the completion of this work.

\section{Reminders on $\otimes$-categories}
\subsection{Terminology}\label{Term} Here we adopt the terminology of \cite[VII, \S 1 and XI, \S\S 1,2]{mcl}, with some minor modifications.
 
For categories, we say \emph{$\otimes$-category} for unital symmetric monoidal category, and \emph{monoidal (\resp $\otimes$-)category without unit} for a monoidal (\resp symmetric monoidal) category in which no unit structure is provided.

A \emph{$\otimes$-functor} $ (F, \mu, \eta ): (\sC,\otimes_\sC , \un_\sC)\to (\sD,\otimes_\sD, \un_\sD)$ between $\otimes$-categ\-or\-ies is given by a functor $F: \sC \to \sD$,  a natural transformation
$$\mu_{X, Y}: F (X)\otimes_\sD F(Y)\to F(X\otimes_\sC Y)$$ 
compatible with the associativity constraints and the symmetry isomorphisms, and a morphism $\eta :  \un_\sD \to F(\un_\sC)$ satisfying the unitality condition. A $\otimes$-functor is \emph{strong} if $\mu$ and $\eta$ are isomorphisms and \emph{strict} if $\mu$ and $\eta$ are equalities. 

In this paper, we shall use the three types of $\otimes$-functors: to avoid ambiguities, we 
write \emph{lax} $\otimes$-functor instead $\otimes$-functor.

A \emph{$\otimes$-natural transformation} between $\otimes$-functors is a natural transformation which it is compatible with the $\mu$'s and the units.

An \emph{additive $\otimes$-category} is a $\otimes$-category which is additive and such that the tensor product is biadditive. An additive $\otimes$-functor between additive $\otimes$-categories is a (lax, strong, strict) $\otimes$-functor which is additive. This tensor structure carries over canonically to the pseudo-abelian completion.

\subsection{Notation} \label{Not}
Let $\Cat^\otimes$ be the $2$-category of $\otimes$-categories, lax  $\otimes$-functors and $\otimes$-natural isomorphisms.
 Let $\Add^\otimes$ be the 2-category of additive $\otimes$-categories,  strong  additive $\otimes$-functors and $\otimes$-natural isomorphisms.  Let $\Ex^\otimes$ be the 2-category of abelian $\otimes$-categories, exact strong  $\otimes$-functors and $\otimes$-natural isomorphisms. Let  $\Add^\rig$ (\resp $\Ex^\rig$) be the 2-full, 1-full subcategory of $\Add^\otimes$ (\resp $\Ex^\otimes$) given  by rigid categories. 

Recall that for $\sC\in \Add^\otimes$,  $Z(\sC):=\End_\sC(\un)$ is a commutative ring \cite[I.1.3.3.1]{saa}, and that the category $\sC$ is $Z(\sC)$-linear. 

\begin{defn}\label{d2.1} Let $\sA\in \Ex^\rig$. We say that $\sA$ is \emph{connected} if $Z(\sA)$ is a field.
\end{defn}

\subsection{Useful tools} The following results, which were already used in \cite{BVK}, will be used here several times so we recall them for quotation purposes.

\begin{lemma}\label{l2.1} Let $\sA,\sB\in \Ex^\rig$ with $\sA$ connected (Definition \ref{d2.1}).\\
a) \cite[Prop. 1.19]{delmil} Any exact $\otimes$-functor $F:\sA\to\sB$ is faithful. In particular, if $F$ is a localisation it is an equivalence.\\
b) \cite[Th. 2.4.1 and Rk. 2.4.2]{CEOP} The converse is true if $\sB$ is also connected.\\
c) \cite[Prop. 3.5 b)]{BVK} $\sA$ and $\sB$ are reduced: for any morphism $f$ and any $N>0$, $f^{\otimes N}=0$ $\Rightarrow$ $f=0$.\\
d)  \cite[Prop. 4.2  and Th. 4.18]{standard-schur}. $Z(\sB)$ is absolutely flat;  there is a $1-1$ correspondence between the Serre $\otimes$-ideals of $\sB$ (\ie Serre subcategories closed under external tensor product) and the ideals of $Z(\sB)$.
\end{lemma}

In the next lemma, $\sC\in \Add^\rig$; recall the $\otimes$-ideal $\sN\subset\sC$ of negligible morphisms, \cf \cite[7.1]{AK2}.

\begin{lemma}[\protect{\cite[Th. 1 a)]{AK3}}]\label{lak} Suppose that there exists a finite extension $L/K$ and a  $K$-linear $\otimes$-functor  $F:\sC\to \sA$ to a $L$-linear rigid $\otimes$-category $\sA\in \Ex^\rig$ in which Homs are finite $L$-dimensional. Then  $\sC/\sN$ is semi-simple and the only $\otimes$-ideal $\sI$ de $\sC$ such that $\sC/\sI$ is semi-simple is $\sI=\sN$.
\end{lemma}

We shall use the following results repeatedly.

\begin{lemma}[\protect{\cite[Th. 6.1 and Cor. 6.2]{BVK}}]\label{bvk}
Let $\sC\in \Add^\rig$. Then the $2$-functor
\[\sA\mapsto \Add^\rig(\sC,\sA)\]
from $\Ex^\rig$ to $\Cat$ is $2$-representable by a category $T(\sC)$. If $\sC$ is abelian, the canonical $\otimes$-functor
\[\lambda_\sC:\sC\to T(\sC)\]
is faithful; if $\sC$ is further semi-simple, $\lambda_\sC$ is an equivalence of categories.
\end{lemma}

\begin{rk} The assumption semi-simple can be weakened to split  \cite[Def. 5.2]{BVK}.
\end{rk}

\begin{lemma}[\protect{\cite[Th. 6.3]{standard-schur}}]\label{t6.1} Let $\sC\in\Add^\rig$ and let $\mathbb{I}$ be a $\otimes$-ideal of $\sC$. Then we have a $\otimes$-equivalence $T(\sC)/\sI\iso T(\sC/\mathbb{I})$, where $\sI$ is the Serre $\otimes$-ideal generated by the $\IM \lambda_\sC(f)$ for $f\in \mathbb{I}$ and the left hand side is the corresponding Serre localisation.\end{lemma}

\begin{lemma}[\protect{\cite[Prop. 2.2.8]{krause}}]\label{l2.3} Let $\sC\subseteq \sC’$ be Serre subcategories of an abelian category $\sA$. Then $\sA/\sC’$ is a Serre localisation of $\sA/\sC$, with kernel $\sC’/\sC$. 
\end{lemma}

For the last lemma, we use the following definition:

\begin{defn}\label{d2.2}
An additive functor  $\iota:\sC'\to\sC$ between additive categories is \emph{dense} (or essentially surjective up to idempotents)  if any object of $\sC$ is isomorphic to a direct summand of an object of $\iota(\sC')$.
\end{defn}

\begin{lemma}\label{l2.2} Let $F:\sC\to \sD$ be a fully faithful additive functor between additive categories, with $\sC$ pseudo-abelian. If there exists a full  subcategory $\sC'\subseteq \sC$ such that $F_{|\sC'}$ is dense, then $F$ is essentially surjective.
\end{lemma}

\begin{proof} Let $D\in \sD$. By hypothesis, there exists $C'\in \sC'$ such that $D$ is isomorphic to a direct summand of $F(C')$. Let $e=e^2\in \End_\sD(F(C'))$ be an idempotent such that $\IM e$ is isomorphic to $D$. Then $e=F(e')$ where $e'$ is an idempotent of $\End_\sC(C')$; if $D'=\IM e'$, then $F(D')=\IM e$.
\end{proof}

\section{K\"unneth formula}
\subsection{A universal construction}

The following theorem is one of our main tools in this paper.

\begin{thm}\label{thm1m}
 Let $(\sC, \times, 1)$ be a $\otimes$-category. Then there exists a lax $\otimes$-functor $(i,\boxtimes,\upsilon ): \sC \to \sC^\kappa$ such that for any other lax $\otimes$-functor $(H, \mu,\eta): \sC\to \sD$ there is a unique  strict $\otimes$-functor $F^\kappa_H: \sC^\kappa\to \sD$ such that $F^\kappa_H i=H$, $F^\kappa_H(\boxtimes)=\mu$ 
and $F^\kappa_H(\upsilon)=\eta$.
\end{thm}

\begin{proof}[Sketch of proof]
The version without units of this theorem is proven in \cite[Part II I.2.4.3]{levine}, see also \cite[Part I I.1.4.3]{levine}: it is obtained from the free $\otimes$-category without unit $(\bar\sC, \otimes)$ on $\sC$ together with freely adjoined morphisms $\boxtimes_{X, Y}: X\otimes Y\to X \times Y$ for $X, Y\in \sC$,  modulo relations providing naturality, associativity and commutativity. The same construction certainly works with units, \emph{mutatis mutandis}; anyway, as Ross Street pointed out, this is a special case of a much more general theorem \cite[Th. 3.5]{BKP} (see \loccit, 6.1 for the link with $\otimes$-categories).
\end{proof}

\subsection{Künneth structures}\label{s1.2} In \cite[Part II Chap. I Def. 2.4.1]{levine}, Levine says \emph{external product} instead of lax $\otimes$-functor, by analogy with external products in cohomology. This reflects in the terminology of this subsection.

\begin{defn}\label{d1} a) Let $(\sC, \times),(\sD, \otimes)$ be two  monoidal categories without unit, and let $(M, +)$ be a commutative semi-group. Let $H = \{H^i\}_{i\in M}$ be a family of functors $H^i:\sC\to \sD$.  An \emph{$M$-graded external product on $H$} is a family $\kappa = \{\kappa^{i,j}\}_{i, j\in M}$ of natural transformations, as shown in the variables $X,Y\in \sC$
\[\kappa^{i,j}_{X,Y}:H^i(X)\otimes  H^j(Y)\to H^{i+j}(X\times Y)\]
which are compatible with the associativity constraints. If $\sC$ and $\sD$ are (pre)additive say that $(H, \kappa)$ is \emph{additive} if all $H^i$'s are additive.\\

 b) We say that $(H,\kappa)$ is \emph{strong} if, for any $X,Y\in \sC$ and $k\in M$, the coproduct $\coprod_{i+j=k} H^i(X)\otimes H^j(Y)$ exists in $\sD$ and the corresponding morphism
\[\coprod_{i+j=k} H^i(X)\otimes  H^j(Y)\to H^{k}(X\times Y)\]
is an isomorphism.\\

c) Let $(\sC, \times, 1)$ be monoidal, $(\sD, \otimes, \un)$ be additive monoidal, and let $(M, +, 0)$ be a commutative monoid. 
A \emph{unital $M$-graded external product} $(H, \kappa, \upsilon)$ is an $M$-graded external product $(H,\kappa)$ such that $H^i(1) =0$ if $i\ne 0$, provided with an additional isomorphism $\upsilon^0 : \un  \iso H^0(1)$
such that the diagram
\[\begin{CD}
\un\otimes H^i(X)@>\sim >> H^i(X)\\
@V\wr VV @V\wr VV\\
H^0(1)\otimes H^i(X) @>\sim >> H^i(1 \times X)
\end{CD}\]
and the symmetric diagram commute for all $X,i$.
\end{defn}

\begin{rk}\label{r1} Consider $M$ as a discrete category.  Then we have two ways to convert the family $H=\{H^i\}_{i\in M}$ into a single functor: 
\begin{enumerate}
\item $H^*:\sC\to \sD^M$, $H^*(X)(i)=H^i(X)$,
\item $\tilde H:\sC\times M\to \sD$, $\tilde H(X, i) = H^ i(X)$.
\end{enumerate}

In (1), assume first that coproducts indexed by $M$ exist in $\sD$; then $\sD^M$ inherits a monoidal structure without unit by the rule
\begin{equation}\label{eq3.1}
(A\otimes B)_k= \coprod_{i+j=k} A_i\otimes B_j 
\end{equation}
where we identify an object $A\in \sD^M$ with a family of objects $A_i\in \sD$ for $i\in M$. Note that $(\sD^M, \otimes)$ automatically inherits an associativity constraint and $(H^*, \kappa)$ is an external product
$$H^*(X)\otimes H^*(Y)\to H^*(X\times Y).$$ 
The conditions to be strong and unital in b) and c) of Definition \ref{d1} amount to say that $H^*$ is a strong (unital) monoidal functor in the usual sense, see \S \ref{Term}. 
Note that the object of  $\sD^M$ with value $\un$ for $i=0$ and $0$ otherwise is a unit object. 
Thus $H$ is unital if and only if $H^*$ is strongly compatible with units. 

If $\sD$ does not admit all coproducts indexed by $M$, \eqref{eq3.1} still makes sense if 
\begin{itemize}
\item either at least one of $A,B$, say $A$, has finite support (\ie the set of $i$'s such that $A_i\ne 0$ is finite);
\item or $M=\N$. 
\end{itemize}

The first condition is of course verified if $H^*$ takes values in $\sD^{(M)}=\{A\in \sD^M\mid A_i=0\text{ for all but finitely many $i$'s}\}$.

In (2), provide $\sC\times M$ with the monoidal structure $(X,i)\times (Y,j)=(X\times Y,i+j)$. If $1$ is a unit for $(\sC, \times)$ and if $(M, +, 0)$ is a monoid, then $(1, 0)$ is a unit for $(\sC\times M, \times)$. 

If $\sC$ is preadditive, then $\sC\times M$ is made preadditive by setting the group of morphisms between $(C,i)$  and $(C',j)$ if $i\ne j$ to be zero. If $(H,\kappa)$ is additive, $\tilde H$ extends canonically to an additive functor on $\sC\times M$  with its natural additive monoidal structure. 
We can then translate a): an (additive) $M$-graded external product on $H$ is an (additive) lax $\otimes$-stucture on $\tilde H$.  Unitality of $H$ in the sense of c) is equivalent to the strong unitality of $\tilde H$, \ie $\un\iso \tilde H(1, 0)$, and $\tilde H(1, i)=0$ for $i\neq 0$. 

Suppose that $\sD$ is additive. Then $\tilde H$ factors through the additive hull of $\sC\times M$, which is nothing else than $\sC^{(M)}$ (send $(X,i)\in \sC\times M$ to $X[i]\in \sC^{(M)}$ where $X[i]_j=X$ if $i=j$ and $0$ otherwise). Thus, in this case, an $M$-graded external product from $\sC$ to $\sD$ is the same as a lax monoidal functor $\sC^{(M)}\to \sD$.
\end{rk}

\begin{defn}\label{d3} Let $(\sC,\times,1)\in \Cat^\otimes$, $(\sD,\otimes,\un)\in \Add^\otimes$ and $M=\Z$. A \emph{Künneth product} is a unital $\Z$-graded external product $(H, \kappa,\upsilon)$  
 satisfying the following condition: for any $X,Y\in \sC$ and any $i,j\in \Z$, the diagram
\[\begin{CD}
H^i(X)\otimes  H^j(Y) @>\kappa^{i,j}>> H^{i+j}(X\times  Y)\\
@V  (-1)^{ij} c VV @VH^{i+j}(c) VV\\
H^j(Y)\otimes H^i(X) @>\kappa^{j,i}>> H^{i+j}(Y\times  X)
\end{CD}\]
commutes, where $c$ denotes both commutativity constraints. We say that $(H, \kappa,\upsilon)$ \emph{satisfies the Künneth formula} if it is strong in the sense of Definition \ref{d1} b). 
\end{defn}

\begin{rk} \label{r2}
As a sequel to Remark \ref{r1}, in (1) Definition \ref{d3} amounts to requiring that $H^*$ is symmetric monoidal for the symmetric monoidal structure on $\sD^\Z$ given by the symmetry $c_{A,B}: A\otimes B\iso B\otimes A$ such that
\[(c_{A,B})_{| A_i\otimes B_j} =  (-1)^{ij}c_{A_i,B_j}.\]
(One can check that this rule does define a symmetric monoidal structure on $\sD^\Z$.) See also Remark \ref{r1} for the case where $\sD$ does not have enough coproducts. We call this the \emph{Koszul constraint}.

In (2), if $\sC$ is preadditive, we provide $\sC\times \Z$ with the commutativity constraint
\[c_{(X,i),(Y,j)}= (-1)^{ij}c_{X,Y}: (X,i)\times (Y,j) \to (Y,j)\times(X,i)\]
where $c$ is the commutativity constraint of $\sC$.

If $\sD$ is additive, this amounts as in Remark \ref{r1} to a lax $\otimes$-functor $\sC^{(\Z)}\to \sD$ for the above Koszul commutativity constraint in $\sC^{(\Z)}$, this functor being strong if and only if $H$ verifies the Künneth formula.
\end{rk}

\begin{rk}\label{r3.3}If $\sC\in \Add^\otimes$, then $\sC^{(\Z)}$, provided with the unital monoidal structure of Remark \ref{r1}, has two symmetric structures: the Koszul constraint of Remark \ref{r2} and the naïve commutativity constraint which does not include the signs of the Koszul rule. The latter will not be used in this text. If $\sC\in \Add^\rig$ then $\sC^{(\Z)}$ is an object of $ \Add^\rig$, and of $ \Ex^\rig$ if $\sC\in \Ex^\rig$.
 
The ``direct sum'' functor
\[\bigoplus:\sC^{(\Z)}\to \sC; \quad (C_i)\mapsto \bigoplus C_i\]
is strong monoidal and unital, but  it is not symmetric for the Koszul constraint. The next subsection deals with this issue.
\end{rk}

\subsection{Gradings}

\begin{defn}\label{d3.5} a) Let $F:\sC\to \sD$ be an additive functor between additive categories, and let $M$ be a set. An \emph{$M$-grading} of $F$ (with finite support) is a factorisation of $F$ into
\[\sC\by{F^*} \sD^{(M)}\by{\bigoplus} \sD\]
where $\bigoplus$ is the direct sum functor. In particular, an $M$-grading of the identity functor of $\sC$ is a section of the direct sum functor; we call it a \emph{weight grading} of $\sC$ and say that $C\in \sC$ is \emph{of weight $m$} if $F^n C=0$ for $n\ne m$.\\
b) Suppose $\sC,\sD\in \Add^\otimes$, $F\in \Add^\otimes(\sC,\sD)$, and $M=\Z$. A \emph{$\Z$-$\otimes$-grading} of $F$ is a $\Z$-grading of $F$ in which $F^*$ is a strong $\otimes$-functor for the naïve commutativity constraint of Remark \ref{r3.3}. If $F$ is the identity functor of $\sC=\sD$, we call this a \emph{weight $\otimes$-grading}.
\end{defn}

\begin{rks} \label{r3.2} a) In Definition \ref{d3.5} a), if $\sC,\sD$ are abelian and $F$ is exact, then $F^*$ is automatically exact: indeed, all $F^m$ for $m\in M$ are exact because a direct summand of an exact sequence is an exact sequence.\\
b)  One can use a weight $\otimes$-grading on $\sC\in \Add^\otimes$ as in Definition \ref{d3.5} b) to change its commutativity constraint by introducing the Koszul rule, \ie multiplying the original commutativity constraint between an object of weight $i$ and and object of weight $j$ by $(-1)^{ij}$. This notation is involutive.
\end{rks}

\begin{nota}\label{n3.1} We write $\fake{\sC}$ for the $\otimes$-category deduced from $\sC$ as in Remark \ref{r3.2} b).
\end{nota}

\begin{lemma}\label{l3.4} a) In Definition \ref{d3.5} a), suppose that 
\[\sD(F^m(C),F^n(C'))=0 \text{ for } C,C'\in \sC \text{ and } m\ne n.\]
If $F$ is dense in the sense of Definition \ref{d2.2}, then $F^*$ extends to a unique $M$-grading of $\sD$, which is a weight $\otimes$-grading in the situation of Definition \ref{d3.5} b).\\ 
b) In Definition \ref{d3.5} a), suppose $F$ faithful. For $C\in \sC$ and $m\in M$, write $\pi^m_C$ for the projector of $\End_\sD(F(C))$ with image $F^m(C)$.  Then the full subcategory of $\sC$
\[\sC'=\{C\in \sC\mid \pi^m_C\in \End_{\sC}(C)\ \forall\ m\in M\}\]
is additive, stable under direct summands, and $\sC'\in \Add^\otimes$ in the situation of Definition \ref{d3.5} b). If $\sC'$ is pseudo-abelian (\eg if $\sC$ is), $F^*_{|\sC'}$ factors  uniquely through a weight grading $\sC'\to (\sC')^{(M)}$, which is a weight $\otimes$-grading in the situation of Definition \ref{d3.5} b).
\end{lemma}

\begin{proof} a) Let $\sD'$ be the full subcategory of $\sD$ given by the $F^m(C)$'s. The hypothesis implies that $F^*$ extends to a unique functor $\sD'\to \sD^{(M)}$, which then extends to $\sD$ by density.

In b), the stability properties of $\sD'$ are obvious except perhaps the stability under direct summands: this holds because the $\pi^m_C$'s are central idempotents in $\End_{\sD^{(M)}}(F^*(C))\subseteq \End_\sD(F^*(C))$. The grading is then defined by sending $C$ to $(\IM \pi^m_C)_{m\in M}$.
\end{proof}

\begin{lemma}\label{l3.2} Let $\sC\in \Add^\otimes$, let $\hat{\sC}=\Mod\sC$ be its additive dual provided with its canonical $\otimes$-structure (``Day convolution''), and let $y_\sC:\sC\to \hat{\sC}$ be the additive Yoneda functor: it is a strong $\otimes$-functor. Then any dualisable object $X$ of $\hat{\sC}$ is a direct summand of an object of the form $y_\sC(C)$.
\end{lemma}

\begin{proof} Since $X$ is dualisable, it is compact, because $\un_{\hat{\sC}}$ is compact (as the Yoneda image of $\un_\sC$) and $\otimes_{\hat{\sC}}$ commutes with arbitrary colimits as a left adjoint of the internal Hom. The conclusion now follows from \cite[Prop. 1.3.6 f)]{AK2}.
\end{proof}

Let $\sC\in \Add^\otimes$. Then $\sC^{(\Z)}\in \Add^\otimes$  by Remark \ref{r3.3}. Observe that $\widehat{\sC^{(\Z)}}$ is canonically $\otimes$-equivalent to $\hat{\sC}^\Z$. Applying Lemma \ref{l3.2}, we get Part a) of the following lemma.

\begin{lemma}\label{l3.3} a) Any dualisable object $X$ of $\hat{\sC}^\Z$ is a direct summand of an object of the form $y_\sC(C)$ for $C\in \sC^{(\Z)}$.\\
b) If $B\in \sC$ and $i\in \Z$, write $B[i]$ for the object $(B_j)_{j\in \N}$ of $\sC^{\Z}$ such that $B_i=B$ and $B_j=0$ for $j\ne i$. Then $C=\bigoplus_{i\in \Z} C_i[i]$ for any $C\in \sC^{(\Z)}$. Moreover, $C$ is dualisable $\iff$ $C_i[i]$ is dualisable for all $i$ $\iff$ $C_i$ is dualisable in $\sC$ for all $i$.
\end{lemma}

\begin{proof}[Proof of b)]
The first claim is obvious, so is the first equivalence and the second is easily checked: the unit and counit do not change.
\end{proof}

\subsection{Relative additive completion}\label{s4.5} Let $F:\sC\to \sD$ be a functor between two categories, with $\sC$ preadditive.
Consider the category $^{F}\sD$ whose objects are functors $G:\sD\to \sE$ such that $\sE$ is preadditive and $G\circ F$ additive, and morphisms $(\sE,G)\to (\sE',G')$ are additive functors $H:\sE\to \sE'$ such that $G'=H\circ G$.

\begin{prop}\label{p4.3} The category $^{F}\sD$ has an initial object ${\rm Add}(F)$.
\end{prop}

\begin{proof} Let $\Z \sD$ be the free preadditive category on $\sD$, and let $\sE_0$ be the category with same objects as $\Z \sD$ (or $\sD$) quotiented by the congruence \cite[p. 52]{mcl} generated by the relations $[F(f+g)]-[F(f)]-[F(g)]$ for $f,g$ parallel morphisms of $\sC$. Let $G_0:\sD \to \sE_0$ be the induced functor. The initiality of ${\rm Add}(F)=(\sE_0,G_0)$ is immediate.
\end{proof}

\begin{cor} Proposition \ref{p4.3} remains true when replacing preadditive by additive in the condition on $\sE$.
\end{cor}

\begin{proof} In Proposition \ref{p4.3}, replace ${\rm Add}(F)$ by its additive hull.
\end{proof}

Here is a symmetric monoidal variant. Let $F:\sC\to \sD$ belong to $\Cat^\otimes$, with $\sC$ preadditive, and $^{F}\sD^\otimes$ be the category whose objects are  strong  $\otimes$-functors $G:\sD\to \sE$ 
such that $\sE$ is preadditive and $G\circ F$ additive, and morphisms $(\sE,G)\to (\sE',G')$ are additive strong $\otimes$-functors $H:\sE\to \sE'$ such that $G'=H\circ G$.

\begin{prop}\label{p4.3t} The category $^{F}\sD^\otimes$ has an initial object $G_0:\sD\to {\rm Add}^\otimes(F)$ which is the identity on objects. This remains true when replacing preadditive by additive in the condition on $\sE$.
\end{prop}

\begin{proof} The free preadditive category $\Z \sD$  on $\sD$ inherits from $\sD$ an additive symmetric monoidal structure. Let $\sE_0^\otimes$ be the category with same objects, quotiented by the congruence $\otimes$-generated by the relations $[F(f+g)]-[F(f)]-[F(g)]$ for $f,g$ parallel morphisms of $\sC$. Let $G_0:\sD \to \sE_0^\otimes$ be the induced functor. The initiality of ${\rm Add}^\otimes(F)=(\sE_0^\otimes,G_0)$ is immediate. The argument for ``additive'' is as before.
\end{proof}

\subsection{Inverting morphisms in additive and $\otimes$-categories} Let $\sC$ be an essentially small additive category, and let $S$ be a set of morphisms of $\sC$. Write $S_\oplus$ for the smallest set of morphisms closed under finite direct sums and containing $S$ and all identities. 

\begin{lemma}\label{l1s2.4}
The Gabriel-Zisman localisation $\sC[S_\oplus^{-1}]$ is additive and the localisation functor $\sC\to \sC[S_\oplus^{-1}]$ is additive; it is universal for additive functors $F:\sC\to \sD$ to other additive categories $\sD$ such that $F(s)$ is invertible for all $s\in S$.
\end{lemma}

\begin{proof} The first claim is \cite[Th. A.3.4]{birat-pure}; the second is obvious, since the hypothesis on $F$ implies that $F(s)$ is invertible also for all $s\in S_\oplus$.
\end{proof}

Replace ``additive category'' by $\otimes$-category in the above, and $S_\oplus$ by $S_\otimes$, in whose definition ``finite direct sums'' is replaced by ``tensor product's'. Similarly, we get

\begin{lemma}\label{l2s2.4}
The Gabriel-Zisman localisation $\sC[S_\otimes^{-1}]$ is symmetric monoidal and the localisation functor $\sC\to \sC[S_\otimes^{-1}]$ is a strong $\otimes$-functor; it is universal for strong $\otimes$-functors $F:\sC\to \sD$ to other $\otimes$-categories $\sD$ such that $F(s)$ is invertible for all $s\in S$.
\end{lemma}
 
 The proof is the same, replacing \cite[Th. A.3.4]{birat-pure} by \cite[Prop. A.1.2]{birat-pure}.
\smallskip

Finally, we can mix the two constructions when starting from an additive $\otimes$-category, according with the notation adopted in  \S \ref{Not}: let $\sC\in \Add^\otimes$ and let $S$ be a set of morphisms of $\sC$. Write $S_{\oplus,\otimes}$ for the smallest set of morphisms which is stable under finite direct sums and tensor products, containing all identities. 

\begin{prop}\label{s2.4} 
The Gabriel-Zisman localisation $\sC[S_{\oplus,\otimes}^{-1}]$ is in $\Add^\otimes$, \ie is additive and symmetric monoidal, and the localisation functor $\sC\to \sC[S_{\oplus,\otimes}^{-1}]$ is a strong additive $\otimes$-functor; it is universal for strong additive $\otimes$-functors $F:\sC\to \sD$ to other additive $\otimes$-categories $\sD$ such that $F(s)$ is invertible for all $s\in S$.\qed
\end{prop}

\section{Generalised Weil cohomologies}\label{s4}

\subsection{The set-up} We work over a field $k$. Let $\Sm^\proj(k)$ be the category of smooth projective $k$-varieties and $k$-morphisms.

Note that to give a subclass of objects of a category is equivalent to give a full subcategory. We are going to use this for $\Sm^\proj(k)$  without further mention. The following definition is in the spirit of \cite[2.1]{A}.

\begin{defn}\label{d3.1}
A class $\V\subset \Sm^\proj(k)$  of smooth projective $k$-varieties is \emph{admissible} if $\Spec k\in \V$, $X\coprod Y, X\times Y\in\V$ for $X,Y\in \V$ and $\V$ is stable under taking connected components. It is \emph{strongly admissible} if, moreover
\begin{itemize}
\item $\P^1\in \V$;
\item $X\in \V$ $\Rightarrow$ $\pi_0(X)\in\V$, where $\pi_0(X)$ is the scheme of constants of $X$.
\end{itemize}
Given any class $S\subset \Sm^\proj(k)$ we write $S_\adm$ for the smallest strongly admissible class containing $S$. 
\end{defn}

 The following are useful examples:

\begin{exs}\label{e3.1} a) $\V=\Sm^\proj(k)$.\\
b) $\V=\{X\in \Sm^\proj(k)\mid \dim X=0\}_\adm$.\\
c) $\V=\{\text{coproducts of abelian $k$-varieties}\}$.\\
d) $\V=\{ A_i\}_\adm$ where $A_i$ runs through the class of abelian schemes over \'etale $k$-schemes.\\ 
e) $\V=\{ C_i\}_\adm$ where $C_i$ runs through the class of geometrically connected $k$-curves.\\ 
f) $\V=\{X \}_\adm$, where $X$ is a fixed smooth projective  $k$-variety.
\end{exs}

For an admissible class $\V\subset  \Sm^\proj(k)$, we define motives modelled on $\V$ as in \cite[(4.2)]{A}: let $\Corr(k,\V)$ be the category of Chow correspondences (with $\Q$ coefficients), such that $\Corr(k, \V)(X,Y)=CH^m(X\times Y)_\Q$ for $X,Y\in \V$ if $X$ is of pure dimension $m$, the general case being obtained by direct sums. 

\begin{defn}\label{d3.1.1}
Denote by $\sM_\rat(k,\V)$  the category of \emph{Chow motives over $k$ \cite{scholl} modelled on $\V$}, with $\Q$ coefficients: objects  $M = (X, p, n)\in \sM_\rat(k,\V)$ are given by $X\in\V$, $p^2 =p\in\Corr(k,\V)(X,X)$ an idempotent, and $n$ a continous (= locally constant) function $X\to \Z$.\footnote{\label{fn5} In Jannsen's original description of pure motives, $n$ is an integer; with Andr\'e's trick \cite[(4.2)]{A} $\sM_\rat(k,\V)$ is immediately seen to be additive and rigid.} 
\end{defn}

By sending a morphism to  its graph  we obtain a  functor 
\[h:\V^{op}\to \sM_\rat(k,\V)\]
which on objects is $h(X) = (X, \id, 0)$ where $\id\in \Corr(k,\sV)(X,X)$ is the diagonal.  We write $\sM_\rat^\eff(k,\V)\subset \sM_\rat(k,\V)$ for the strictly full subcategory of effective motives $M = (X, p, 0)$: the functor $h$ takes values in $\sM_\rat^\eff(k,\V)$.   These categories are additive and $\Q$-linear, and the product of varieties gives them a symmetric monoidal structure with unit $\un= (\Spec k, \id, 0)$. We also write $\L =(\Spec k, \id, -1), \T=(\Spec k, \id, 1)\in \sM_\rat(k,\V)$: these are the Lefschetz  and the Tate motive. We have $\L\otimes \T=\un$; the category $\sM_\rat(k,\V)$ is rigid, the dual of $(X, p, n)$ being $(X, p, d-n)$ where $d:X\to \N$ is the dimension function.

If $\V=\Sm^\proj(k)$, we simply write $\Corr(k), \sM_\rat^\eff(k)$ and $\sM_\rat(k)$; in general, $\Corr(k,\V), \sM_\rat^\eff(k,\V)$ and $\sM_\rat(k,\V)$ are full subcategories of those.


\begin{lemma} If $\P^1\in \V$ (\eg if $\V$ is strongly admissible), then $\L\in \sM_\rat^\eff(k,\V)$.
\end{lemma}
\begin{proof} Indeed, $h(\P^1) = \un\oplus \L$ \cite[1.13]{scholl}.
\end{proof}

Given an admissible class $\V$, we can  define its ``saturation''
\[\V^\sat:=\{X\in \Sm^\proj(k)\mid h(X)\in \sM_\rat(k,\V)\};\]
 these are the \emph{varieties of $\V$ type} (for rational equivalence). We say that $\V$ is \emph{saturated} if $\V=\V^\sat$. Clearly, $\V^\sat$ is saturated.
 
\begin{lemma}\label{l4.7} We have $\sM_\rat(k,\V)=\sM_\rat(k,\V^\sat)$, and $\V^\sat$ is strongly admissible.\end{lemma}

\begin{proof}The first statement is obvious. Then $\P^1\in \V^\sat$ because $h(\P^1)=\un\oplus \L$. Finally,  $h(\pi_0(X))$ is a direct summand of $h(X)$ for any $X \in \V^\sat$ \cite[1.13]{scholl}, which concludes the proof.
\end{proof}

In the sequel, we fix an admissible $\V\subseteq \Sm^\proj(k)$. 
 
\subsection{The axioms} \label{s1.1} Let $\sC$ be an additive $\Q$-linear $\otimes$-category, together with a distinguished invertible object which we denote by $L_\sC$. For any $C\in \sC$ and any $i\in \Z$, we write $C(i):= C\otimes L_\sC^{\otimes -i}$.

\begin{defn}\label{d1.1} A \emph{Weil cohomology on $\V$ with values in $(\sC,L_\sC)$}  is given by
\begin{enumerate}
\item[(a)] a $\Z$-indexed family $H= \{H^i\}_{i\in\Z}:\V^{op}\to \sC$ of functors;
\item[(b)] a Künneth product 
\[\kappa^{i,j}_{X,Y}:H^i(X)\otimes H^j(Y)\to H^{i+j}(X\times Y),\] for every $X, Y\in \V$, and $i, j\in \Z$;
\item[(c)] a trace morphism
\[\Tr_X:H^{2n}(X)(n)\to \un\]
for $X\in \V$ of pure dimension $n$, and 
\item[(d)] a cycle class map given by a $\Q$-linear homomorphism 
\[\cl_X^i:CH^i(X)_\Q\to \sC(\un,H^{2i}(X)(i))\]
for every $X\in \V$ and for every $i\ge 0$.
\end{enumerate}
The data $(H, \kappa, \Tr, \cl)$ are subject to the following axioms:
\begin{thlist}
\item $H^0(\Spec k)\iso\un$ induced by (c) for $X = \Spec k$.
\item If $\dim X=n$, $H^i(X)=0$ for $i\notin [0,2n]$.
\item $H^1(\P^1)=0$ and the trace morphism of (c) for $X = \P^1$  induces an isomorphism $H^2(\P^1)\iso L_\sC$.  
\item \emph{Künneth formula}: $\kappa^{i,j}_{X,Y}$ yields a graded isomorphism 
\[\kappa_{X,Y}: H^*(X)\otimes H^*(Y)\iso H^*(X\times Y)\]
which is natural in $X,Y\in \V$ and verifies the conditions of associativity, unity and graded commutativity. 
\item \emph{Trace and Poincaré duality}: the trace map $\Tr$ is such that $\Tr_{X\times Y}=\Tr_X\otimes \Tr_Y$ modulo the Künneth formula, it is an isomorphism if $X$ is geometrically connected, and the ``Poincaré pairing''
\begin{equation}\label{eq1.1}
H^i(X)\otimes H^{2n-i}(X)\by{\kappa_{X,X}} H^{2n}(X\times X)\by{\Delta_X^*} H^{2n}(X)\by{\Tr_X} \un(-n)
\end{equation}
makes $H^{2n-i}(X)(n)$ the dual of  $H^{i}(X)$.
\item \emph{Cycle classes}: For all $X\in \V$ and all $i\ge 0$, the $\Q$-linear homomorphism $\cl_X^i$ is
contravariant in $X$ and compatible with the Künneth formula and the intersection product. Furthermore, if $\dim X=n$, the diagram
\[\begin{CD}
CH^n(X)_\Q@>\cl_X^n>> \sC(\un,H^{2n}(X)(n))\\
@V\deg VV @V(\Tr_X)_* VV\\
\Q@>>> Z(\sC)
\end{CD}\]
commutes.
\end{thlist}
\end{defn}

\begin{nota}\label{not1} We write $(\sC,H)$ for a Weil cohomology $(H, \kappa, \Tr, \cl)$ with values in $(\sC,L_\sC)$ as in Definition \ref{d1.1}, and simply talk of $H$ as a Weil cohomology with values in $\sC$; if necessary, we shall specify the other implicit data, \eg additionally, $\sC$ will sometimes be assumed to be pseudo-abelian or abelian.
\end{nota}

\begin{rks}\label{r3.1}  a) Let $(\V, \times, 1)$ with $1=\Spec k$ be the natural symmetric mono\-idal structure. Axioms (i), (ii) and (iv) imply that $(H, \kappa)$ satisfies the Künneth formula in the sense of Definition \ref{d3}. Moreover, $H^*:\V^{op}\to \sC^\Z$ takes values in $\sC^{(\Z)}$ and is a strong $\otimes$-functor for $\sC^{(\Z)}$ provided with the $\otimes$-structure given by Remarks \ref{r1} - \ref{r2} (1), as already noted in general.\\
b) The precise meaning of (v) is the following. The morphism
\[H^i(X)\otimes H^{2n-i}(X)(n)\by{\epsilon} \un\]
deduced from \eqref{eq1.1} induces a morphism
\[H^{2n-i}(X)(n)\by{\eta\otimes 1} H^i(X)^\vee\otimes H^i(X)\otimes H^{2n-i}(X)(n)\by{1\otimes \epsilon} H^i(X)^\vee\]
where $\eta$ is the unit map $\un\to H^i(X)^\vee\otimes H^i(X)$; this morphism is an isomorphism. In particular, this requires the existence of $H^i(X)^\vee$ a priori; see however Proposition \ref{p1.4} below.
\end{rks}

\begin{lemma}\label{l3.1}
a) For $X,Y\in \V$, the canonical morphism
\[H^*(X\coprod Y)\to H^*(X)\oplus H^*(Y)\]
is an isomorphism.\\
b) If $X$ is geometrically connected, then the morphism $\un\to H^0(X)$ induced by the projection $X\to \Spec k$ is an isomorphism.
\end{lemma}

\begin{proof} a) The axioms imply that Chow correspondences $\Corr$ act on $H^*$ (see proof of Proposition \ref{p1.1} a) below): in the groups $\Corr(X\coprod Y,X)\simeq \Corr(X,X) \oplus \Corr(Y,X)$ and $\Corr(X\coprod Y,Y)\simeq \Corr(X,Y)\oplus \Corr(Y,Y)$, the graphs of the inclusions $X\inj X\coprod Y$ and $Y\inj X\coprod Y$ have obvious retractions. It is clear that these retractions make $H^*(X\coprod Y)$ a biproduct of $H^*(X)$ and $H^*(Y)$ in the sense of the definition in \cite[VIII.2]{mcl}.

b) follows from Axiom (v). 
\end{proof}

\begin{lemma}\label{l1.1} Let $(\sC, H)$ be a Weil cohomology, and let $X,Y\in \V$, with $X$ of pure dimension $n$. Then we have a canonical isomorphism
\[\sC^{(\Z)}(H^*(X),H^*(Y))\simeq \sC(\un,H^{2n}(X\times Y)(n)).\]
\end{lemma}

\begin{proof} ``As usual'': the proof of \cite[3.45]{zetaL} applies \emph{mutatis mutandis}, using axioms (iv) and (v). 
\end{proof}

\subsection{Some definitions}

\begin{defn}\label{d.ab} A Weil cohomology $(\sC,H)$ is \emph{abelian-valued} if $\sC\in \Ex^\rig$.
\end{defn}

\begin{defn}\label{d3.2} A Weil cohomology $(\sC, H)$ is \emph{traditional} if $\sC=\Vec_K$, the category of finite dimensional vector spaces over a field $K$, with $L_\sC$ a fixed $1$-dimensional $K$-vector space. We say that $H$ is \emph{classical} if it is traditional and belongs to the following list:
\begin{itemize}
\item $\ell$-adic cohomology in any characteristic $\ne \ell$ ($K=\Q_\ell$),
\item Betti or de Rham cohomology in characteristic $0$ ($K=\Q$, \resp $K=k$),
\item crystalline cohomology if $k$ is perfect of characteristic $>0$ ($K=Q(W(k))$, where $W(k)$ is the ring of Witt vectors over $k$).
\end{itemize}
\end{defn}

Obviously, $Z(\sC)=K$ if $H$ is traditional.

\begin{rk}\label{r3.1a} 
When $H$ is traditional, we recover the usual notion of a Weil cohomology as in \cite[3.1.1.1]{andremotifs}.  Condition $H^1(\P^1)=0$ in Definition \ref{d1.1} (iii) is skipped for the latter because it follows from the axioms: by the Lefschetz trace formula, we have
\[\dim H^0(\P^1)-\dim H^1(\P^1)+\dim H^2(\P^1)=\chi(\P^1)=2\]
and $\dim H^0(\P^1)=\dim H^2(\P^1)=1$ (by (iii) without this condition, and (v)), hence $\dim H^1(\P^1)\allowbreak=0$. In general this argument would only give $\chi_\sC(H^1(\P^1))=0$ if $\sC$ is rigid, where $\chi_\sC$ is the Euler characteristic of $\sC$.
\end{rk}

\begin{defn}\label{d3.4}
We say that $H$ is \emph{normalised} if, for any $X\in \sV$ with scheme of constants $\pi_0(X)$, the canonical map $H^0(\pi_0(X))\to H^0(X)$ is an isomorphism.
\end{defn}

(Lemma \ref{l3.1} b) says that this condition is automatic if $X$ is geometrically connected.)\\

Classical Weil cohomologies are normalised.

\begin{defn}[\cf \protect{\cite{os2}}]\label{d6.1} A Weil cohomology $(\sC,H)$ with $\sC\in \Add^\rig$ is \emph{pseudo-tannakian} if there exists a faithful $\otimes$-functor from $\sC$ to a category of $\Z/2$-graded finite-dimensional vector spaces over a field.
\end{defn}

\begin{defn}\label{d3.3} A Weil cohomology $(\sC,H)$ has \emph{weights} if, for any $X,Y\in \sV$, we have
\[\sC(H^i(X),H^j(Y))=0 \text{ for } i\ne j.\]
\end{defn}

(This is the condition of Lemma \ref{l3.4} a).)

\begin{exs} \label{exw}  a) Let $k$ be a subfield of $\C$. Consider the additive $\otimes$-category $\sH$ of pure, polarisable $\Q$-Hodge structures provided with $L_\sH:= \Q(-1)$. The Hodge enrichment of Betti cohomology $(\sH,H)$ is a Weil cohomology and it has weights.\\
b) If $k$ is finitely generated, $\car k =p>0$ and $\ell\neq p$ is a prime number, let $\sR_\ell$ be the category of $\Q_\ell$-adic representations of $Gal(k^s/k)$, where $k^s$ is a separable closure. We get a Weil cohomology $(\sR_\ell, H_\ell)$ given by $\ell$-adic cohomology: this Weil cohomology has weights thanks to Deligne \cite{weilI}.
\end{exs}

\subsection{Weil cohomologies as functors on Chow motives}

\begin{prop} \label{p1.1} Suppose $\sC$ pseudo-abelian.\\
a) Any Weil cohomology $(H, \kappa, \Tr, \cl)$ with values in $(\sC,L_\sC)$ lifts to a strong additive $\otimes$-functor $$\uH^*:\sM_\rat(k,\V)\to \sC^{(\Z)}$$ 
together with a morphism $\Tr: \uH^2(\L)\to L_\sC$  such that 
\begin{enumerate}
\item $\uH^*(\L)$ is concentrated in degree $2$ and $\Tr$ is an isomorphism;
\item $\uH^*(\sM_\rat^\eff(k,\V)) \subset \sC^\N$;
\item if $X$ is geometrically connected, then $\un=\uH^0(h(\Spec k))\allowbreak\to \uH^0(h(X))$ is an isomorphism.
\end{enumerate}
b) Conversely, if $\uH^*:\sM_\rat(k,\V)\to \sC^{(\Z)}$ is a strong additive $\otimes$-functor plus a morphism $\Tr$, which verifiy Conditions (1) -- (3) of a), then $$H^*=\uH^*\circ h:\V^{op}\to \sC^{(\Z)}$$ yields a Weil cohomology.
\end{prop}

\begin{proof} 
a) Using Lemma \ref{l1.1}, the cycle class $\cl^m_{X\times Y}$ yields a homomorphism
\[CH^m(X\times Y)_\Q= \Corr(k,\V)(X,Y)\to \sC^{(\Z)}(H^*(X),H^*(Y)).\]
 One checks as usual that it respects composition of correspondences. Since $\sC$ is  pseudo-abelian, via $\cl$ the functor $H^*:\V^{op}\to \sC^{(\Z)}$ extends to a $\Q$-linear functor 
\[\sM_\rat^\eff(k,\V)\to \sC^{(\Z)}\]
which is symmetric monoidal by Axiom (iv) of Definition \ref{d1.1}.  Conditions (1), (2) and (3) of Proposition \ref{p1.1}  follow from Axioms  (iii) and  (ii) of Definition \ref{d1.1} plus Lemma \ref{l3.1} b). Since, still by Axiom (iii) of Definition \ref{d1.1}, $H^2(\L)$ is invertible, the above functor extends to $\sM_\rat(k,\V)$ as a $\otimes$-functor $\uH^*$ (on objects $M = (X, p, n)\in \sM_\rat(k,\V)$ we have $\uH^i(M)= p_*H^{i+n}(X)$). Moreover, we get the isomorphism $\Tr$ from (c) and (iii).\\

b)  We need to provide the data and check the axioms of Definition \ref{d1.1}. The K\"unneth structure (b) and Axiom (iv) are obtained from the strong monoidality of $\uH^*$ and the equalities $h(X)\otimes h(Y)= h(X\times Y)$. Axiom (i) follows from the unitality of $\uH^*$. The lower bound in Axiom (ii) follows from (2). 

For the sequel, we recall the isomorphism
\begin{equation}\label{eqdual1}
h(X)^\vee\simeq h(X)\otimes \T^{n}
\end{equation}
for any $X\in \V$ of dimension $n$; the unit and counit of this duality are induced by the morphisms
\begin{equation}\label{eqdual2}
\L^n\to h(X)\otimes h(X), \quad h(X)\otimes h(X)\to \L^n 
\end{equation}
both given by the class of the diagonal in $CH^n(X\times X)$. 

 Since $\uH^*$ is symmetric monoidal, we have an isomorphism
\begin{equation}\label{eqdual}
(\uH^*(h(X))^\vee\simeq \uH^*(h(X)^\vee)\end{equation}
for any $X$ of dimension $n$, where $^\vee$ denotes duals in both categories $\sM_\rat(k,\V)$ and $\sC^{(\Z)}$. This isomorphism is obtained by applying $\uH^*$ to \eqref{eqdual2}. 

For Axiom (v), note that 
\[
\uH^*(h(X)\otimes \T^n)\simeq \uH^*(h(X))\otimes \uH^*(\T)^{\otimes n},
\]
hence
\[
\uH^{2n}(h(X))(n)\simeq \uH^*(h(X)\otimes \T^{n})^0
\]
by (1); $\Tr_X$ is then defined by the morphism $h(X)\otimes \T^{n}\to \un$ dual by \eqref{eqdual1} of the ``structural'' morphism $\un\to h(X)$.  The identity $\Tr_X\otimes \Tr_Y=\Tr_{X\times Y}$ follows. Axiom (iii) also follows from (1). Using (3), we get the isomorphism of (v) when $X$ is geometrically connected.

  More generally, we get from \eqref{eqdual1} and \eqref{eqdual}
\[H^{i}(X)^\vee =\uH^{-i}(h(X)^\vee) \simeq \uH^{-i}(h(X)\otimes \T^{n}) = H^{2n-i}(X)(n)\]
which yields the last part of (v), and in particular the upper bound in (ii) by (2).

In (vi), $\cl^i_X$ is induced by the functoriality of $\uH^*$, which sends $CH^i(X)\allowbreak\otimes \Q=\sM_\rat(k,\V)(\un,h(X)\otimes \T^i)$ to $\sC(\un,H^{2i}(X)(i))$; the commutativity of the square also follows from this functoriality and the definition of $\Tr_X$.
\end{proof}

\subsection{An intermediate version: Chow correspondences}

\begin{prop}\label{p1.3} Let us still assume $\sC$ pseudo-abelian. Then a Weil cohomology $(H, \kappa, \Tr, \cl)$  with values in $(\sC,L_\sC)$ is equivalent to a strong additive $\otimes$-functor $H^*:\Corr(k,\V)\to \sC^\N$
together with a map $\Tr : H^2(\P^1)\to L_\sC$, satisfying the following conditions:
\begin{enumerate}
\item $H^1(\P^1)=0$ and $\Tr$ is an isomorphism;
\item if $X$ is geometrically connected, then $\un\allowbreak\to H^0(X)$ is an isomorphism.
\end{enumerate}
\end{prop}

\begin{proof} Let $\uH^*$ verify the conditions of Proposition \ref{p1.1} a). Then its composition with $\Corr(k,\V)\to \sM_\rat(k,\V)$ obviously verifies the conditions of Proposition \ref{p1.3}.

Conversely, let $H$ be as in Proposition \ref{p1.3}. Since $\sC$ is pseudo-abelian, $H$ extends canonically to $\sM_\rat^\eff(k,\V)$, and then to a strong $\otimes$-functor $\sM_\rat(k,\V)\to \sC^{\lfloor\Z\rfloor}$ by Condition (1), where $\sC^{\lfloor\Z\rfloor}$ denotes the $\otimes$-category of $\Z$-indexed objects of $\sC$ with bounded below support. The resulting functor clearly verifies the conditions of Proposition \ref{p1.1} a), except perhaps for the fact that $\uH^*$ takes values in $\sC^{(\Z)}$. This is granted by Lemma \ref{l3.3} a). 
\end{proof}

In the sequel, we shall use the following categories several times.

\begin{defn}\label{d4.2}
Let $\Corr(k,\V)[\L]$ (\resp $\Corr(k,\V)[\L,\L^{-1}]$) denote the strictly full subcategory of $\sM_\rat(k,\V)$ whose objects are the direct sums of $h(X)\otimes \L^n$ for $X\in \V$ and $n\in\N$ (\resp $n\in\Z$). 
\end{defn}

We observe:

\begin{lemma}\label{l4.8} $\Corr(k,\V)[\L]$ and $\Corr(k,\V)[\L,\L^{-1}]$ are strict $\otimes$-sub\-cat\-eg\-ories of $\sM_\rat(k,\V)$, and $\Corr(k,\V)[\L,\L^{-1}]$ is rigid.\qed
\end{lemma}

In contrast to the  previous conditions, those of Proposition \ref{p1.3} continue to make sense when $\sC$ is not  pseudo-abelian (this hypothesis was inserted to connect with Proposition \ref{p1.1}). Moreover this case reduces to the rigid one as follows: 

\begin{prop}\label{p1.4} Proposition \ref{p1.3} remains valid for any additive $\otimes$-category  $\sC$. Moreover,\\
a) For any $X\in \V$ and any $i\in \N$, $H^i(X)$ is dualisable in $\sC$.\\
b) The smallest strictly full additive $\otimes$-subcategory  $\sD$ of $\sC$ containing the $H^i(X)$ is rigid and contains $L_\sC$.
\end{prop}

\begin{proof} For the first claim, start from $(\sC,L_\sC)$ and $H$ verifying the hypotheses of Proposition \ref{p1.3}. They remain true when replacing $\sC$ by its pseudo-abelian hull $\sC^\natural$. Applying Proposition \ref{p1.3}, this yields a Weil cohomology as in Definition \ref{d1.1} with values in $\sC^\natural$, which in fact takes values in $\sC$. Same process in the other direction.
In a), $\uH^*(X)$ is dualisable by Lemma \ref{l4.8}, hence the claim follows from Lemma \ref{l3.3} b). Finally b) follows from a).
\end{proof}

\subsection{Adequate equivalences}

\begin{prop} \label{prop:alg} Let $(\sC, H)$ be a normalised Weil cohomology (Definition \ref{d3.4}) and assume that $\V$ contains all curves. 
 Then $H^*: \Corr\to \sC^{(\N)}$ factors through algebraic equivalence. Hence so does the induced functor $\uH^*:\sM_\rat\to (\sC^\natural)^{(\Z)}$. 
\end{prop}

\begin{proof} Let $X\in \V$ and $\alpha \in CH^i(X)_\Q=\sM_\rat(k)(\L^i,h(X))$ be algebraically equivalent to $0$: we must show that $H^*(\alpha)=0$. By the Weil-Bloch trick (\cf \cite[proof of Lemma 7.10]{bo}), $\alpha$ is the image of some $\beta\in \Pic^0(C)_\Q$ for some curve $C$ under an algebraic correspondence from $C$ to $X$: this reduces us to $i=1$, $X=C$. Choose a  non constant morphism $f:C\to \P^1$, whence a morphism $g:C\to \P^1\times \pi_0(C)$; we have a commutative  diagram:
\[\begin{CD}
\Pic(C)_\Q@>g_*>> \Pic(\P^1\times \pi_0(C))_\Q\\
@V \cl_W^1 VV @V \cl_W^1 VV \\
\Hom(\un,H^2(C)(1))@>g_*>> \Hom(\un,H^2(\P^1\times \pi_0(C))(1))
\end{CD}\]
where $g_*$ is the action of the transpose of the graph of $g$ viewed as a correspondence as in Propositions \ref{p1.1} and \ref{p1.3}  --- the commutation of the diagram follows from this construction, and the bottom horizontal arrow is an isomorphism by Axioms (v) and (vi) of a Weil cohomology plus the property of being normalised. The result thus follows from the vanishing of $\Pic^0(\P^1\times \pi_0(C))$.
\end{proof}

\begin{rk} \label{rk:tweil}
More is true if $H$ is abelian-valued (Definition \ref{d.ab}): by Lemma \ref{l2.1} c), $\uH^*$ even factors through Voevodsky's smash-nil\-pot\-en\-ce equivalence (which is coarser than algebraic equivalence by \cite{voe}). Here we don't need $\sV$ to contain all curves. \end{rk}

\section{The main theorem}

\subsection{The 2-functor of Weil cohomologies} We still fix an admissible category $\V$ of smooth projective varieties, as in Definition \ref{d3.1}.

\begin{defn}\label{d5.1} Let $\Add^\otimes_*$ be the $2$-category whose 
\begin{itemize}
\item objects are pairs $(\sC,L_\sC)$ where $\sC\in \Add^\otimes$ and $L_\sC$ is an invertible object of $\sC$;
\item $1$-morphisms $(\sC,L_\sC)\to (\sD,L_\sD)$ are pairs $(F,u)$ where $F\in \Add^\otimes(\sC,\sD)$ and $u$ is an isomorphism $F(L_\sC)\iso L_\sD$ (composition: $(G,v)\circ (F,u)=(G\circ F,v\circ G(u))$).
\item $2$-morphisms $\theta:(F,u)\Rightarrow (F',u')$ are $2$-morphisms $\theta:F\Rightarrow F'$ in $\Add^\otimes$ such that $u=u'\circ \theta_{L_\sC}$.
\end{itemize}
We define $\Add^\rig_*$ and $\Ex^\rig_*$ similarly, replacing $\Add^\otimes$ by $\Add^\rig$ or $\Ex^\rig$.
\end{defn}

This is the same as the Tate $\Q$-pretensor categories of \cite[p. 4]{os}.

\begin{defn}\label{d5.2} Let $(\sC,L_\sC)\in \Add^\otimes_*$.  We denote by  $\Weil(k,\V; \sC,L_\sC)$ the category whose objects $(H^*,  \Tr)$  are as in Proposition \ref{p1.3} (see also Proposition \ref{p1.4}), \ie $H^*: \Corr (k, \V)\to \sC^{(\N)}$. 
A morphism $\phi:(H^*,\Tr)\to ({H'}^*,\Tr')$ in
 $\Weil(k,\V; \sC,L_\sC)$ is a graded natural transformation $\phi:H^*\Rightarrow {H'}^*$ such that $\Tr=\Tr'\circ \phi_{\P^1}$. 
\end{defn} 

\begin{lemma}\label{l4.2} The category $\Weil(k,\V; \sC,L_\sC)$ is a groupoid.
\end{lemma}

\begin{proof} This amounts to saying that any morphism $\phi:(H^*,\Tr)\to ({H'}^*,\Tr')$ as above is invertible. By Propositions \ref{p1.3} and \ref{p1.4}, $H^*$ and ${H'}^*$ correspond to strong additive $\otimes$-functors $\Corr(k,\V)[\L,\L^{-1}]\to \sC^{(\Z)}$, where $\Corr(k,\V)[\L,\L^{-1}]$ is as in Definition \ref{d4.2}. Since it is rigid (Lemma \ref{l4.8}), the statement follows from \cite[Prop. I.5.2.3]{saa}.\footnote{Since the image of a dualisable object by a $\otimes$-functor is dualisable, the hypothesis in \loccit that the target category be rigid is not needed.}
\end{proof}

\begin{cons}\label{cons1}  Definition \ref{d5.2} provides a strict $2$-functor
\begin{equation}\label{eq5.10}
\Weil(k,\V; -): \Add^\otimes_*\to \Cat.
\end{equation}
\end{cons}

\begin{proof}[Detailed definition] 1) The $2$-functor is given on objects by Definition \ref{d5.2}.

2) Let $(F,u):(\sC,L_\sC)\to (\sD,L_\sD)\in \Add^\otimes_*$ be a $1$-morphism. We define  a ``push-forward''  functor 
\begin{equation}\label{eq.fu}
(F,u)_*:\Weil(k,\V; \sC,L_\sC)\allowbreak\to \Weil(k,\V; \sD,L_\sD)
\end{equation}
as follows:
\begin{itemize}
\item\emph{On objects}: $(F,u)_*(H^*,\Tr) =(F^{(\N)}\circ H^*,  u\circ F(\Tr))$.
\item\emph{On morphisms}: for $\phi:(H^*,\Tr)\to ({H'}^*,\Tr')$, $(F,u)_* \phi=F*\phi$.
\end{itemize}

By definition, we have $(G,v)_*\circ (F,u)_* = ((G,v)\circ (F,u))_*$. (This is the meaning of ``strict'').

In the sequel, we shall use this construction repeatedly; we abbreviate it to $H'=F_*H$ and simply say that $H'$ is the push-forward of $H$ by $F$.

3) Let $\theta:(F,u)\Rightarrow (F',u')$ be a $2$-morphism. For any $(H^*,\Tr)\in\Weil(k,\V; \sC,L_\sC)$, the natural transformation $\theta*H^*$ defines a morphism $\theta_*:(F,u)_*(H^*,\Tr)\to (F',u')_*(H^*,\Tr)$ (immediate verification).

Checking the axioms of a $2$-functor is trivial.\end{proof}

\begin{exs}\label{ex5.1}
a) \emph{Extension of scalars}. Let $(\sC,L_\sC)\in \Add^\otimes_*$, $R=Z(\sC)$ and $f:R\to S$ be a homomorphism to a commutative ring $S$. Then $(S\otimes_R \sC,L_\sC)$ is an object of $\Add^\otimes_*$ and extension of scalars $E$ defines a $1$-morphism $(E,1_{L_\sC}):(\sC,L_\sC)\to (S\otimes_R \sC,L_\sC)$.\\
b) \emph{Comparison isomorphisms}.
Let $k$ be a subfield of $\C$;  here, $\V=\Sm^\proj(k)$. Consider the category $\sH$ of pure, polarisable $\Q$-Hodge structures provided with $L_\sH:= \Q(-1)$. Then $(\sH,L_\sH)\in \Add^\otimes_*$ and we have a Weil cohomology $(H_B^*,\Tr_B)\in \Weil(k;\sH,L_\sH)$ given by Betti cohomology. We also have a forgetful functor $\iota_\sH:\sH\to \Vec_\Q$ forgetting the Hodge structure. On the other hand, let $\ell$ be a prime number, and let $\sR_\ell$ be the category of $\Q_\ell$-adic representations of $Gal(\bar k/k)$, where $\bar k$ is the algebraic closure of $k$ into $\C$: if $L_{\sR_\ell}:=\Q_\ell(-1)$, $\ell$-adic cohomology gives a Weil cohomology $(H_\ell^*,\Tr_\ell)\in \Weil(k;\sR_\ell,L_{\sR_\ell})$. We also have a forgetful functor $\iota_\ell:\sR_\ell\to \Vec_{\Q_\ell}$. Then Artin's comparison theorem gives an isomorphism
\[(\iota_\ell)_*(H_\ell^*,\Tr_\ell)\simeq \Q_\ell\otimes_\Q (\iota_\sH)_*(H_B^*,\Tr_B) \]
 in $\Weil(k;\Vec_{\Q_\ell},\iota_\ell(L_{\sR_\ell}))$, thanks to the isomorphism $\iota_\ell(\Q_\ell(-1))\allowbreak\simeq \Q_\ell\otimes_\Q \iota_\sH(\Q(-1))$ given by the exponential.

We leave it to the reader to refomulate the de Rham-Betti isomorphism in the same fashion, using the compatibility of $d$ and $d\log$ via the exponential.
\end{exs}

\begin{rk}\label{r4.1} We may reformulate Construction \ref{cons1} by defining a $2$-category $\Weil(k,\V)$ whose objects are those of $\Weil(k,\V; \sC,L_\sC)$ for varying $(\sC,L_\sC)$, provided with a $2$-functor $\Weil(k,\V)\to \Add^\otimes_*$ which is ``$2$-cofibred'' in a sense generalising \cite[\S 6]{SGA1}; the push-forward functors of Construction \ref{cons1} play the rôle of cocartesian morphisms. Details are left to the reader.
\end{rk}
\bigskip

\subsection{The representability theorem}
\begin{thm}\label{thm2}
The $2$-functor  $\Weil(k,\V, -)$ is strictly $2$-represent\-able.  A representing object is called a \emph{universal Weil cohomology} (relative to $\V$) and is denoted by 
\[W_\V^*:\Corr(k,\V)\to \weil(k,\V)^{(\N)},\quad \Tr_W:W_\V^2(\P^1)\iso \L_W\] 
(where $(\weil(k,\V),\L_W)\in \Add^\otimes_*$).
\end{thm}

Let us specify the meaning of strictly $2$-representable here: $W_\sV^*$ induces an isomorphism of categories
\begin{equation} \label{eq5.9}
\Add_*^\otimes((\weil (k,\V),\L_W),
(\sC,L_\sC) )\iso\Weil(k,\V; \sC,L_\sC).
\end{equation}

Explicitly, for $(\sC,L_\sC)\in \Add_*^\otimes$ and  $(H^*, \Tr)\in \Weil(k,\sV; \sC,L_\sC)$, there exists a unique 1-morphism $(F_H, u_H) :
(\weil(k,\V), \L_W)\to (\sC,L_\sC)$  of $\Add_*^\otimes$ such that the induced diagram
\begin{equation}\label{eq4.2}
\begin{gathered}
\xymatrix{\Corr(k,\V)\ar[r]^{W_\V^*} \ar[d]_{H^*}& \weil(k,\V)^{(\N)}
\ar[dl]^{F_H^{(\N)}} \\ \sC^{(\N)}}
\end{gathered}
\end{equation}
strictly commutes. Moreover, if $F,F'\in \Add_*^\otimes((\weil (k,\V),\L_W)$ and $H=F_*W_\sV$, $H'=F'_*W_\sV$ as in \eqref{eq.fu}, 
any morphism $\phi:H\Rightarrow H'$ as in Definition \ref{d5.2} extends uniquely to a morphism $F\Rightarrow F'$.

\begin{cor} \label{cor:thm2}
Theorem \ref{thm2} remains true if we compose $\Weil(k,-)$ with the inclusions $\Add^\rig_*\subset \Add^\otimes_*$ and $\Ex^\rig_*\subset \Add^\rig_*$.  In particular, $(\weil(k,\V),\L_W)\in \Add^\rig_*$. A representing object for the second universal problem is called a \emph{universal abelian Weil cohomology} (relative to $\V$) and is denoted by 
\[W_{\V,\ab}^*:\Corr(k,\V)\to \weil_\ab(k,\V)^{(\N)},\quad \Tr_{W,\ab}:W_{\V,\ab}^2(\P^1)\iso \L_{W,\ab}\] 
(where $(\weil_\ab(k,\V),\L_{W,\ab})\in \Ex^\rig_*$). We have $\weil_\ab(k,\V)=T(\weil(k,\V))$.
\end{cor}

\begin{proof} Let $\sD\subseteq \weil(k,\V)$ be as in Proposition \ref{p1.4} b) (where we take $\sC=\weil(k,\V)$): note that $L_\sD:=\L_W\in \sD$. Since $W_\sV^*$ takes its values in $\sD^\N$, by the $2$-universal property of $\weil(k,\V)$ there exists a functor $(F,u):(\weil(k,\V),\L_W) \to (\sD,L_\sD)$ in $\Add^\otimes_*$ such that $F^{(\N)}\circ W_\V^*\simeq H^*$, whose composition with the inclusion  $(\sD,L_\sD)\subseteq (\weil(k,\V),\L_W)$ is isomorphic to $\Id_{(\weil(k,\V),\L_W)}$. Since $\sD\subseteq \weil(k,\V)$ is strictly full, this implies that $\sD=\weil(k,\V)$. So $\weil(k,\V)\in \Add^\rig$. 

Let $W_{\sV,\ab}$ be the push-forward of $W_\sV$, in the sense of \eqref{eq.fu}, by the functor $\lambda_{\weil(k,\V)}$ of Lemma \ref{bvk}:  it is clear that $(T(\weil(k,\V)),W_{\sV,\ab})$ has the desired universal property.
\end{proof}

\begin{proof}[Proof of Theorem \ref{thm2}]  To construct a universal Weil cohomology,  we use version (2) in Remarks \ref{r1} and \ref{r2}: namely, we start from the $\otimes$-category $\Corr(k,\V)\times \N$ that we modify step by step until we obtain a $\otimes$-category $\weil(k,\V)$ and a $\otimes$-functor $\Corr(k,\V)\times \N\to \weil(k,\V)$ with the correct properties.

For simplicity, we abbreviate $\Corr(k,\V)$ to $\Corr$ and $\weil(k,\V)$ to $\weil$ in this proof. By Theorem \ref{thm1m} applied to $\Corr\times \N$, we first obtain  
$$ (W^\flat,\boxtimes, \upsilon) : \Corr\times \N\to (\Corr\times \N)^\kappa$$ where $(\Corr\times \N)^\kappa$ is the universal $\otimes$-category together with 
\[\boxtimes_{X,Y}^{i,j}: W^\flat(X,i)\otimes W^\flat(Y,j)\to W^\flat(X\times Y,i+j)\]
the induced $\N$-graded unital external product for $X,Y\in \Corr$, $i,j\in \N$, 
and 
\[\upsilon: \omega\to W^\flat(\un,0)\]
 where $\omega$ is the unit of $(\Corr\times \N)^\kappa$. Let $(\Corr\times \N)^{\kappa,\add}\allowbreak:={\rm Add}^\otimes(W^\flat)$, see Proposition \ref{p4.3t}; thus the functor $(\Corr\times \N)^\kappa\to (\Corr\times \N)^{\kappa,\add}$ is  a strong $\otimes$-functor and its composition $W^\add$ with $W^\flat$ is additive.  Whence morphisms
\[\delta= \sum\boxtimes_{X,Y}^{i,j}: \bigoplus_{i+j=k} W^\add(X,i)\otimes W^\add(Y,j)\to W^\add(X\times Y,k)\]
 in $(\Corr\times \N)^{\kappa,\add}$, for $i,j,k\in \N$.

Now we need to impose the Künneth formula (Definition \ref{d3}) and the conditions of Proposition \ref{p1.3}. Consider the set $S$ of morphisms of $(\Corr\times \N)^{\kappa,\add}$ given by $\delta$, $\upsilon$ and $W^\add(\un,i)\to 0$ for $i> 0$, $W^\add(\P^1,i)\allowbreak\to 0$ for $i\neq 0, 2$ and $W^\add(\un ,0)\allowbreak\to W^\add(h(X),0)$ for $X$  geometrically connected. 
Enlarging $S$ to $S_{\oplus\otimes}$ as in Proposition \ref{s2.4}, we get an additive $\otimes$-category 
$$\weil^\eff:= (\Corr\times \N)^{\kappa,\add}[S^{-1}_{\oplus\otimes}]$$ and a functor $W^\eff:\Corr\times\N\to \weil^\eff$, given by the composition of $W^\add$ with the localisation functor.
We finally $\otimes$-invert $W^\eff(\P^1,2)$, hence a category $\weil$. We want to show that it is still a $\otimes$-category. For this, let $\Corr[\L]$ be as in Definition \ref{d4.2}. Then $W^\eff$ canonically factors through $\Corr[\L]\times\N$, sending $(\L,2)$ to $W^\eff(\P^1,2)$.  Since the permutation involution $c_{\L,\L}$ on $\L\otimes\L$ is the identity in $\Corr[\L]$ and $W^{\eff}(\L\otimes\L, 4)=W^\eff(\P^1,2)\otimes W^\eff(\P^1,2)$  hence we have that $c_{W^\eff(\P^1,2),W^\eff(\P^1,2)}$ is the identity in $\weil^\eff$.
Therefore Voevodsky's condition \cite[Th. 4.3]{voeicm} is verified and the claim holds. 

Write $W$ for the composition of $W^\eff$ with the functor $\weil^\eff\to \weil$. Define $\L_W$ as $W(\P^1,2)$.  We thus have finally constructed a strong additive $\otimes$-functor  
$$W^*: \Corr\to \weil^\N$$
with a tautological isomorphism $\Tr:W(\P^1,2)\iso \L_W$. We then apply Propositions \ref{p1.3} and \ref{p1.4}.

For $(\sC,L_\sC)\in \Add_*^\otimes$ and  $(H^*, \Tr)\in \Weil(k; \sC,L_\sC)$, consider  the induced unital $\N$-graded external product  
\[(H, \kappa, \eta): \Corr\times \N\to \sC\] 
given by Remark \ref{r1}. It extends successively to $(\Corr\times \N)^*$ by Theorem \ref{thm1m}, $(\Corr\times \N)^{\kappa,\add}$ by Proposition \ref{p4.3t}, $\weil^\eff$ by Proposition \ref{s2.4} and $\weil$ since $H^2(\P^1)\cong L_{\sC}$ is invertible in $\sC$. We then obtain a 1-morphism $(F_H, u_H) :
(\weil, \L_W)\to (\sC,L_\sC)$  of $\Add_*^\otimes$ where
$F_H:\weil\to  \sC$ is the induced additive strong
$\otimes$-functor such that $F_H(W^i(X)) = H^i (X)$, $u_H=\Tr :
F_H(\L_W)= H^2(\P^1)\iso L_\sC$ and the push-forward $(F_H,
u_H)_*(W, \Tr)= (H,\Tr)$.
Therefore Diagram \eqref{eq4.2} strictly commutes; moreover, $F_H$ is obviously unique. 

The full faithfulness of \eqref{eq5.9} is now proven step by step: indeed, each step of the above construction is the solution of a $2$-universal problem.
\end{proof}

\begin{rk} Of course, the category $\weil(k,\V)$ depends on the choice of $\V$. Given $\V\subseteq \V'$, restricting $W^*_{\V'}$ to $\V$ yields by the universal property a canonical $\otimes$-functor $\weil(k,\V)\to \weil(k,\V')$. Here is a computation in the minimal case $\V=\{\Spec k\}_\adm $: the category $\sC=\sM_\rat(k,\V)$ consists of (pure) Tate motives; since $\sC(\L^m,\L^n)=0$ for $m\ne n$, the functor $h:\V\to \sC$ induces a canonical Weil cohomology $h^*:\Corr(k,\V)\to \sC^{(\N)}$ which is obviously initial. Therefore, $(\weil(k,\V),\L_W)=(\sM_\rat(k,\V),\L)$.
\end{rk}

\begin{rk}\label{r5.1} A similar technique would prove the existence of a universal normalised Weil cohomology (see Definition \ref{d3.4}). However, there are further natural properties enjoyed by usual Weil cohomologies, that we shall study in Section \ref{s8}; this is where we shall prove refined representability theorems.
\end{rk}

\subsection{Extension to other adequate equivalences}\label{s5.3a}

Let $\sim$ be an adequate equivalence relation on algebraic cycles on $\sV$: it corresponds to a $\otimes$-ideal of $\sM(k,\sV)$ by \cite[Lemma 4.4.1.1]{andremotifs}. We write $\Corr_\sim(k,\sV)$ and $\sM_\sim(k,\sV)$ for the corresponding categories of correspondences and motives. 

\begin{defn}\label{d5.3} A Weil cohomology $(H^*,\Tr)$ is \emph{compatible with $\sim$} if $H^*$ factors through $\sim$ in Proposition \ref{p1.3}, \ie induces a functor $\Corr_\sim(k,\sV)\to \sC^{{(\Z)}}$.
\end{defn}

For $(\sC,L_\sC)\in \Add_*^\otimes$, let $\Weil_\sim(k,\sV;\sC,L_\sC)$ be the full subcategory of $\Weil(k,\sV;\sC,L_\sC)$ consisting of the  $(H^*,\Tr)$ compatible with $\sim$. This defines a strict $2$-functor $\Weil_\sim(k,\sV;-)$ as in Construction \ref{cons1}.

\begin{thm}\label{t5.3} The analogues of Theorem \ref{thm2} and Corollary \ref{cor:thm2} hold for $\Weil_\sim(k,\sV;-)$, yielding universal Weil cohomologies $W_{\sV,\sim}$ and $W_{\sV,\sim}^\ab$ with values in $\weil_\sim(k,\sV)$ and $\weil_{\sim,\ab}(k,\sV)=T(\weil_\sim(k,\sV))$.
\end{thm}

\begin{proof} One checks that the proofs of Theorem \ref{thm2} and Corollary \ref{cor:thm2} go through without change.
\end{proof}

Let $\sim\ge \sim'$ be two comparable adequate equivalence relations. The universal properties of $W_\sim$ and $W_{\sim,\ab}$ yield canonical strong $\otimes$-functors
\begin{equation}\label{eq5.8}
\weil_\sim(k,\sV)\to \weil_{\sim'}(k,\sV), \quad \weil_{\sim,\ab}(k,\sV)\to \weil_{\sim',\ab}(k,\sV)
\end{equation}
the second being exact. For simplicity, let us drop $(k,\sV)$ from the notation in the following proposition.

\begin{prop}\label{p5.7} In \eqref{eq5.8}, the first functor is full and essentially surjective, and the second is a localisation. The kernel (\resp Serre kernel) of the first (\resp second) one is generated by $W_\sim^*(\Ker(\Corr_{\sim}\to \Corr_{\sim'})$ (\resp by the images of the morphisms in $W_{\sim,\ab}^*(\Ker(\Corr_{\sim}\to \Corr_{\sim'})$).
\end{prop}

\begin{proof} Let $\weil_{\sim'}(k,\sV)'$ be the quotient of $\weil_\sim(k,\sV)$ by the said $\otimes$-ideal. The functor \eqref{eq5.8} clearly factors through $\weil_{\sim'}(k,\sV)'$. But pushing forward $W_\sim$ to $\weil_{\sim'}(k,\sV)'$ by the projection functor in the style of \eqref{eq.fu} provides the latter category with a Weil cohomology which factors through $\Corr_{\sim'}$; hence $\weil_{\sim'}(k,\sV)'\to \weil_{\sim'}(k,\sV)$ is an equivalence of categories by universality. Same proof for $W_{\sim',\ab}$, \emph{mutatis mutandis}.
\end{proof}

\begin{rk} As Proposition \ref{prop:alg} and Remark \ref{rk:tweil} show, it may well happen that \eqref{eq5.8} is an equivalence even if $\sim \ne \sim'$, either in the abelian case or in both cases. Note also that $\weil_\num\ne 0$ by \cite{AK} if $k$ is a finite field.
\end{rk}

\section{Initiality}

In the sequel, most $\otimes$-functors we shall encounter will be strong; therefore we drop the adjective `strong' to lighten the exposition, and add `lax' if our $\otimes$-functor turns out not to be strong. 

We drop the mention of $(k,\sV)$ for lightness of notation, except when it may create an ambiguity, e.g. when a statement depends on this pair.

\subsection{Initial and final Weil cohomologies}
Any Weil cohomology $H^*:\Corr\allowbreak\to \sC^{(\N)}$ defines an adequate equivalence relation on $\V$. Specifically, extend $H^*$ to $\uH^*:\sM_\rat\to (\sC^\natural)^{(\Z)}$ by Proposition \ref{p1.3}. Then the kernel of $\uH^*$ is a $\otimes$-ideal. We set (\cf \cite[3.3.4 \& 4.4.1]{A}): 

\begin{defn} \label{defhom}
Denote $\sim_H$ the adequate equivalence relation corresponding to $\Ker \uH^*$ and let 
$$\sM_H:=\sM_{\sim_H}(=(\sM_\rat/\Ker \uH^*)^\natural)$$ 
be the category of \emph{$H$-homological motives}. 
  Composing the induced functor  $\uH^*: \sM_H\to(\sC^\natural)^{(\Z)}$  with the direct sum functor, we get a faithful ``realisation functor''
\[\uH:\sM_H\to \sC^\natural \ \ \ M \leadsto \bigoplus_i \uH^i(M)\]
which is monoidal but not symmetric (see Remark \ref{r2}). (Note that $\uH^i(M)=H^i(X)$ if $M=h(X)$ for $X\in \sV$.) In the universal case, we abbreviate
\[\sM_\hun:= \sM_{W_\ab}\]
and get a faithful functor
\begin{equation}\label{eq6.3}
 w :\sM_\hun\to \weil_\ab= T(\weil).
 \end{equation}
\end{defn}

\begin{lemma}\label{l4.3} 
Let $(F,u):(\sC,L_\sC)\to (\sC',L_\sD')\in \Add^\otimes_*$ be a $1$-morphism. Let 
$H'= F_*H$ be the push-forward of $H$ by $F$ as \eqref{eq.fu} in Construction \ref{cons1}. We have
an induced $\otimes$-functor  $$\sM_H\to \sM_{H'}$$ which is the identity if $F$ is faithful.
\end{lemma}
\begin{proof}
In fact, ${H'}^*= F^{(\N)} H^*$, thus $\Ker \underline{H}^*\subseteq \Ker {\underline{H}'}^*$ and the equality holds if $F$ is faithful.
\end{proof}

\begin{defn} \label{relh} In the situation of Lemma \ref{l4.3}, we say that 
\begin{itemize} \item $H$ is an \emph{enrichment} of ${H'}$  if $F$ is faithful. We say that ${H'}$ is \emph{initial} if any such $F$ is an equivalence of categories. 
\item  ${H'}$ is a \emph{specialisation} of $H$ if $F$ is full. We say that $H$ is \emph{final} if any specialisation of $H$ is an isomorphism of categories.
\end{itemize}
\end{defn}

We have an analogous definition for the abelian case:

\begin{defn} \label{relhab} In the situation of Lemma \ref{l4.3}. Suppose that $(F,u):(\sC,L_\sC)\to (\sC',L_\sD')\in \Ex^\rig_*$. We say that 
\begin{itemize} \item $H$ is an \emph{abelian enrichment} of ${H'}$  if $F$ is faithful. We say that ${H'}$ is \emph{ab-initial} if any such $F$ is an equivalence of categories. 
\item  ${H'}$ is a \emph{ab-specialisation} of $H$ if $F$ is localisation. We say that $H$ is \emph{ab-final} if any ab-specialisation of $H$ is an isomorphism of categories.
\end{itemize}
\end{defn}

The following is obvious:

\begin{lemma} \label{l6.1}
 If $H$ has weights (see Definition \ref{d3.3}), any enrichment and any special\-is\-ation of $H$ has weights.
 \qed
\end{lemma}

\begin{rks} a)  Lemma \ref{l6.1} is also true for ab-specialisations, but nontrivially: any $\otimes$-localisation of a rigid abelian $\otimes$-category is full \cite[Th. 4.21]{standard-schur}.\\
b) Note that any traditional Weil cohomology is final (and even ab-final). In Proposition \ref{p6.1}, we shall characterise those Weil cohomologies which are initial and final.
\end{rks}

Note that $\sim_H = \sim_{H'}$ and $\sM_H= \sM_{H'}$ for all enrichments  $H^*$ of ${H'}^*$. Also, if $(\sC,L_\sC)\in \Add^\otimes_*$ happens to be in $\Ex^{\rm rig}_*$, we have two different universal problems: to distinguish them we shall refer to Definitions \ref{relh}-\ref{relhab}.

\begin{thm}\label{thm2bis}
a) Any Weil cohomology $(\sC,H)$ has an initial enrichment $(\weil_H,W_H)$ and an ab-initial enrichment $(\weil_H^\ab,W_H^\ab)$ if $\sC\in \Ex^\rig$. In this case, there is a canonical faithful $\otimes$-functor 
$\iota_H:\weil_H\to \weil_H^\ab$ such that $W_H^\ab=(\iota_H)_*W_H$.\\
b) There is a $1-1$ correspondence between initial Weil cohomologies and $\otimes$-ideals of $\weil$ (resp. Serre $\otimes$-ideals of $\weil_\ab$). 

Moreover, the target of any initial or ab-initial Weil cohomology is rigid. \\
c) The category $\weil_H$ is a $\otimes$-quotient of $\weil$, the category $\weil_H^\ab$ is a $\otimes$-Serre localisation of $\weil_\ab=T(\weil)$, and $\iota_H$ induces a $\otimes$-localisation $T(\weil_H)\surj \weil_H^\ab$. 
\end{thm}

\begin{proof}
a) Let $(\sC,H)$ be a Weil cohomology, and let 
\[(F_H ,u_H) : (\weil, \L)\allowbreak\to (\sC, L_\sC)\] 
be the classifying additive $\otimes$-functor. Set $$\weil_H:= \weil/\Ker F_H$$ and $\bar F_H  : \weil_H\to \sC$ for the induced faithful additive $\otimes$-functor. The push-forward of the universal Weil cohomology $W$ along the projection $\weil\to\weil_H$ yields $W_H$ such that 
$(\bar F_H)_*W_H = ( F_H)_*W  \allowbreak= H$. If $F_*H'= H$ then   
$F_*(F_{H'})_*W \allowbreak= H$ and $F\circ F_{H'} = F_H$ by unicity in the universal property. Since $F$ is faithful we have that $\Ker F_H = \Ker F_{H'}$ and $F_{H'}$ factors through $\weil_H$ providing $\bar F_{H'}:\weil_H \to \sC'$ such that $(\bar F_{H'})_*W_H= H'$ as claimed. Moreover, $\weil_H$ is rigid as a $\otimes$-quotient of $\weil$.

Similarly, for $(\sC, L_\sC)\in \Ex^{\rm rig}_*$ the $\otimes$-functor 
\[(F_H ,u_H) : (T(\weil), \L)\to (\sC, L_\sC)\] 
is exact and factors through the Serre localisation $\weil_H^\ab$ of \allowbreak $T(\weil)$ by the (Serre) kernel of $F_H$, and $\weil_H^\ab\in \Ex^\rig$ by \cite[Prop. 4.5]{BVK}, yielding a Weil cohomology with values in $\weil_H^\ab$. Moreover the induced exact functor $\weil_H^\ab\to \sC$ is injective on objects by construction, hence faithful.

b) Any (initial) Weil cohomology defines a $\otimes$-ideal of $\weil$. Conversely, any such $\otimes$-ideal $\sI$ defines an initial Weil cohomology by push-forward of $W$ along $\weil\to \weil/\sI$. Same reasoning in the abelian case, \emph{mutatis mutandis}. 

c) The first two statements follow from the proof of a), the third then follows from Lemmas \ref{t6.1} and \ref{l2.3}.
\end{proof}

\begin{ex} \label{exunivin} For $H=W$ we get  $\weil_W = \weil$ (as follows from the proof of Theorem \ref{thm2bis} a)). For $H=W_\ab$, we get $\weil_{W_\ab}^\ab = T(\weil)$ and $\weil_{W_\ab} = \weil/\ker\lambda_\sW$, where $\lambda_\weil:\weil \to T(\weil)$ is the canonical functor of Lemma \ref{d3.4};
the faithful functor $\iota_{W_\ab}$ of Theorem \ref{thm2bis} a) is identified with the induced functor. 
\end{ex}

\begin{prop}\label{p6.1} Let $F,H,{H'}$ be as in Lemma \ref{l4.3}. Then $F$ induces a full functor $\weil_H\to\weil_{H'}$. If $\sC,\sC'\in \Ex^\rig$ and $F$ is exact, then it induces a localisation $\weil_H^\ab\to\weil_{H'}^\ab$. These functors are the identity if $F$ is faithful.\\
There is a $1-1$ correspondence between maximal $\otimes$-ideals of $\weil$ and Weil cohomologies which are both  initial and final. Similarly, there is a $1-1$ correspondence between maximal Serre $\otimes$-ideals of $\weil^\ab$ and Weil cohomologies with target in $\Ex^\rig_*$ which are both  initial and ab-final.
\end{prop}

\begin{proof} The additive case is obvious from Theorem \ref{thm2bis} c). For the abelian case, by Lemmas \ref{l2.3} and \ref{l2.1} d) we have to show that the localisation functor $T(\weil)\to \weil_{H'}^\ab$ factors through a $\otimes$-localisation  $\weil_{H}^\ab\to \weil_{H'}^\ab$. This follows from the inclusion $\Ker F_H\subseteq \Ker F_{H'}$.
\end{proof}

\begin{rk} By \cite[Th. 4.18]{standard-schur}, maximal Serre $\otimes$-ideals of \break $T(\weil)$ are in $1-1$-correspondence with the maximal ideals of its centre (which is an absolutely flat ring by Lemma \ref{l2.1} d)). Similarly, maximal $\otimes$-ideals of $\weil$ contain the ideal $\sN$ of negligible morphisms (ibid., Cor. 5.14). 
\end{rk}
Moreover:
\begin{lemma}\label{p6.2} If $\sW_H^\natural$ is abelian, and
\begin{thlist}
\item either $\sW_H^\natural$ and $\sW_H^\ab$ are both connected (Definition \ref{d2.1})
\item  or $\sW_H^\natural$ is semisimple
\end{thlist}
then $\iota_H^\natural:\sW_H^\natural\iso\sW_H^\ab$ is an equivalence.
\end{lemma}

\begin{proof} Assuming (i), we get from Lemma \ref{l2.1} b) that the faithful functor $\iota_H^\natural$ is exact.  Assuming (ii), $\iota_H^\natural$ factors as a composition
\[\sW_H^\natural\iso T(\sW_H^\natural)\to \sW_H^\ab \]
in which the first functor is an equivalence by Lemma \ref{bvk} and the second is exact. The conclusion follows from Theorem \ref{thm2bis} a) in both cases. \end{proof}

\subsection{When two worlds meet} \label{s6.3}

\begin{prop}\label{p6.3} Let $\sim$ be an adequate equivalence. Then, with the notation of Theorem \ref{t5.3},  $(\weil_\sim,W_\sim)$ is initial and $(\weil_{\sim,\ab},W_{\sim,\ab})$ is ab-initial.
\end{prop}

\begin{proof} By Proposition \ref{p5.7}, $\weil_\sim$ is a quotient of $\weil$ and $\weil_{\sim,\ab}$ is a localisation of $\weil_\ab$. Therefore the claim follows from Theorem \ref{thm2bis} b) and c).
\end{proof}

Recall from Definition \ref{defhom} that any Weil cohomology $H$ defines an adequate equivalence $\sim_H$ (homological equivalence with respect to $H$). Let $\Ad$ be the poset of adequate equivalences, $\Wl$ the poset of (isomorphism classes of) initial Weil cohomologies and $\Wl^\ab$ the poset of (isomorphism classes of) ab-initial Weil cohomologies: by Proposition \ref{p6.3}, we  get nondecreasing maps
\begin{equation}\label{eq6.4}
W:\Ad\leftrightarrows\Wl:\sim, \quad W^\ab:\Ad\leftrightarrows\Wl^\ab:\sim
\end{equation}

\begin{thm}\label{t6.2} Consider the objects in \eqref{eq6.4} as categories and functors. Then $(W,\sim)$ and $(W^\ab,\sim)$ are two pairs of adjoint functors where $\sim$ is the right adjoint. In particular, $\sim W \sim = \sim$, $\sim W^\ab \sim = \sim$, $W\sim W = W$ and $W^\ab\sim W^\ab = W^\ab$ (``Galois correspondence'').
\end{thm}

\begin{proof} Given $H$ and $\sim$, $H$ is compatible with $\sim$ in the sense of Definition \ref{d5.3} if and only if  $\sim\ge \sim_H$. If $H$ is initial (\resp ab-initial), this is also equivalent to $W_\sim\ge H$ (\resp $W_{\sim,\ab}\ge H$), hence the adjunction claims. The Galois correspondences then follow from the adjunction identities.
\end{proof}

To be more explicit, for any Weil cohomology $H$ one has a commutative diagram of $\otimes$-categories:
\begin{equation}\label{eq6.5}
\begin{CD}
\sM_{\sim_H}@>>> \sW_{\sim_H} @>>> T(\sW_{\sim_H})@>\sim>> \sW_{\sim_H}^\ab\\
@V||VV@Va VV @Vb VV @Vc VV\\
\sM_H@>>> \sW_H@>>> T(\sW_H)@>d >> \sW_H^\ab
\end{CD}
\end{equation}
where $a$ is full surjective and $b,c,d$ are Serre localisations. It is unclear when $a$ and $d$ are equivalences in practice.

\subsection{The case of traditional Weil cohomologies}\label{s6.2a} Let $(\sC,H)$ be a Weil cohomology, with $(\sC,L_\sC)\in \Ex^\rig_*$ and $\sC$ connected (Definition \ref{d2.1}). 
Then $Z(\weil_H)\subseteq Z(\weil_H^\ab)\subseteq Z(\sC)$ are domains; but $Z(\weil_H^\ab)$ is absolutely flat (Lemma \ref{l2.1} d)), so $\weil_H^\ab$ is also connected. Let $\weil_H^Q$ be the extension of scalars of $\weil_H$ to the field of fractions $Q$ of its centre; then the functor $\iota_H:\weil_H \to \weil_H^\ab$ extends to a faithful $\otimes$-functor
$\weil_H^Q \to \weil_H^\ab$ (note that the centres may a priori differ).

If Hom groups in $\sC$ are finite dimensional over its centre, we can apply Lemma \ref{lak}. We thus get a commutative diagram of rigid $\otimes$-categories
\begin{equation}\label{eq6.1}
\begin{gathered}
\xymatrix{
\sM_H\ar[r]^{\underline{W_H}}\ar@/^2.8pc/[rrrr]^{\uH}\ar[d]&(\weil_H)^\natural\ar[r] &(\weil_H^Q)^\natural\ar[r]\ar[d] & \weil_H^\ab \ar[r]& \sC\\
\sM_{H,\sN}\ar[rr]\ar[d] &&(\weil_H^Q/\sN)^\natural\\
\sM_\num
}
\end{gathered}
\end{equation}
in which the horizontal functors are faithful, the vertical functors are full, and $\sM_\num,(\weil_H^Q/\sN)^\natural$ are abelian semi-simple. The (pseudo-abelian) category $\sM_{H,\sN}$ is defined by the diagram. All functors are symmetric monoidal, except the horizontal ones starting from $\sM_H$ and $\sM_{H,\sN}$ which are only monoidal (see Definifion \ref{defhom}).

This applies in particular to any traditional Weil cohomology. 

\begin{rk} One could further decorate Diagram \eqref{eq6.1} by adding Diagram \eqref{eq6.5}. We shall refrain from this and leave it to the pleasure of enthusiastic readers.
\end{rk}

\subsection{The case of classical Weil cohomologies}\label{s6.2b} 

As an instance of Examples \ref{exw} - \ref{ex5.1}, let $k$ be a subfield of $\C$ and $\sV=\Sm^\proj(k)$.  Hodge cohomology in $\Weil(k, \V; \sH, L_\sH)$, for $\sH$ the category of pure polarisable Hodge structures and $L_\sH:= \Q(-1)$, is an enrichment of the classical Betti cohomology ${H}=H_B$, with $F=\iota_\sH: \sH\to \Vec_\Q$ the forgetful functor. 
Other examples are the de Rham-Betti enrichement of Betti cohomology, and the Ogus enrichment of de Rham cohomology as in \cite[Ch. 7]{andremotifs}. These examples are all superseded by Deligne's (abelian semi-simple) category $\sM^D$ of motives for absolute cycles \cite[\S 6]{delmil}. Here, contrary to \cite{delmil}, we don't modify the commutativity constraint, in order that the natural functor $\sM_H\to \sM^D$ be symmetric monoidal.

In any characteristic, we also have the Galois enrichment of $\ell$-adic cohomology.

\subsection{Andr\'e's category}\label{s6.2}
The previous examples are further enriched by André's category of motivated motives $\sM^A_H(k,\sV)=\sM^A_H$ (\cite[4.2]{A} and \cite[\S 2]{A2}) associated to a classical Weil cohomology $H$ (same comment for the commutativity constraint).  Theorem \ref{thm2bis} gives a universal enrichment of all these cases:
\begin{equation}\label{eq6.2}
\sM_H\by{w} (\weil_H)^\natural \to \sM^A_H
\end{equation}
where all functors are faithful. We distinguish two cases:

\subsubsection{Characteristic $0$, $k$ embeddable in $\C$} Then none of the categories in \eqref{eq6.2} depends on the choice of $H$: we drop it from the notation $\sM^A_H$ and replace it by $\hom$ for the others this is the usual homological equivalence). Moreover, $\sM^A$ is abelian semi-simple by \cite[Th. 0.4]{A}. We thus get the following refinement of \eqref{eq6.2} still into faithful functors, with the last two exact:
\begin{equation}\label{eq6.2a}
\sM_\hom\by{w} (\weil_\hom)^\natural\by{\iota^\natural_\hom} \weil_\hom^\ab \by{\theta} \sM^A\to \sM^D\to \sH.
\end{equation}

Here all centres are equal to $\Q$ (hence passing from $\weil_\hom$ to $\weil_\hom^Q$ is not necessary). By Lemma \ref{l6.1}, all Weil cohomologies appearing in \eqref{eq6.2a} have weights since $\sH$ has.

\begin{thm}\label{t6.4} In \eqref{eq6.2a}, $\theta$ is an equivalence of categories; hence $\weil_\hom^\ab$ is semi-simple.
\end{thm}

\begin{proof} Recall that $\sM^A$ is constructed by adjoining to algebraic cycles the inverses $\Lambda^{n-i}$ of the Lefschetz isomorphisms $L^{n-i}:H^i(X)\iso H^{2n-i}(X)(n-i)$ associated to a polarised $X\in \sV$ of dimension $n$ (see \S \ref{s8.5} below for details on these operators). Being exact and faithful, $\theta$ is conservative; therefore $\Lambda^{n-i}\in \IM\theta$ and $\theta$ is full. Its essential surjectivity now follows from Lemma \ref{l2.2} since the functor $\sM_\hom\to \sM^A$ is dense  (see Definition \ref{d2.2} for ``dense'').
\end{proof}

\subsubsection{$k$ is finitely generated of characteristic $>0$} Taking for $H$ in \eqref{eq6.1} $\ell$-adic cohomology $H_\ell$ where $\ell$ is a prime number different from $\car k$ and for $\sC=\sR_\ell$ $\ell$-adic representations, except that we now have inclusions
\[\Q\subseteq Z(\weil_{H_\ell})\subseteq Z(\weil_{H_\ell}^Q)\subseteq Z(\weil_{H_\ell}^\ab)\subseteq Z(\sR_\ell)= \Q_\ell\]

Here again all Weil cohomologies have weights since $\sR_\ell$ has \cite{weilII}. Since $\sM^A_{H_\ell}$ is not known to be abelian, we have another set of inclusions
\[Z(\weil_{H_\ell})\subseteq Z(\sM^A_{H_\ell})\subseteq  \Q_\ell.\]

Finally, it is not known whether $\ell$-adic homological equivalence is independent of $\ell$, so this picture a priori varies with $\ell$. We shall refine it in \S \ref{S:A}.

\subsection{Comparison of Weil cohomologies}
We may introduce the following relations.

\begin{defn}\label{dcomp.1} 
a) Two Weil cohomologies $(\sC, H)$ and $(\sC',H')$ are \emph{equivalent} if $\weil_H = \weil_{H'}$. For $\sC, \sC'\in \Ex^\rig$ we say that are \emph{ab-equivalent} if  $\weil_H^\ab = \weil_{H'}^\ab$. Clearly, these are equivalence relations.

b) Two Weil cohomologies $(\sC, H)$ and $(\sC',H')$ are \emph{comparable} if there exists a third pointed category $(\sC'',L_{\sC''})\in \Add^\otimes_*$, two faithful $1$-morphisms $(F,u):(\sC,L_\sC)\to (\sC'',L_{\sC''})$  and $(F',u'):(\sC',L_{\sC'})\to (\sC'',L_{\sC''})$ and a comparison isomorphism of Weil cohomologies $ F_*H \iso F'_*H'$ in $\Weil(k,\V; \sC'',L_{\sC''})$. If $\sC, \sC'\in \Ex^\rig$ we additionally require that $\sC''\in \Ex^\rig$, and  $F, F'$ should be exact. 
\end{defn}

\begin{lemma}\label{dcomp.l1} 
If $(\sC, H)$ and $(\sC',H')$ are comparable, they are equivalent. 
\end{lemma}

\begin{proof}
If $(\sC, H)$ and $(\sC',H')$ are comparable, then $(\sC, H)$ and $(\sC',H')$ are both enrichments of $H'':=  F_*H \iso F'_*H'$, and  then $\weil_H = \weil_{H''}= \weil_{H'}$ by Proposition \ref{p6.1}. Same argument in the abelian case. 
\end{proof}

Let  $\sS =\{(\sC, H)\mid\sC\in\Ex^\rig\}$ be a class of Weil cohomologies; write $\sI_{\sS} =\bigcap_{H\in \sS} \Ker F_H$ and let
$$\sA_\sS := T(\weil)/\sI_{\sS}$$ 
be the abelian rigid $\otimes$-category given by the Serre quotient. Let $r_H : \sA_\sS\to \sC$ be the induced exact $\otimes$-functor for each $(\sC, H)\in\sS$.

\begin{thm}\label{thm:AS}
For $\sA_{\sS}$ as above, consider the conditions
\begin{thlist}
\item $\sA_\sS$ is connected
\item  every $r_H$ is faithful
\item  all $H\in\sS$ are equivalent
\item $\sA_\sS$ is Tannakian.
\end{thlist}
Then (i) $\Rightarrow$ (ii) $\Rightarrow$ (iii) and (iv) $\Rightarrow$ (i); if $\sS$ contains a traditional Weil cohomology, then (iii) $\Rightarrow$ (iv) and all conditions are equivalent. Moreover, $\sA_S$ is semi-simple if $\car k=0$ and $\sS$ contains a classical Weil cohomology. 
\end{thm}

\begin{proof} If $\sA_\sS$ is connected then all $r_H$ are faithful by Lemma \ref{l2.1} a), which in turn implies $\weil_H^\ab = \weil_{H'}^\ab$ since $\Ker F_H=\Ker F_{H'}$ for $H, H'\in \sS$. This shows (i) $\Rightarrow$ (ii) $\Rightarrow$ (iii), and (iv) $\Rightarrow$ (i) is obvious. If $H\in\sS$ is traditional, $r_H$ is a fiber functor and $\sA_\sS$ is then Tannakian under (iii) and all conditions are equivalent.  Finally, the last statement follows from Theorem \ref{t6.4}.
 \end{proof}

\section{K\"unneth decompositions and numerical equivalence}

\subsection{Condition C}\label{s5.2}
Let  $(\sC,H)$ be a Weil cohomology with $\sC$ pseudo-abelian, and consider the corresponding category $\sM_H(k,\V)=\sM_H$ of $H$-homological motives of Definition \ref{defhom}.

\begin{defn}\label{d4.1} We say that $X\in \V$ \emph{verifies Condition C relatively to $H$} if, for  any $i\ge 0$, the K\"unneth projector $\End_{\sC^\natural}(\uH(X))\ni \pi^i_X:\uH(X)\surj H^i(X)\inj \uH(X)$ is in the image of $\uH$. We say that $H$ \emph{verifies Condition C} if all $X\in \V$ do.
\end{defn}

(If $H$ is a traditional Weil cohomology and $\V=\Sm^\proj_k$, we recover the standard conjecture C of \cite[p. 195]{gro-standard}.)

\begin{exs} \label{exC} Here are examples where Condition C holds:\
\begin{itemize}
\item $0$-dimensional varieties and $\P^1$;
\item $\sV=\sAb$, complete intersections, etc., $H$ any Weil cohomology, see Theorem \ref{7.1} below, and
\item over a finite field $k$, $\sV =\Sm^\proj(k)$ \cite{km}, $H$ classical.
\end{itemize}
\end{exs}

\begin{lemma}\label{l4.3bis} Let  $H$, ${H'}$ and $F$ be as in Lemma  \ref{l4.3}.
If $H$ verifies Condition C, so does ${H'}$, and conversely if $F$ is faithful (\ie if $H$ is an enrichment of ${H'}$ as in Definition \ref{relh}).
\end{lemma}

\begin{proof}
We have the factorisation
$$\xymatrix{\sM_H\ar[r]^{\ \ \ \uH}\ar[d]& \sC^{\natural}\ar[d]^{F^\natural}
\\ \sM_{H'}\ar[r]^{\ \ \ \uH'}&{\sC'}^{\natural}.  }$$

If the projector $\pi^i_X\in\End_{\sC^\natural}(\uH(X))$ is in the image of $\uH$ the projector $\pi'^i_X= F^\natural(\pi^i_X)\in\End_{{\sC'}^\natural}(\uH'(X))$ is in the image of $\uH'$. If $F$ is faithful then $F^\natural$ is faithful, $\sM_H=\sM_{H'}$ and the converse is true. 
\end{proof}

The following is a  special case of Lemma \ref{l3.4} b):

\begin{lemma}\label{l4.5} 
a) If $X,Y$ verify Condition C, so do $X\coprod Y$ and $X\times Y$.\\
b) If $X$ verifies Condition $C$ and $h_H(Y)$ is isomorphic to a direct summand of $h_H(X)$, then $Y$ also verifies Condition C.\qed
\end{lemma}

(Here we write $h_H$ instead of $h$, to stress that the statement concerns motives in $\sM_H$.)

Suppose that $H$ verifies Condition C.  Still by Lemma \ref{l3.4} b), $\sM_H$ acquires a weight grading $w$ in the sense of Definition \ref{d3.5} b) and we get the twisted category $ \fake{\sM}_H$ as in  Notation \ref{n3.1}.

For $X\in\Corr$, keep the notation $\pi^i_X$ for the element of \break $\End_{\sM_H}(h(X))\allowbreak=\End_{\fake{\sM}_H}(h(X))$ mapping to $\pi^i_X$ via $\uH$. This is a set of orthogonal idempotents with sum $1_{h(X)}$, which yields a decomposition
\[h(X)=\bigoplus_{i\ge 0} h^i(X)\]
such that $\uH(h^i(X))=H^i(X)$. This defines a Weil cohomology $h_H$ with values in $\fake{\sM}_H$,  which has weights and such that $(\sC,H)$ is the push-forward of $(\overset{\boldsymbol{\cdot}}{\sM}_H,h_H)$ by $\uH$. 

\begin{rk}\label{r7.1} Even if, additionally,  $\overset{\boldsymbol{\cdot}}{\sM}_H$ and $\sC$ are abelian the faithful $\otimes$-functor $\uH$ is not granted to be exact, in general. However, this is the case if $\sC$ is connected by Lemma \ref{l2.1} b). 
\end{rk}

\begin{thm}\label{wconjC} a) Let $(\sC, H)$ be a Weil cohomology and let $(\weil_H,W_H)$ be as in Theorem \ref{thm2bis}. 
The following conditions are equivalent:
\begin{thlist}
\item $H$ verifies Condition C;
\item the functor $\sM_H\to \weil_H^\natural$ induced by $W_H$ is an equivalence of categories.
\end{thlist}
If this is true, then $Z(\weil_H)=\Q$. 

b) Moreover, if $\sC\in\Ex^\rig$ let $(\weil_H^\ab,W_H^\ab)$ also be as in Theorem \ref{thm2bis}. The following conditions are equivalent:
\begin{thlist}\stepcounter{spec}\stepcounter{spec}
\item $H$ verifies Condition C, $\sM_H$ is abelian and $\uH$ is exact;
\item $H$ verifies Condition C, $\weil_H^\natural$  is abelian and $\iota_H^\natural : \weil_H^\natural\to \weil_H^\ab$ is exact;
\item the functors $\sM_H\to \weil_H^\natural$  and $\iota_H^\natural: \weil_H^\natural\to \weil_H^\ab$ are equivalences;
\item the functor $w_H :\sM_H\to \weil_H^\ab$ induced by $W_H^\ab$ is an equivalence of categories.
\end{thlist}
If this is true, then $Z(\weil_H^\ab)=\Q$. 
\end{thm}

\begin{proof} 
(i) $\Rightarrow$ (ii) is obvious by definition of $(\weil_H,W_H)$, and (ii) $\Rightarrow$ (i) is more obvious.
If $\sC\in\Ex^\rig$, from Theorem \ref{thm2bis}  we get the following factorisation
$$\xymatrix{
\sM_{H}\ar@/^1.6pc/[rr]^{\uH}\ar[r]^{}\ar[dr]_{w _H}&\weil_H^\natural\ar[d]^{\iota_H^\natural}\ar[r]^{\alpha}&\sC\\
&\weil_H^\ab\ar[ru]_{\beta}&
}$$
where all functors are faithful and $\beta$ is exact by construction. Using the equivalence (i) $\iff$ (ii) and the initiality of $\weil_H^\ab$, we get (iv) $\Rightarrow$ (v), and  (v) $\Rightarrow$ (vi) $\Rightarrow$ (iii) are obvious. Finally, under (iii) 
 $\alpha$ is identified with the functor $\uH$, and (iv) follows because $\alpha$ exact $\Rightarrow$ $\iota_H^\natural$  exact.

The statements regarding the centres are obvious, since $Z(\sM_H)=\Q$.
\end{proof}

In the universal case, this gives:

\begin{cor} \label{corwconjC}
If the conditions of Theorem \ref{wconjC} a) hold for $H=W$, then $\sM_W\iso \weil^\natural$ and  every Weil cohomology verifies Condition C.

If the conditions of Theorem \ref{wconjC} b) hold for $H=W_\ab$, then \eqref{eq6.3} is an equivalence of categories and every Weil cohomology with abelian target verifies Condition C. 
\end{cor}

We also have the following significant fact. Let $F,H,H'$ be as in Lemma \ref{l4.3}, with $(F,u):(\sC,L_\sC)\to (\sC',L_\sD')\in \Ex^\rig_*$. Applying Proposition \ref{p6.1}, get the following commutative square
\begin{equation}\label{diagram}
\begin{gathered}
\xymatrix{
\sM_H\ar[r]^{w_H} \ar[d]&\weil_H^\ab\ar[d]^{\varphi}& \\
\sM_{H'}\ar[r]^{w_{H'}}&\weil_{H'}^\ab& 
}
\end{gathered}
\end{equation}
where $\varphi$ is a localisation.

\begin{lemma} \label{cab} In \eqref{diagram}, assume that $\sM_H=\sM_{H'}$ is abelian, $H'$ verifies Condition C  and $\uH'$ is exact. Then\\
a) $w_{H'}$ is an equivalence of categories,\\
b) the following are equivalent:
\begin{thlist}
\item $H$ verifies Condition C and $\uH$ is exact
\item $Z(\weil_H^\ab)=\Q$
\item $\weil_H^\ab$ is connected
\item $\varphi$ is an equivalence in \eqref{diagram}.
\end{thlist}
\end{lemma}

\begin{proof}  Note that the hypothesis of Lemma \ref{cab} is the same as condition (iii) of Theorem \ref{wconjC} b) for $H'$. This gives a), by the implication (iii) $\Rightarrow$ (vi) in this theorem. This also gives (i) $\Rightarrow$ (ii).  (ii) $\Rightarrow$ (iii) is obvious, and (iii) $\Rightarrow$ (iv) since $\varphi$ is then faithful by Lemma \ref{l2.1} a). Finally,  (iv) implies  by a) that all functors in \eqref{diagram} are equivalences, which in turn readily implies (i).  
\end{proof}

\subsection{Conditions D and V}\label{s5.3} 
For any Weil cohomology $(\sC,H)$, recall that we get $H$-homological equivalence $\sim_H$ (Definition \ref{defhom}). Moreover, denote $\sim_\tnil$ Voevodsky's  smash-nilpotence equivalence and $\sim_\num$ numerical equivalence: as any other adequate equivalence relation $\sim_\tnil$ and $\sim_H$ are both finer than $\sim_\num$.

\begin{defn} \label{d1.5v}
We say that $(k,\sV)$ \emph{verifies Condition V} if $\sM_\tnil\to \sM_\num$ is an equivalence.
\end{defn}

(Of course, V stands for Voevodsky!)

\begin{defn} \label{d1.5} Let $(\sC,H)$ be a Weil cohomology.
 We say that $H$ \emph{verifies Condition D} if the functor $\sM_H\to \sM_\num$ is an equivalence.
\end{defn}

(If $H$ is a traditional Weil cohomology and $\V=\Sm^\proj_k$, a) is conjecture D in \cite[p. 17]{kst}.)

\begin{prop}\label{p7.1} Condition V implies Condition D if $H$ is abelian-valued.
\end{prop}

(This does not \emph{a priori} extend to general Weil cohomologies.)

\begin{proof} By  Remark \ref{rk:tweil}, we have a factorisation
\begin{equation}\label{eq5.4}
\sM_\tnil \to\sM_H\to \sM_\num.
\end{equation}
\end{proof}

Suppose that $H$ is abelian-valued; consider the following diagram 
\begin{equation}\label{eq5.5}
\begin{gathered}
\xymatrix{
&\sM_{H}\ar[d]^\lambda\ar[dl]_\pi\ar[dr]^{w _H}&\\
\sM_\num&T(\sM_{H})\ar[l]^{\bar\pi}\ar[r]_{\bar w _H}&\weil_H^\ab
}
\end{gathered}
\end{equation}
where $ w _H$ is  as in Theorem \ref{wconjC} (vi) and $\bar  w _H$ is the exact $\otimes$-functor induced by the universal property of $T(\sM_H)$, as $\weil_H^\ab$ is abelian; similarly, $\bar\pi$ exists because $\sM_\num$ is also abelian semisimple by \cite{jannsen3}. Moreover, we have:

\begin{lemma}\label{l7.1} In \eqref{eq5.5} the functors $\pi$ and $\bar\pi$ are full, $\bar \pi$ is a localisation, and $\lambda$ is faithful.
\end{lemma}

\begin{proof} The fullness of $\pi$ is obvious. Applying  \cite[Ex. 6.4]{standard-schur} to it, we get that $T(\pi)$ is a full localisation. But $\sM_\num$ is abelian semi-simple, hence $\sM_\num\iso T(\sM_\num)$ by Lemma \ref{bvk}. This identifies $\bar \pi$ with $T(\pi)$.  Finally, $\lambda$ is faithful because $w _H$ is faithful.
\end{proof}

\begin{thm} \label{t5.1} Let $(\sC,H)$ be an abelian-valued Weil cohomology. 
Then\\
a) The following conditions are equivalent:
\begin{thlist}
\item $H$ verifies Condition D (\ie $\pi$ is an equivalence of categories).
\item $\bar\pi$ and/or $\lambda$ is an equivalence of categories.
\item $T(\sM_H)$ is connected.
\item $Z(T(\sM_H))=\Q$.
\end{thlist}
These conditions imply that $\bar w_H$ is faithful and $\uH$ is exact.\\
b) Suppose that $H$ verifies Condition C.  Then the following are equivalent:
\begin{thlist}\stepcounter{spec}\stepcounter{spec}\stepcounter{spec}\stepcounter{spec}
\item $H$ verifies Condition D.
\item $\sM_H$ is abelian, $\uH$ is exact and $\bar w _H$ is faithful.
\item The functors $w _H,\bar w _H$ are equivalences.
\item The functors $w _H,\bar w _H$ and $\lambda$ are equivalences.
\item All functors in \eqref{eq5.5} are equivalences.
 \end{thlist}
\end{thm}

\begin{proof} 
a) (ii) $\Rightarrow$ (iv) $\Rightarrow$ (iii) are trivial. It remains to prove (i) $\Rightarrow$ (ii) and (iii) $\Rightarrow$ (i).
If $H$ verifies Condition D, then $\sM_H$ is semisimple hence $\lambda$ is an equivalence by Lemma \ref{bvk}, therefore so is also $\bar \pi$. 
If $T(\sM_H)$ is connected, $\bar\pi$ is faithful being exact (Lemma \ref{l2.1} a)), and so  is $\pi$ as a composition of two faithful functors. Moreover, $\bar w_H$ is faithful, and $\uH$ is exact because $\sM_H$ is semi-simple.\\
b) (v) $\Rightarrow$ (vi) has been seen in a) (without assuming Condition C).  (vi) $\Rightarrow$ (vii): for $w_H$ it follows from Theorem  \ref{wconjC} b), (iii) $\Rightarrow$ (v), and for $\bar w_H$ we use the same argument. (vii) $\Rightarrow$ (viii) is obvious. (viii) $\Rightarrow$ (ix) $\Rightarrow$ (i) now follow from a) (equivalence (i) $\iff$ (ii)). 
\end{proof} 

\begin{ex}  \label{ex7.1}
Suppose that $H$ is an enrichment of a classical Weil cohomology $H'$. If Condition D holds for $H$, it obviously holds for  $H'$ (see Lemma \ref{l4.3}). Then $H'$ verifies Condition C by \cite[5.4.2.1]{A} (see Example \ref{exwsa} and Theorem \ref{t8.1} below), hence also $H$ by Lemma \ref{l4.3bis}, and we are in the situation of Theorem \ref{t5.1} b).

However, we don't know any proof of the implication D $\Rightarrow$ C in this generality. The situation will change in Section \ref{s8} when we introduce abstractions of the weak and hard Lefschetz theorems; Theorem \ref{t5.1} will be upgraded to Theorem \ref{wconjB}.
\end{ex}

Now consider the universal case of $W_\ab$-homological motives. If $H=W_\ab$ then $\weil_{H}^\ab=T(\weil)$ in  \eqref{eq5.4}. We simplify its notation to $w_H=w$ and $\bar w_H=\bar w$. To say that $W_\ab$ verifies Condition D means that $\sim_\hun = \sim_\num$ and, equivalently, that $Z(T(\sM_\hun)=\Q$, by Theorem \ref{t5.1} a).

\begin{thm}  \label{t5.2} Assume that $W_\ab$ verifies Condition D. Then $W_\ab$ verifies Condition C
$\iff$ $T(\weil)$ is connected $\iff$ $Z(T(\weil))=\Q$.\end{thm}
\begin{proof}
If Condition D holds for $W_\ab$ it holds for any Weil cohomology with values in an abelian category. In particular, pick any classical Weil cohomology $H$; then $H$ verifies Conditions C and D (as in Example \ref{ex7.1}).  We then apply Lemma \ref{cab}. 
\end{proof}

As a special case of (v) $\Rightarrow$ (ix) in Theorem \ref{t5.1}, we get:

\begin{cor} \label{c7.1}
Under Conditions C and D for $W_\ab$, all functors in \eqref{eq5.5} are equivalences and these equivalences identify $W_\ab$ with $h:\sM_\rat\to \sM_\num$. \qed
\end{cor}

(See also Corollary \ref{c7.1a} below.)

\subsection{Chow-Künneth decompositions} Let $(\sC,H)$ be a Weil cohomology. Let us start with a definition:

\begin{defn} Let $X\in \sV$. We say that $X$ admits a Chow-Künneth decomposition relatively to $H$ if
\begin{itemize}
\item $X$ verifies Condition C relatively to $H$ (Definition \ref{d4.1});
\item the Künneth projectors $\pi^i_X$ lift to a set of orthogonal idempotents in $\End_{\sM_\rat}(h(X))$.
\end{itemize}
A \emph{Chow-Künneth decomposition relative to $H$} is such a lift.
\end{defn}

Here are some elementary observations on this definition.

\begin{rks}\label{chkrk}\
\begin{enumerate}
\item A Chow-Künneth decomposition relative to $H$ is also a Chow-Künneth decomposition relative to any Weil cohomology $H'$ which is a push-forward of $H$ as in \eqref{eq.fu} of Construction \ref{cons1}, \cf Lemmas \ref{l4.3} - \ref{l4.3bis}.
\item A Chow-Künneth decomposition, if it exists, is not unique, because one can conjugate it by any invertible self-correspondence of $X$ which is $\equiv 1\pmod{\sim_H}$. Moreover, it has \emph{a priori} no reason to be unique up to such conjugation. This will hold, however, if $\Ker(\End_{\sM_\rat}(h(X))\to \End_{\sM_\num}(h(X)))$ is nilpotent, for example if $X$ is an abelian variety \cite{kimura}.
\item Given $H$, if $X$ and $Y$ admit a Chow-Künneth decomposition, so do $X\coprod Y$ and $X\times Y$ as in Lemma \ref{l4.5} a). However, the analogue of Lemma \ref{l4.5} b) (stability under direct summands) is not clear \emph{a priori}; see nevertheless Proposition \ref{p5.6}.
\item There may be partial Chow-Künneth decompositions, as we shall see now. (The same is true for Künneth decompositions.)
\end{enumerate}
\end{rks}

In \cite{murre}, Murre constructed a partial Chow-Künneth decomposition of the form
\begin{equation}\label{eq5.1}
h(X)=h^0(X)\oplus h^1(X)\oplus h^{[2,2n-2]}(X)\oplus h^{2n-1}(X)\oplus h^{2n}(X)
\end{equation}
and isomorphisms
\begin{equation}\label{eq7.1}
h(\pi_0(X))\iso h^0(X), h^0(X)\otimes \L^n\iso h^{2n}(X) , h^1(X)\otimes \L^{n-1} \iso h^{2n-1}(X)
\end{equation}
for any $X\in \Sm^\proj(k)$,  where $\pi_0(X)$ is the scheme of constants of $X$. This was refined by Scholl in \cite[\S 1 and \S 4]{scholl} to a self-dual decomposition:
 \begin{equation}\label{eq7.2}
h^j(X)^\vee\simeq h^{2n-j}(X)(n).
\end{equation}

By \cite[3.7]{murre}, we have
\begin{equation}\label{eq5.7}
H^i(h^j(X))=
\begin{cases}
0 &\text{ if } i\ne j\\
H^i(X) &\text{ if } i= j
\end{cases}
\end{equation}
for $j\in \{0,1,2n-1,2n\}$ if $H$ is $\ell$-adic cohomology, hence for $H$ classical in the sense of Definition \ref{d3.2} if we are in characteristic $0$. The basic reason (see Example \ref{p5.3} below) is that \eqref{eq5.1} and  \eqref{eq7.1} lift the partial Künneth decomposition and isomorphisms established by Kleiman in \cite{kdix} to rational equivalence. We now examine what happens for a general Weil cohomology $H$: obviously, if \eqref{eq5.7} holds then $X$ verifies Condition C relatively to $H$ in the sense of Definition \ref{d4.1}.

\begin{prop}\label{p5.1} a) The first isomorphism of \eqref{eq7.1} yields a decomposition
\[H^0(X)\simeq H^0(\pi_0(X))\oplus S
\]
where $S=0$ (for all $X$) $\iff$ \eqref{eq5.7} holds for $i=0$ (for all $X$) $\iff$ $H$ is normalised  in the sense of Definition \ref{d3.4}.\\
b) Suppose that $H$ is normalised.  If the saturation of $\V$ contains all curves, we have $H^i(h^j(X))=0$ for $i\ne j$ for any $X\in \sV$ and $j\in \{0,1,2n-1,2n\}$.
\end{prop}

\begin{proof} a) is tautologically true. b) is obvious for $j=0$ since $h^0(X)$ is an Artin motive \cite[4.1.6.1]{andremotifs}. For $j=1$, suppose first $n=1$. For any $i$, we have
\[H^i(X)=H^i(h^0(X))\oplus H^i(h^1(X))\oplus H^i(h^2(X))\]
so the statement is true for $i\notin [0,2]$; it also holds for $i=0$ by the normalised hypothesis, and finally for $i=2$ by Axiom (v) of Definition \ref{d1.1}, and \eqref{eq7.2}.

If $n>1$, we use the fact that $h^1(X)$ is a direct summand of $h^1(C)$ for an ample curve $C$ (which follows from \cite[4.3]{scholl}). The cases $j=2n-1$ and $j=2n$ then follow again by Poincar\'e duality.
\end{proof}

\begin{prop}\label{p5.6}  Suppose that $H$ is normalised.\\
a) \eqref{eq5.7} holds for $i=1$ if and only if one has $H^1(h^{[2,2n-2]}(X))=0$; in particular it holds for curves.\\
b) The condition of a) is stable under coproducts, products and direct summands as in Lemma \ref{l4.5}.
\end{prop}

\begin{proof} The first statement of a) follows from Proposition \ref{p5.1} b), and the case of curves then follows from Proposition \ref{p5.1}. The first statement of b) follows from the Künneth formula for $H^*$ and for the partial Chow-Künneth decompositions, while the second one is immediate from a).
\end{proof}

 Let $T_X$ be the canonical torsor under the Albanese variety of $X$, and let $a_X:X\to T_X$ be the canonical Albanese map. 

\begin{prop}\label{c5.1} The map $a_X^*:H^1(T_X)\to H^1(X)$ is split injective. Moreover, the condition of Proposition \ref{p5.6} a) holds if and only if $a_X^*$ is an isomorphism.
 \end{prop}

\begin{proof}
 The map $a_X$ induces an isomorphism 
\begin{equation}\label{eq5.2}
h^1(T_X)\iso h^1(X)
\end{equation}
\cf \cite[Ex. 6.38]{zetaL}. Using this and Proposition \ref{p5.1}, we get a commutative diagram
\[\begin{CD}
H^1(T_X)@>>> H^1(X)\\
@A\wr AA @AAA\\
H^1(h^1(T_X))@>\sim>> H^1(h^1(X))
\end{CD}\]
in which the left vertical map and the bottom horizontal map are isomorphisms. The claim follows.
\end{proof}

\section{Lefschetz operators} \label{s8}

In this section, we generalise to arbitrary Weil cohomologies important properties satisfied by the classical ones: Normalised (Definition \ref{d3.4}), Albanese-invariance (Definition \ref{defalbinv}), Weak and Strong Lefschetz (Definition \ref{d5.2a}).  
By the same technique as before, we then get a universal Weil cohomology enjoying these extra properties. Moreover, as in the classical case, Lefschetz theory provides a natural context for the yoga of the standard conjectures, as in Definition  \ref{d1.8b}. 

\subsection{More formalism of Weil cohomologies} Let $H$ be a Weil cohomology. 
 If $X\in \V$, we write $\cup_X$ for the ``cup-product''
\[H^*(X)\otimes H^*(X)\by{\kappa_{X,X}^{*,*}} H^*(X\times X)\by{\Delta_X^*} H^*(X)\]
where the first map is given by the K\"unneth product and $\Delta_X$ is the diagonal of $X$.

If $f:Y\to X$ is a morphism in $\V$, we write $f^*$ for its (contravariant) action on $H^*$, and $f_*$ for the action of the transpose of the graph of $f$, viewed as a correspondence (see proof of Proposition \ref{p1.1} a)). 

\begin{lemma}[Projection formula]\label{l5.1} The diagram
\[\xymatrix{
H^*(X)\otimes H^*(Y) \ar[r]^{1\otimes f_*}\ar[d]^{f^*\otimes 1}&H^*(X)\otimes H^*(X)\ar[r]^-{\cup_X}&H^*(X)\\
H^*(Y)\otimes H^*(Y)\ar[r]^-{\cup_Y} &H^*(Y)\ar[ur]_{f^*}
}\]
commutes.
\end{lemma}

\begin{proof} Passing through $\kappa$, we are reduced to the following identity:
\[f_* \circ \Delta_Y^*\circ (1\times f)^* = \Delta_X^* \circ (1\times f)_* \]
which follows from the same identity in the category of graded correspondences, which in turn follows by Manin's identity principle from the projection formula for Chow groups \cite[Prop. 8.3 (c)]{fulton}.
\end{proof}

Let $i:Y\subset X$ be a closed immersion of codimension $1$ in $\V$. 

\begin{defn}\label{lefdef}
The class $\cl^1_X(Y):\un\to H^2(X)(1)$ induces 
\[L_Y:H^k(X)\by{1\otimes \cl^1_X(Y)} H^k(X)\otimes H^2(X)(1)\by{\cup_X} H^{k+2}(X)(1)\]
which we call a \emph{Lefschetz operator}.
\end{defn}

\begin{lemma}\label{l5.2}  The operator $L_Y$ may be factored as
\[H^k(X)\by{i^*} H^k(Y)\by{i_*} H^{k+2}(X)(1).\]
\end{lemma}

\begin{proof} Indeed, the projection formula for Chow groups plus Manin's identity principle \cite[\S 2]{scholl} show that the following diagram commutes in $\sM(k,\sV)$:
\[\begin{CD}
h(X)@>1\otimes [Y]>> h(X)\otimes h(X)(1)\\
@V i^* VV @V\Delta_X^* VV\\
h(Y)@> i_* >> h(X)(1)
\end{CD}
\]
where $\Delta_X$ is the diagonal and $[Y]\in CH^1(X)$ is viewed as a morphism $\un\to h(X)(1)$, and we apply the $\otimes$-functor $H^*$.
\end{proof}

\begin{rk}\label{r8.1} Let  $\varpi:X\inj \P^N$ be a polarisation and $i=Y\inj X$ be a corresponding smooth hyperplane section. Then $\cl^1_X(Y)$ only depends on  $\varpi$ in Definition \ref{lefdef}: indeed, the class of $Y$ in $CH^1(X)$ is $\varpi^* \sO(1)$ where $\sO(1)$ is the canonical generator of $CH^1(\P^N)\simeq \Z$.
\end{rk}

\subsection{Albanese-invariant cohomologies}

\begin{defn}\label{defalbinv}
We say that $H$ is \emph{Albanese-invariant} if \eqref{eq5.7} holds for $i=1$ for any $X\in \V$  or, equivalently, if the condition of Proposition \ref{c5.1} holds for any $X\in \V$, \ie $a^*_X$ is an isomomorphism. 
\end{defn}

\begin{ex}\label{p5.3} 
Any  classical Weil cohomology is Albanese-invariant. To see this, we can reduce to the condition in Proposition \ref{c5.1}. For $\ell$-adic cohomology, see \cite[Th. 2A9 6.]{kdix}. In characteristic $0$, we get the other classical Weil cohomologies by the comparison theorems. Over a finite field, we get crystalline cohomology by applying \cite[Th. 1]{km}. Over a general field $k$ of characteristic $>0$, we then reduce to a finite field by reducing first to $k$ finitely generated and then using smooth and proper base change.
\end{ex}

\subsection{Weak and Strong Lefschetz}
Assume that $\sV$ is closed under taking smooth hyperplane sections.

\begin{defn}\label{d5.2a} Let $(\sC,H)$ be a Weil cohomology with $(\sC, L_\sC)\in \Add^\otimes_*$. We say that
\begin{itemize}
\item $H$ \emph{verifies Weak Lefschetz}  if, for any connected $X\in \V$ of dimension $n$ and any smooth hyperplane section $i:Y\inj X$, connected with the same field of constants as $X$, the map $i^*:H^l(X)\to H^l(Y)$ is an isomorphism for $l\le n-2$.
\item $H$ \emph{verifies Strong Lefschetz} if, for $(X,Y)$ as above and $j\le n$, the morphism
\[L^{j}: H^{n-j}(X)\by{}H^{n+j}(X)(j),\]
induced by the Lefschetz operator $L:= L_Y$ (Definition \ref{lefdef}), is an isomorphism.
\end{itemize}
\end{defn}

\begin{rk} This induces isomorphisms
$$L^{n-2i}: \sC (\un, H^{2i}(X)(i))\iso \sC(\un, H^{2(n-i)}(X)(n-i)).$$
\end{rk}

\begin{rk} \label{rk:split}
The Weak Lefschetz property is usually stated with the additional condition: $i^*$  is injective for $l= n-1$. This is automatic in the presence of Strong Lefschetz, by Lemma \ref{l5.2}. This argument even gives split injectivity for all $l= n-j$.
\end{rk}

\begin{defn}\label{d5.2b}
A Weil cohomology $(\sC,H)$ is \emph{tight} if it is normalised (Definition \ref{d3.4}),  Albanese-invariant (Definition \ref{defalbinv}), and verifies 
Weak and Strong Lefschetz (Definition \ref{d5.2a}). 
\end{defn}

(For $\V =\Sm^\proj_k$ and $(\Vec_K,H)$ traditional we recover Kleiman's axiomatisation of a Weil cohomology in \cite{kst} plus Albanese-invariance.) 

\begin{ex}\label{exwsa} (see also Example \ref{p5.3}.) All classical Weil cohomologies  (and the abelian enrichments in Example \ref{exw}) in the sense of Definition \ref{d3.2} are tight. \end{ex}

\begin{lemma}\label{l8.1}
If $(\sC,H)$ is tight and $(F,u):(\sC,L_\sC)\to (\sC',L_\sD')\in \Add^\otimes_*$, the push-forward $H'= F_*H$ is also tight.  Additionally, for $(F,u)\in \Ex^\rig_*$ with $F$ faithful we have that $H$ is tight if and only if $H'$ is tight. 
\end{lemma}

\begin{proof} The first fact is true because the tightness properties are given by isomorphisms, which are preserved by $F$.  In the second case, if $F$ faithful and exact then it is conservative. 
\end{proof}

\subsection{Another representability theorem}

\begin{thm} \label{thm2m} Assume that $\sV$ is closed under taking smooth hyperplane sections. 
For $(\sC,L_\sC)\in \Add^\otimes_*$, let $\Weil^+(k,\V;\sC,L_\sC)$ be the $1$-full and $2$-full sub-$2$-category of $\Weil(k,\V;\sC,L_\sC)$ consisting of those $H$ which are tight. Then the $2$-functor $\Weil^+(k,\V;-)$ is strongly $2$-representable, as well as its restriction to $\Ex^\rig_*$.
\end{thm}

\begin{proof} We obtain a representing object by localising $\weil(k,\V)$ with respect to the morphisms $i^*$, $a^*_X$ and $L^j$ given by any $(X,Y)$ as in Definitions \ref{d5.2a} - \ref{defalbinv}, in the style of Proposition \ref{s2.4} (\cf the proof of Theorem \ref{thm2}).  For its restriction to $\Ex^\rig_*$, we compose with $T$ as in the proof of Corollary \ref{cor:thm2}.
\end{proof}

\begin{nota} \label{n:wsa}
We denote by $(\weil^+, W^+)$ and $(\weil_\ab^+, W_\ab^+)$ the representing objects in Theorem \ref{thm2m}; by its proof, we have $\weil_\ab^+=T(\weil^+)$. 
\end{nota}

\begin{rk}\label{rkloc}
The abelian $\otimes$-category $\weil_\ab^+$ is a $\otimes$-Serre localisation of $\weil_\ab$. In fact, let $\sI^+$ be the $\otimes$-Serre ideal of $\weil_\ab$ generated by kernels and cokernels of the morphisms in the proof of Theorem \ref{thm2m} under $\lambda_\weil$. Then $\weil_\ab^+\simeq \weil_\ab/\sI^+$. Thus, in Theorem \ref{thm:AS}, for $\sS$ the class of tight Weil cohomologies we get that $\sI_\sS=\sI^+$ and $\sA_S =\weil_\ab^+$. 
\end{rk}

\begin{warning}\label{warn} By Lemma \ref{l4.7} and Proposition \ref{p1.1}, we have $\weil(k,\sV)^\natural\allowbreak=\weil(k,\sV^\sat)^\natural$ where $\sV^\sat$ is the saturation of $\sV$. The same is not true (or at least not clear) for $\weil^+$, because it is not clear whether a Weil cohomology on $\sV^\sat$ whose restriction to $\sV$ is tight, is also tight on $\sV^\sat$. Note also that $\weil^+(k,\sV)$ is not defined if $\sV$ is not closed under hyperplane sections, because one cannot formulate weak Lefschetz.
\end{warning}

For tight Weil cohomologies we set the same framework of Definitions \ref{relh} and \ref{relhab}. The analogue of Theorem \ref{thm2bis} is the following:

\begin{thm}\label{thm2tris}
a) Any tight Weil cohomology $(\sC,H)$ has an initial tight enrichment $(\weil_H^+,W_H^+)$ such that $W_H^+= (\epsilon_H)_*W_H$ is the push-forward along a faithful $\otimes$-functor $\epsilon_H:\weil_H\to \weil_H^+$. If $H$ is abelian-valued, then the ab-initial enrichment $(\weil_H^\ab,W_H^\ab)$ is tight; in this case, there is a canonical faithful $\otimes$-functor $\iota_H^+$ fitting in the following commutative diagram
\begin{equation}\label{eq+ab}
\begin{gathered}
\xymatrix{
\weil_H\ar[r]^{\iota_H} \ar[d]_{\epsilon_H}&\weil_H^\ab\\
\weil_H^+\ar[ru]_{\iota_H^+}&
}
\end{gathered}
\end{equation}
and $W_H^\ab=(\iota_H^+)_*W_H^+$.\\
\hspace{0.7cm}
b) There is a $1-1$ correspondence between initial tight Weil cohomologies and $\otimes$-ideals of $\weil^+$ (resp. Serre $\otimes$-ideals of $T(\weil^+)$). 

Moreover, the target of any initial or ab-initial tight Weil cohomology is rigid.\\
c) The category $\weil_H^+$ is a $\otimes$-quotient of $\weil^+$, and the category $\weil_H^\ab$ is a $\otimes$-Serre localisation of $\weil_\ab^+$.
\end{thm}

\begin{proof} The proof of Theorem \ref{thm2bis} applies \emph{verbatim} to $\weil^+$ and $\weil_\ab^+$ appealing to Theorem \ref{thm2m} and observing that $(\weil_H^\ab,W_H^\ab)$ is tight by Lemma \ref{l8.1}. 
\end{proof}

\begin{ex} By Example \ref{exwsa}, Theorem \ref{thm2tris} applies to classical Weil cohomologies.  
\end{ex}

\begin{rk} Let $(\sC^\natural, H)$ be  tight and assume Condition C. Let $\sM_H^\Theta:= (\sM_H[\Theta_{\oplus,\otimes}^{-1}])^\natural$ be the pseudo-abelian completion of the localisation of $\sM_H$ at the set $\Theta:=\{L^{n-i}: h^{i}_H(X) \to h^{2n-i}_H(X)(n-i)\ \text{for all}\ X\in\V\ \text{and}\ i\leq n =\dim(X)\}$ (as in Proposition \ref{s2.4}). We have that $\sM_H^\Theta\iso (\weil_H^+)^\natural$. Actually, the push-forward of $h_H$ to $\sM_H^\Theta$ defines a tight enrichment of $H$ which is universal by construction.  
\end{rk}

\subsection{Some Lefschetz algebra}\label{s8.5} Here we verify that the operators and identities of \cite{kdix} and \cite{kst} continue to make sense and hold for any Weil cohomology verifying Strong Lefschetz.

Let $(\sC,H)$ be  tight. 
For $X\in\V$ and $\dim (X) =n$, provided with a polarisation as in Remark \ref{r8.1}, define the operator $\Lambda$ as usual \cite[1.4.2.1]{kdix}, \cite[\S 4]{kst} on $H^i(X)$ with $i\leq n$ by the following commutative diagram 
\begin{equation}\label{eq8.1}
\begin{gathered}
\xymatrix{
H^i(X)\ar[r]^{L^{n-i}\hspace{0.5cm}}_{\sim\hspace{0.5cm}} \ar@{.>}[d]_{\Lambda}&H^{2n-i}(X)(n-i)\ar[d]^{L}\\
H^{i-2}(X)(-1)\ar[r]^{L^{n-i+2}\hspace{0.7cm}}_{\sim\hspace{0.6cm}}&H^{2n-i+2}(X)(n-i +1)}
\end{gathered}
\end{equation}
and on $H^{2n-i+2}(X)$ by the following one
\begin{equation}\label{eq8.2}
\begin{gathered}
\xymatrix{
H^i(X)(i-n-1)\ar[r]^{L^{n-i}}_{\sim\hspace{0.1cm}} &H^{2n-i}(X)(-1)\\
H^{i-2}(X)(i-n-2)\ar[u]^{L}\ar[r]^{\hspace{0.5cm}L^{n-i+2}}_{\hspace{0.3cm}\sim}&H^{2n-i+2}(X)\ar@{.>}[u]_{\Lambda}
}
\end{gathered}
\end{equation}
so that $\Lambda L=1$ in \eqref{eq8.1} and $L\Lambda =1$ in \eqref{eq8.2}.

The primitive decomposition of $H(X)$ then carries over in $\sC^\natural$. Namely, for $i\le n$ we define $P^i(X,L)$ as the image of $1-L\Lambda$ in 
$H^i(X)$ (primitive classes) and $p^i$ as the projector of $H(X)$ with image $P^i(X,L)$; the decomposition $H^i(X)\simeq  P^i(X,L)\oplus H^{i-2}(X)(-1)$ yields inductively a decomposition
\begin{equation}\label{eq8.3}
H^i(X)\simeq \bigoplus_{j=0}^{[i/2]} P^{i-2j}(X,L)(-j).
\end{equation}

For $i\ge n$, the isomorphism $L^{i-n}$ of \eqref{eq8.2} yields a similar decomposition
\begin{equation}\label{eq8.4}
H^i(X)\simeq \bigoplus_{j=i-n}^{[(i-n)/2]} P^{i-2j}(X,L)(-j)
\end{equation}
and we define $p^i$ as the projector of $H(X)$ onto $P^{2n-i}(X,L)(n-i)$. For \eqref{eq8.3} (\resp \eqref{eq8.4}) and modulo twists, $L$ acts like inclusion (\resp projection) and $\Lambda$ acts like projection (\resp inclusion).

We then get the additional operators $h$ (of degree $0$) and  ${ }^c\Lambda$, where $h =\sum_{i=0}^{2n}(i-n)\pi^i$ and, according to \eqref{eq8.3} and \eqref{eq8.4}, ${ }^c\Lambda$ is multiplication by $j(n-i+j+1)$ on $P^{i-2j}(X,L)(-j)$. We also get the Lefeschetz and Hodge involutions $\star_L$ and $\star_H$ as in \cite[1.1]{A}, exchanging the decompositions \eqref{eq8.3} for $H^i(X)$ and $\eqref{eq8.4}$ for $H^{2n-i}(X)$ with signs and multiplicities; then  \cite[1.2]{A} holds verbatim, \ie one has the $\mathfrak{sl}_2$-triple identities 
\begin{equation}\label{eq8.5}
[h,{ }^c\Lambda]=2\ { }^c\Lambda, \quad [h,L]=- 2L,\quad [L, { }^c\Lambda] = h
 \end{equation}
  which define a representation of $\mathfrak{sl}_2$ on $H(X)$ sending $\left(\begin{smallmatrix} 0&-1\\1&0\end{smallmatrix}\right)$ to $\star_H$, up to signs. Also, $\Lambda = \star_L L\star_L=\star_H L \star_H$ is easy to check (see  \cite[Prop. 1.4.3 and Lemma 1.4.6]{kdix} or \cite[pp. 13-14]{kst}).

In order to avoid the ``heresy'' of neglecting the Tate twists, we propose the following formalism (see also \cite[Rem. 1.10]{criteres}): in the ind-category $\ind\ \sC =\ind\ \sC^\natural$, consider within $\End(\bigoplus_{i,j} H^i(X)(j))$ the sub-graded algebra $R$ generated by the homogeneous operators of bidegrees $(2i,i)$. Then $R$ is finite-dimensional if $\sC=\Vec_K$ for some field $K$, and all above operators belong to $R$ in general.

\begin{prop}\label{p8.1} Let $S$ be the (graded) subalgebra of $R$ generated by $L$ and $\Lambda$.  Then $S$ contains the operators ${}^c\Lambda$, $\star_L$, $\star_H$, $\pi^i$, and is also generated by $L$ and ${}^c\Lambda$. \\
Moreover,  ${ }^c\Lambda$ is the only operator of degree $(-2,-1)$ verifying the third identity of \eqref{eq8.5}.
\end{prop}

\begin{proof} If $H$ is traditional, this follows from \cite[Prop. 1.4.3 and 1.4.4]{kdix}; but the proofs of \loccit work in general. Indeed, Kleiman works with elements but, for example, the identities of his lemma 1.4.5 are readily checked in our case by restricting them to the direct summands of \eqref{eq8.3}. The last claim  is shown as in the proof of \cite[Prop. 1.4.6 (i)]{kdix} (see also \cite[\S 11, lemme 1]{bbki}).
\end{proof}

\begin{rk} For the sake of exposition, let us give two explicit proofs that $S$ contains the $\pi^i$'s. The first is by the following identity of \cite[Lemma 2.4]{kdix}:
\[\Lambda^{n-i}(1-\sum_{j>2n-i}\pi^j)L^{n-i}(1-\sum_{j<i}\pi^j) = \pi^i.\]

To check this formula in $\sC$, it suffices to show that the two sides agree after composing with $\pi^l$ on the right for every $l$. For $l\le i-1$ this is clear because $(1-\pi^0+… + \pi^{i-1})\pi^l= 0$.  For $l\ge i$ we have $(1-\pi^0+… + \pi^{i-1})\pi^l = \pi^l$ and $L^{n-i}$ carries $H^l$ into $H^{2n-2i+l}$, so the left hand side becomes $0$ if $l\ne i$ since $2n-2i+l> 2n-i$, while for $l=i$ this boils down to the identity $\Lambda^{n-i} L^{n-i}=1$ on $H^i$.

For the second argument,  formula \eqref{eq8.5} implies that $h\in S$. Since
\[h^l = \sum_{i=0}^{2n}(i-n)^l\pi^i\]
for all $l>0$, the Vandermonde theorem then implies that all $\pi^i$'s belong to $S$ for $i\ne n$; finally, $\pi^n=1-\sum_{i\ne n} \pi^i$. 
\end{rk}

The uniqueness of ${}^c\Lambda$ in Proposition \ref{p8.1} implies the identity
\begin{equation}\label{eq8.6}
{ }^c\Lambda_{X\times Z} = { }^c\Lambda_X\otimes 1+1\otimes { }^c\Lambda_z
\end{equation}
as in the proof of \cite[Prop. 1.4.6 (ii)]{kdix}, where the Lefechetz operator of the product $X\times Z$ comes from the Segre embedding associated to the respective polarisations of $X$ and $Z$. We also note the identity
\begin{equation}\label{eq8.7}
\Lambda_Y =i^*\Lambda_X^2i_*
\end{equation}
of \cite[proof of Prop. 2.12]{kdix} for a smooth hyperplane section $i:Y\inj X$, where $\Lambda_Y$ is relative to the same polarisation as for $X$; it can be checked using \eqref{eq8.1}, \eqref{eq8.2} and Lemma \ref{l5.2}.

\subsection{Lefschetz type conditions} \label{s8.6}
Keep the previous notation. Denote by 
\begin{equation}\label{eqcycle}
A^i_H(X) \subseteq \sC(\un, H^{2i}(X)(i))
\end{equation}
the image of cycle class map or, equivalently: it is isomorphic to $ CH^i(X)_\Q/\Ker \cl_H^i$.

\begin{defn}\label{d1.8b} Let $(\sC,H)$ be  tight. We say that $(H,X,L)$ with $\dim (X) =n$\\
1)   \emph{verifies Condition A} if the restriction $$L^{n-2i}: A^i_H(X)\to A^{n-i}_H(X)$$ is an isomorphism for all $i\geq 0$;\\ 
2)  \emph{verifies Condition B} if the operator $\Lambda$ is in the image of $H$, \eg for $i\leq n$ is the class of a  correspondence (of degree $-1$) in $A^{n-1}_H(X\times X)$.\\
The choice of $H$ being implicit, we abbreviate to $A(X,L)$, $B(X,L)$ 
as usual.
\end{defn}

\begin{rk} In \cite[Th. 2A9 4.]{kdix}, Kleiman shows that $H$, traditional and assumed to verify Strong Lefschetz  as in Definition \ref{d5.2a}, is Albanese-invariant (Definition \ref{defalbinv}) if it satisfies Condition B of Definition \ref{d1.8b} 2). 
\end{rk}

\begin{thm}\label{t8.1}
For a tight Weil cohomology $H$ and any polarised $X\in \sV$, we have:
\begin{enumerate}
\item $B(X,L)\Rightarrow A(X,L)$, $B(X,L)\Rightarrow C(X)$ and $A(X\times X,L\otimes 1+ 1\otimes L) \Rightarrow B(X,L)$.
\item  $D(X\times X) \Rightarrow B(X,L)$,
\item Suppose that $\uH$ is an enrichment of a traditional tight Weil cohomology $\uH'$. Then $B(X,L) \iff B(X,L')$ for any   $L'$ coming from another polarisation.
\end{enumerate}
\end{thm}

\begin{proof} Replacing $\sC$ by $\sC^\natural$, we may assume $\sC$ pseudo-abelian and therefore have the primitive decompositions.

In (1), the first implication  is trivial, the second follows from Proposition \ref{p8.1} and the third is proven exactly as in \cite[proof of Theorem 4.1 (1)]{kst}. 

For (2), we cannot reason as in \cite{kdix} or \cite{kst} where Kleiman uses the Cayley-Hamilton theorem, which is not available for a general Weil cohomology. Instead, we use Smirnov's argument in \cite{sm}: by  Proposition \ref{p8.1}, it suffices to show that ${}^c\Lambda$ is algebraic, which follows from \cite[Th. 1]{sm} using Jannsen's semi-simplicity theorem \cite{jannsen3}.

For (3), we  observe that $B(X,L)$ for $\uH$ is equivalent to $B(X,L)$ for $\uH'$; this reduces the statement to \cite[Cor. 2.11]{kdix}.
\end{proof}

As in \cite{kdix} and \cite{kst}, we deduce from \eqref{eq8.6}, \eqref{eq8.7} and Proposition \ref{p8.1}:

\begin{prop}\label{p8.2} Let $X\in \sV$ be provided with a polarisation yielding a Lefschetz operator $L_X$.\\ 
a) Let $Y\subset X$ be a smooth hyperplane section, $L_Y$ be the induced Lefschetz operator. Then $B(X,L_X)\Rightarrow B(Y,L_Y)$.\\
b) Let $(Z,L_Z)$ be another polarised variety (with $Z\in \sV$). Then $B(X,L_X)\allowbreak+B(Y,L_Y)\Rightarrow B(X\times Y, L_X\otimes 1+1\otimes L_Y)$.
\end{prop}

From Lemma \ref{l8.1} (and Theorem \ref{t8.1} (1)) we also get (\cf Theorems \ref{wconjC} b) - \ref{wconjB} a)): 

\begin{cor} \label{c8.1b}
Let $(\sC, H)$ be tight and abelian-valued. If the category $\sM_H$ is abelian and  $\uH$ is exact, then Condition C $\iff$ Condition B for $H$.
\end{cor}

Condition V of Definition \ref{d1.5v} then implies all conditions (since it implies Condition D).

\begin{cor} \label{c8.1a}
Under Condition V, all Conditions A, B, C and D hold true for any tight Weil cohomology $(\sC, H)$ with 
$\sC\in\Ex^\rig$.
\end{cor}

\begin{rks} a)  \cite[Prop. 2.7]{kdix} implies that the hypothesis of (3) holds if $\uH$ is an enrichment of a traditional Weil cohomology. It also holds under the conditions of Corollary \ref{c8.1b}. \\
b)  As a special case of Corollary \ref{c8.1b} we recover the first part of \cite[theorem p. 44]{A}. (In \loccit the hypothesis that $\uH$ be exact is missing, but fortunately it is granted by Remark \ref{r7.1}.)
\end{rks}

The pattern of Theorem \ref{wconjC} a) is available under Condition B.

\begin{lemma}\label{hwsa}
If $(\sC, H)$ is tight and verifies Condition B then $h_H : \sM_\rat \to \sM_H$  defines a tight Weil cohomology.
\end{lemma}

\begin{proof} Theorem \ref{t8.1} gives that $H$ verifies Condition C hence $h_H : \sM_\rat \to \sM_H$ defines a Weil cohomology (see discussion between Lemma \ref{l4.5} and Remark \ref{r7.1}). Let $l\le n$; the algebraic cycles giving $\Lambda^{n-l}$ are inverses of $L^{n-l}: h^{l}_H(X) \to h^{2n-l}_H(X)(n-i)$.  
Also, $i^*: h^{l}_H(X)\inj h^{l}_H(Y)$ is split injective for $i:Y\inj X$ a smooth hyperplane section of $X$ by Remark \ref{rk:split}; since  the functor $\uH$ is additive and faithful and $i^*: H^{l}(X)=\uH(h^{l}_H(X))\iso H^{l}(Y)=\uH(h^{l}_H(Y))$ is an isomorphism for $l\le n-2$, the complementary summand is $0$ in this case and we get Weak Lefschetz for $h_H$ as well. For the normalised and Albanese properties, we use Propositions \ref{p5.1} and \ref{c5.1} similarly.
\end{proof}

\begin{thm}\label{wconjC+} Let $(\sC, H)$ be  tight and let $(\weil_H^+,W_H^+)$ be as in Theorem \ref{thm2tris}. 
The following conditions are equivalent:
\begin{thlist}
\item $H$ verifies Condition B;
\item the functor $\sM_H\to \weil_H^\natural$ induced by $W_H$ and the functor $\epsilon_H^\natural: \weil_H^\natural\to (\weil_H^+)^\natural$ induced by $\epsilon_H$ in \eqref{eq+ab} are equivalences;
\item the composition  $\sM_H\to (\weil_H^+)^\natural$ of the two functors of (ii) is an equivalence of categories.
\end{thlist}
If this is true, then $Z(\weil_H^+)=\Q$.
\end{thm}
\begin{proof}
(i) $\Rightarrow$ (ii): the first functor is an equivalence by Theorem \ref{wconjC} a) because B $\Rightarrow$ C by Theorem \ref{t8.1} (1), and $h_H$ defines a tight  Weil cohomology by Lemma \ref{hwsa} whence $\epsilon_H^\natural$ is an equivalence by the universal property of $(\weil_H^+,W_H^+)$. (ii) $\Rightarrow$ (iii) is trivial and (iii) $\Rightarrow$ (i) is clear.
\end{proof}

 For tight Weil cohomologies with abelian target we obtain the commutative diagram \eqref{eq5.5} and the following analogue of Theorems \ref{wconjC} b), \ref{t5.1} and \ref{t5.2}:

\begin{thm}\label{wconjB} Let $(\sC, H)$ be tight, $\sC\in\Ex^\rig$ and let $(\weil_H^{\ab},W_H^{\ab})$ be as in Theorems \ref{thm2bis} - \ref{thm2tris}.\\
a)  The following are equivalent:
\begin{thlist}
\item $H$ verifies Condition B, the category $\sM_H$ is abelian  and $\uH:\sM_H\to\sC$ is exact;
\item $H$ verifies Condition B, the category $(\weil_H^+)^\natural$ is abelian and the functor 
$\iota_H^{+,\natural}: (\weil_H^+)^\natural\to \weil_H^\ab$ is exact; 
\item the functors $\sM_H\to (\weil_H^+)^\natural$ and $\iota_H^{+,\natural}: (\weil_H^+)^\natural\to \weil_H^\ab$ are equivalences;
\item the (faithful) functor $w_H :\sM_H\to \weil_H^\ab$ is an equivalence of categories.
\end{thlist}
These conditions imply $Z(\weil_H^\ab)=\Q$.\\
b) The following are equivalent:
\begin{thlist}\stepcounter{spec}\stepcounter{spec}\stepcounter{spec}\stepcounter{spec}
\item $H$ verifies Condition D;
\item the (faithful) functor $w_H :\sM_H\to \weil_H^\ab$ is an equivalence and $\weil_H^\ab$ is semi-simple;
\item the (exact) functor $\bar w_H: T(\sM_H)\to \weil_H^\ab$ is faithful and $\weil_H^\ab$ is connected; 
\item all functors in \eqref{eq5.5} are equivalences.
\end{thlist}
Moreover these conditions imply those of a).
\end{thm}

\begin{proof} 
a) (i) $\Rightarrow$ (ii): from Theorem \ref{wconjC+} we get that $(\weil_H^+)^\natural$ is abelian; moreover, this theorem says that $\epsilon_H^\natural$ is an equivalence, so it suffices to show that $\iota_H^\natural$ is exact, where $\iota_H$ is as in  \eqref{eq+ab}. But Condition C holds true by Theorem \ref{t8.1} (1), thus what we want follows from (iii) $\Rightarrow$ (iv) in Theorem \ref{wconjC} b). 

Now (ii) $\Rightarrow$ (iii) follows from  Theorem \ref{wconjC+} and the universal property of $(\weil_H^{\ab},W_H^{\ab})$; (iii) $\Rightarrow$ (iv) is clear. Finally, (iv) $\Rightarrow$ (i) is trivial.

b)  First, recall that  D $\Rightarrow$ B $\Rightarrow$  C by Theorem \ref{t8.1},  since $H$ is tight. This being said,

(v) $\Rightarrow$ (i)  because, under D, $\sM_H$ is abelian semi-simple and, similarly, (v) + (iv) $\Rightarrow$ (vi). Thus we get (v) $\Rightarrow$ (vi),  and  (vi) $\Rightarrow$ (vii) is clear since $\lambda : \sM_H \iso T(\sM_H)$ by semisimplicity Lemma \ref{bvk}. (vii) $\Rightarrow$ (viii): since $Z(T(\sM_H))$ is an absolutely flat domain (Lemma \ref{l2.1} d), it is a field which implies Condition D by the implication (iii) $\Rightarrow$ (i)  of Theorem \ref{t5.1}, hence (viii) by the implication (v) $\Rightarrow$ (ix) of the same theorem. Finally, (viii) $\Rightarrow$ (v) is obvious.
\end{proof}

In the universal case, this gives:

\begin{cor} \label{c7.1a}
Under Condition D for $W^+_\ab$,  all functors in \eqref{eq5.5} are equivalences and these equivalences identify $W_\ab^+$ with $h$. 
\end{cor}

\begin{proof} This follows from  Theorem \ref{wconjB} b) and Theorem \ref{t8.1}.
\end{proof}

\subsection{Hodge type condition}\label{s8.7}

Let $H$ be tight, and let $X\in \sV$ be of dimension $n$, provided with a polarisation, with Lefschetz operator $L$.  For $i\le n/2$, we define 
$$A^{i}_{H,P}(X) =   \sC (\un, P^{2i}(X))\cap A^{i}_H(X) \subseteq \sC(\un, H^{2i}(X)(i))$$
where $A^{i}_H(X)$ is the image of the cycle class map $c\ell^i_H$ as in \eqref{eqcycle}, using the decomposition 
$$H^{2i}(X)(i)\simeq  P^{2i}(X)(i)\oplus H^{2i-2}(X)(i-1).$$ 

By Axiom (vi) of Definition \ref{d1.1}, the restriction of the cup-product pairing  $(x, y)\mapsto (-1)^i <L^{2n-i}x \cdot y>$ to $A^{i}_{H,P}(X)$ is $\Q$-valued.

\begin{defn} $H$ satisfies Condition $Hdg(X,L,i)$ if this quadratic form is positive definite.
\end{defn}

\begin{lemma} Let $(\sC,H)$ be tight,  $(F,u):(\sC,L_\sC)\to (\sC',L_\sD')\in \Add^\otimes_*$ be a $1$-morphism and  $H'= F_*H$ be the push-forward of $H$ by $F$ as usual (recall that $H'$ is also tight by Lemma \ref{l8.1}). If $F$ is faithful, the homomorphisms $A^{i}_{H,P}(X)\to A^{i}_{H',P}(X)$ are bijective, and $H$ satisfies Condition $Hdg(X,L,i)$ if and only if $H'$ does.
\end{lemma}

\begin{proof} Clearly, $A^{i}_{H}(X)\to A^{i}_{H'}(X)$ is bijective, and the homomorphism of the lemma is a direct summand of this one.
\end{proof}

\begin{thm}\label{t8.3} Let $(\sC, H)$ be tight, and let $X\in\sV$ be polarised. Then $A(X,L,i)$ \& $Hdg(X,L,i)$ for all $i$ $\Rightarrow$ $D(X)$.
\end{thm}

\begin{proof} Same as in the proof of \cite[Prop 3.8]{kdix}.
\end{proof}

From this theorem and \ref{t8.1} (1) and (2), we get as usual:

\begin{cor} Assume that $H$ is tight and verifies Condition $Hdg$. Then Condition A $\iff$ Condition D.\qed
\end{cor}

In characteristic zero, Condition Hdg is known for Hodge cohomology, hence for any classical cohomology, \cf \cite[beg. of §5]{kdix}.

\subsection{Fullness conditions}\label{s8.8}

\begin{defn} Let $(\sC,H)$ be a Weil cohomology, $X\in \sV$ and $i\ge 0$. We say that $X$ verifies Condition F in codimension $i$ with respect to $H$ (in short: $F(X,i)$) if the map
\[CH^i(X)\otimes Z(\sC)\to \sC(\un,H^{2i}(X)(i)) 
\]
is surjective.
\end{defn}

This definition encompasses all the fullness conjectures of \cite[Ch. 7]{A}. It is not very useful \emph{per se}, so we shall use it only in special cases.

\begin{prop}\label{p8.4} Assume that $H$ is tight and that $Z(\sC)=\Q$. Let $X\in \sV$ be polarised of dimension $n$, with Lefschetz operator $L$, and $i\le n/2$. Then $F(X,i)$ $\Rightarrow$ $A(X,L,i)$ + $F(X,n-i)$. If $H$ satisfies $F(X,i)$ and $Hdg(X,L,i)$ for all $i$, it satisfies $D(X)$. 
\end{prop}

\begin{proof} It is classical: in the commutative diagram
\[\begin{CD}
A^i_H(X)@>L^{n-2i}>> A^{n-i}_H(X)\\
@V\cl^i VV @V\cl^{n-i} VV\\
\sC(\un,H^{2i}(X)(i))@>\cl(L^{n-2i})>> \sC(\un,H^{2(n-i)}(X)(n-i))
\end{CD}\]
$\cl(L^{n-2i})$ is bijective by Strong Lefschetz while $\cl^i$ and $\cl^{n-i}$ are injective, hence $L^{n-2i}$ is injective. If $\cl^i$ is surjective, all maps in the diagram are bijective. The last statement follows from Theorem \ref{t8.3}.
\end{proof}

The next theorem mimicks Tate's yoga in \cite[§2]{tate}, and generalises it. For $X\in \sV$ of dimension $n$ and $i\le n/2$, consider the Poincar\'e pairing
\begin{equation}\label{eqpoin}
\sC(\un,H^{2i}(X)(i))\times \sC(\un,H^{2(n-i)}(X)(n-i))\to \sC(\un,\un)=Z(\sC)
\end{equation}
and the map
\begin{equation}\label{eqpoin1}
A^i_H(X)\otimes Z(\sC)\to \sC(\un,H^{2i}(X)(i))
\end{equation}
induced by the cycle class map. Note that, by Poincar\'e duality, \eqref{eqpoin} amounts to the composition pairing
\begin{equation}\label{eqpoin2}
\sC(\un,H^{2i}(X)(i))\times \sC(H^{2i}(X)(i),\un)\to Z(\sC).
\end{equation}

\begin{thm}\label{t8.4} Assume that $H$ is tight that $K=Z(\sC)$ is a field and that Hom groups in $\sC$ are finite dimensional over $K$  (for example $H$ pseudo-tannakian as in Definition \ref{d6.1}.) Let $S(X,i)$ (\resp $I(X,i)$) denote the condition that \eqref{eqpoin} is non-degenerate (\resp \eqref{eqpoin1} is injective). Then we have
\[F(X,i)+D(X,i)\Rightarrow F(X,n-i)+S(X,i) \Rightarrow D(X,i) \Rightarrow I(X,i).\]
\end{thm}

\begin{rk}\label{r8.2} In the case of the Tate conjecture ($H=$ $\ell$-adic cohomology), $I(X,i)$ is the same as $I^i(X)$ in \cite[p. 72]{tate}. $S(X,i)$ is also the same as $S^i(X)$ in \loccit Indeed,  with Tate's notation $\sC(\un,H^{2i}(X)(i))$ identifies with $V^i(X)^G$, the dual of $\sC(H^{2i}(X)(i),\un)$ identifies with $V^i(X)_G$, and Tate's map $V^i(X)^G\to V^i(X)_G$ coincides with the map induced by \eqref{eqpoin2}. 
\end{rk}

\begin{proof} We copy the one of Lemma 2.5 and Proposition 2.6 in \cite{tate}. For this, we first copy its diagram (2.3) (where we drop $X$ as in \emph{loc. cit.}):
\begin{equation}\Small
\begin{CD}
K\otimes A^i_H@>b>> \sC(\un,H^{2i}(i))@>c>> \sC(\un,H^{2(n-i)}(n-i))^*\\
@Va VV&& @Vd VV\\
K\otimes (A^i_H/N)@>f>> K\otimes \Hom_\Q(A^{n-i}_H,\Q)@>e>> (K\otimes A^{n-i}_H)^*.
\end{CD}
\end{equation}

Here, $(-)^*$ means $K$-dual and $N$ denotes numerically trivial cycles. This diagram is commutative, and the arrows $e$ and $f$ are injective. We now have the same implications as in \cite[(2.4)]{tate}:

\begin{itemize}
\item $D(X,i)\iff$ $a$ is injective;
\item $F(X,i)\iff$ $b$ is surjective;
\item $I(X,i)\iff$ $b$ is injective;
\item $S(X,i)\iff$ $c$ is bijective $\iff$ $c$ is injective $\iff$ $c$ is surjective;
\item $F(X,n-i)\iff$ $d$ is injective.
\end{itemize}

This proves the implications of Theorem \ref{t8.4}.
\end{proof}

We also have the following result of Andr\'e \cite[Prop. 7.1.1.1]{A}:

\begin{thm}\label{tandre} Assume that $H$ is pseudo-tannakian as in Defnition \ref{d6.1}, and that $Z(\sC)$ is a field (\eg $H$ is abelian-valued). Under Condition $F$, Condition $D$ for $H$ is equivalent to the semi-simplicity of $H^i(X)$ for all $X\in \sV$ and all $i\ge 0$.\qed
\end{thm}

Note that, as in \cite{tate}, $S(X,i)$ in Theorem \ref{t8.4} follows from the semi-simplicity of $H^{2i}(X)$ ($SS^i(X)$ in \cite{tate}): more generally, if $C\in \sC$ is semi-simple,  the pairing $\sC(\un,C)\times \sC(C,\un)\to K$ is non-degenerate as one sees by writing  $C\simeq r\un \oplus C'$ with $\sC(\un,C')=\sC(C',\un)=0$. As a by-product of this reformulation of the yoga in \cite{tate}, we get the following converse, which links Theorems \ref{t8.4} and \ref{tandre} and generalises \cite[Th. 6]{ss} (case of $\ell$-adic cohomology) from a finite field to any finitely generated base field:

\begin{thm} \label{t8.5} Under the hypotheses of Theorem \ref{tandre}, let $X\in \sV$ be of dimension $n$. Then $S(X\times X,n)$ implies the semi-simplicity of all $H^i(X)$.
\end{thm}

\begin{proof} For $C,D\in \sC$, the restriction of the composition pairing 
\[\sC(\un,C\oplus D)\times \sC(C\oplus D,\un)\to K
\]
to $\sC(\un,C)\times \sC(D,\un)$ and $\sC(\un,D)\times \sC(C,\un)$ is $0$. By the Künneth formula, $S(X\times X,n)$ is therefore equivalent to the non-degeneracy of the composition pairing
\[\sC(\un,H^{i}(X)\otimes H^{2n-i}(X)(n))\times \sC(H^{i}(X)\otimes H^{2n-i}(X)(n),\un)\to K
\]
for all $i$. By Poincar\'e duality, this pairing is converted into
\[\End_\sC(H^{i}(X))\times \End_\sC(H^{i}(X))\to K
\]
which is checked to be
\[(f,g)\mapsto \tr(g\circ f)\]
where $\tr$ is the rigid trace. This means that the homomorphism
\[\End_\sC(H^{i}(X))\to \End_{\sC/\sN}(H^{i}(X))\]
is bijective. But $\sC/\sN$ is semi-simple by Lemma \ref{lak}, which concludes the proof.
\end{proof}

\section{Examples and consequences}

\subsection{Summary}\label{s9.1}

 In  Definition \ref{defhom}, we introduced the adequate equivalence relation $\sim_\hun$ given by $W_\ab$-homological motives. Consider the similar relation $\sim_\hum$ given by $W^+_\ab$-homological motives (Notation \ref{n:wsa}); it is coarser than $\sim_\hun$ and we have a commutative diagram
\begin{equation}\label{Gpic}
\begin{gathered}
\xymatrix{
\sM_{\hun}\ar[r]^{\lambda_\hun}\ar[d]&T(\sM_{\hun})\ar[d]^{ }\ar[dr]^{\bar w}& \\
\sM_{\hum}\ar[r]^{\lambda_\hum}\ar[d]&T(\sM_{\hum})\ar[dr]^{\bar w^+}\ar[dl]&\sW_\ab\ar[d]^\tau\\
\sM_\num &&\sW_\ab^+
}
\end{gathered}
\end{equation}
where $\bar w^+$ is induced by $W^+_\ab$. Here are some implications between conjectures:

\begin{itemize}
\item  By Proposition \ref{p7.1}, Condition V implies Condition D for $W_\ab$.
\item Obviously, Condition D for $W_\ab$ implies Condition D for $W_\ab^+$.
\item Condition D for $W_\ab$ (\resp for $W_\ab^+$) implies that $\lambda_\hun$  (\resp $\lambda_\hum$) is an equivalence (by Lemma \ref{bvk}).
\item By Corollary \ref{c7.1a},  Condition D for $W_\ab^+$ implies that $\bar w^+$ is also an equivalence, and that $\weil_\ab^+$ is connected.  By Corollary \ref{c7.1}, the same holds for $W_\ab$, $\bar w$ and $\weil_\ab$ under the extra Condition C, which is not automatic in this case.
\item If $\weil_\ab$ (\resp $\weil_\ab^+$) is connected, all abelian-valued (\resp tight abelian-valued) Weil cohomologies are ab-equivalent in the sense of Definition \ref{dcomp.1} a) by Lemma \ref{l2.1} a).
\item Since $\tau$ is a localisation (Remark \ref{rkloc}), $\weil_\ab$ connected $\Rightarrow$ $\tau$ is an equivalence (same lemma).
\end{itemize}

So, under Condition V,  $\weil_\ab$ is the only category left in the diagram which may be different from $\sM_\num$, and they are equivalent if and only if $\weil_\ab$ is connected. 

The hypothesis that $\weil_\ab^+$ is connected is \emph{a priori} weaker than Condition V, and has a rather striking consequence; it may be worth studying for itself. 

Note that, in the absence of Condition D for $W_\ab^+$, \eqref{Gpic} does not give any functor from $\weil_\ab^+$ to $\sM_\num$.

\subsection{Motives of abelian type}(Compare \protect{\cite[\S 2A, esp. Th. 2A9]{kdix}}).  \label{s7.4}

If $X$ is an abelian variety of dimension $g$, one has a (full) Chow-Künneth decomposition
\begin{equation}\label{eq9.1}
h(X)=\bigoplus_{i=0}^{2g} h^i(X)
\end{equation}
with
\begin{equation}\label{eq4.3}
h^i(X)\simeq S^i(h^1(X))
\end{equation}
(\cite{den-mur}, \cite[Th. 5.2]{scholl})\footnote{Recall that the exterior powers appearing in \cite[Th. 5.2 (ii)]{scholl} should be replaced by symmetric powers, to respect signs in the commutativity constraint of $\sM_\rat(k)$.}. Also,

\begin{lemma}[\protect{\cite{den-mur}}]\label{ldm} Multiplication by $n\in\Z$ on $X$ induces multiplication by $n^i$ on $h^i(X)$.
\end{lemma}

We also have Künnemann's isomorphisms \cite{ku}:
\begin{equation}\label{eq4.3a}
h^i(X)\iso h^{2g-i}(X)(g-i)
\end{equation}
induced by any polarisation given by a very ample symmetric divisor.

Therefore we seem to have obtained a tight Weil cohomology on abelian varieties, with values in Chow motives. The catch is that \eqref{eq4.3} is natural for homomorphisms of abelian varieties, but not for general correspondences. We shall now introduce the coarsest adequate equivalence relation which corrects this problem, on a larger class of varieties introduced in Notation \ref{n5.1} below.

\begin{lemma}\label{lsat}
The classes $\V_c$ and $\V_e$ of Examples \ref{e3.1} c) and e) have the same saturation.
\end{lemma}

\begin{proof} First note that $\sM_\rat(k,\sV_c)$ contains the Lefschetz motive $\L$ because $\L=h^2(E)$ for any elliptic curve $E$.\\
a) Let $C$ be a curve. Then $h^0(C)$ and $h^2(C)$ obviously belong to $\sM_\rat(k,\V_c)$, and so does $h^1(C)$ by \eqref{eq5.2}. Therefore $\V_e^\sat\subseteq \V_c^\sat$.

b) Let $A$ be an abelian variety, and $C\subset A$ be an ample curve passing through $0$ (hence geometrically connected). Then $h^1(A)$ is a direct summand of $h^1(C)$ \cite[Lemma 2.3]{murre}, hence belongs to $\sM_\rat(k,\V_e)$, and so do the other $h^i(A)$'s by \eqref{eq4.3}. Therefore $\V_c^\sat\subseteq \V_e^\sat$.
\end{proof}

\begin{nota}\label{n5.1}
 We write $\sAb$ for the common saturation of $\V_c$ and $\V_e$: these are \emph{varieties of abelian type}.
 \end{nota}

\begin{prop}\label{p5.2} If $X$ is an abelian variety, \eqref{eq5.7} holds for all $(i,j)$ for any Weil cohomology $H$. 
\end{prop}

\begin{proof} By the symmetric monoidality of $\uH^*$ in Proposition \ref{p1.1} a), one has
\[H^*(h^i(X))\simeq S^i(H^*(h^1(X)))\]
and one applies Proposition \ref{p5.1}. Indeed, the saturation of the category of abelian varieties contains all curves by Lemma \ref{lsat}. 
\end{proof}

\begin{cor}\label{c4.1} For $X$ an abelian variety, multiplication by $n\in \Z$ on $X$ induces multiplication by $n^i$ on $H^i(X)$.
\end{cor}

\begin{proof} This follows from Proposition \ref{p5.2} and Lemma \ref{ldm}.
\end{proof}

\begin{thm}\label{7.1} The universal Weil cohomology $W=W_{\sAb}$ verifies Condition C; hence $W$ induces an equivalence of $\otimes$-categories 
\[\fake{\sM}_W(k,\sAb)\iso \weil(k,\sAb)^\natural.\]
Its restriction to the subclass of abelian varieties is normalised, Albanese-invariant, and satisfies Hard Lefschetz (but Weak Lefschetz is not defined, see Warning \ref{warn}). It has weights in the sense of Definition \ref{d3.3}. 
\end{thm}

\begin{proof} By Proposition \ref{p5.2}, abelian varieties verify Condition C, and so do other members of $\V_c$ by Lemma \ref{l4.5}. This proves the first statement, the second follows from Theorem \ref{wconjC} and the third follows from \eqref{eq4.3a} and Proposition \ref{p5.2}.

For the statement about weights, consider for $i\ne j$ the bifunctor $(M,N)\mapsto \Hom(W^i(M),W^j(N))$ on ${\sM}_W(k,\sAb)^\op \times {\sM}_W(k,\sAb)$. It vanishes on abelian varieties, hence everywhere thanks to the identity
\[W^i(M\otimes \L)\simeq W^{i-2}(M)\otimes W^{2}(\L)=W^{i-2}(M)\otimes \L\]
\cf Proposition \ref{p1.1} a) (1).
\end{proof}

We now want to compute the homological equivalence defined by $W$.

\begin{defn}\label{d9.1} Let $A,B$ be two abelian varieties and let $\gamma\in \Corr(A,B)$ be a correspondence. For any integer $n\in\Z$, we write
\[ [n,\gamma] = n_B\gamma -\gamma n_A.\]
\end{defn}

\begin{lemma}\label{l4.4} a) if $\gamma$ is (the graph of) a homomorphism, then $[n,\gamma]=0$.\\
b) For any $\gamma$ and any $n$, $[n,\gamma]$ is homologically equivalent to $0$ with respect to any Weil cohomology.\\
c) We have identities
\begin{align*}
 [m,[n,\gamma]] &=[n,[m,\gamma]]\\
[n,\gamma\delta] &= [n,\gamma]\delta +\gamma[n,\delta],\quad [nm,\gamma] = m[n,\gamma] + [m,\gamma]n\\
[n,\gamma\otimes \delta]&= [n,\gamma]\otimes [n,\delta]+[n,\gamma]\otimes \delta n+\gamma n\otimes [n,\delta]
\end{align*}
for any $\gamma,\delta,m,n$.
\end{lemma}

\begin{proof} a) is trivial, b) follows from Corollary \ref{c4.1}. In c), the first identity follows from the Jacobi identity since $m$ and $n$ commute with each other; the second are easy and the third follows from the equality
\[n_{A\times B} = n_A\otimes 1_B\circ 1_A\otimes n_B=1_A\otimes n_B\circ n_A\otimes 1_B\]
for two abelian varieties $A,B$.
\end{proof}

\begin{defn} Let $\mult$ be the coarsest adequate equivalence relation in $\sM_\rat(k,\sAb)$ such that $[n,\gamma]\sim_\mult 0$ for all $A,B,n,\gamma$.
\end{defn}

\begin{thm}\label{t9.1} a) $\sim_\mult$ is also the adequate equivalence relation generated by morphisms $\gamma:h^i(A)\to h^j(B)$ for $i\ne j$ in $\sM_\rat(k,\sAb)$.\\
b) The Chow-K\"unneth decomposition of Deninger-Murre  defines a Weil cohomology 
\[h^*:\Corr(k,\sAb)\to \sM_\mult(k,\sAb)^{(\N)}. \]
This induces an equivalence of categories $\sM_\mult(k,\sAb)\iso \weil(k,\sAb)^\natural$.
\end{thm}

\begin{proof} Let $\gamma\in \Corr(k,\sAb)(A,B)$ and, for $i,j\ge 0$
\[\gamma^{i,j}: h^i(A)\by{\iota} h(A)\by{\gamma_*}h(B)\by{\pi}h^j(B) \]
where $\iota$ is the inclusion and $\pi$ the projection. We have
\[\gamma=\sum_{i,j} \gamma^{i,j}\]
and
\[[n,\gamma^{i,j}] =(n^j-n^i)\gamma^{i,j}\]
by Lemma \ref{ldm}. This proves a). Therefore the action of correspondences respects Chow-Künneth decompositions in $\sM_\mult(k,\sAb)$. The induced functors on the category of abelian varieties immediately extend to the category $\sV_c$ of Example \ref{e3.1} c), hence define a Weil cohomology $W_1^*$ on $\V_c$ with values in $\sM_\mult(k,\sAb)$. Moreover, any Weil cohomology $H^*$ with values in a pseudoabelian $\otimes$-category
factors through $\sM_\mult(k, \sAb)$ by Lemma \ref{l4.4} b) (because $\mult$ is given by a $\otimes$-ideal and $H^*$
is a $\otimes$-functor). Thus $W_1^*$ is universal, $\sM_{W}(k,\sAb) = \sM_\mult(k,\sAb)$ and Theorem \ref{wconjC} a) proves the last claim of b).
\end{proof}

\begin{cor}\label{c9.1} For any polarised abelian variety $X$ of dimension $g$ and any $i\le g/2$, the map
\[L^{g-2i}:A^i_\mult(X)\to A^{g-i}_\mult(X)\]
is an isomorphism.\qed
\end{cor}

\begin{rks} a) In Theorem \ref{t9.1} a), we have $\gamma^{\otimes \binom{2g_A}{i}\binom{2g_B}{j}+1}=0$ if $i-j$ is odd by \eqref{eq4.3} and \cite[Prop. 9.1.9]{AK2}, where $g_A=\dim A$, $g_B=\dim B$ (note that $\chi(h^1(A))= -2g_A$). The smash-nilpotency of $\gamma$ for $i-j$ even and $\ne 0$ is a big open problem.\\
b) The adequate equivalence $\sim_\mult$ is strictly coarser than algebraic equivalence (\cf Proposition \ref{prop:alg}). Indeed, for any abelian $3$-fold $A$, the group $\sM_\alg(t^3(A),\L)$ is isomorphic to the Griffiths group of $A$ tensored with $\Q$ by \cite[Th. 7.7]{tunis}, where $t^3(A)$ is a certain direct summand of $h^3(A)$\footnote{In \loccit the covariant convention is used for motives, so $\sM_\alg(t^3(A),\L)$ corresponds to what is written $\sM_\alg(\L,t_3(A))$ there.}. This group is nonzero for the generic abelian $3$-fold by Nori \cite{nori}.
\end{rks}

\subsection{Back to André's motivated cycles} \label{S:A} We need:

\begin{lemma}\label{l8.2} Let $H$ be a classical Weil cohomology. The associated Weil cohomology $(\sM^A_H,H_A)$ with values in Andr\'e's category  from \S \ref{s6.2} is tight.
\end{lemma}

\begin{proof} This is the same argument as for Lemma \ref{hwsa}. To be clear, we reproduce it. Let $l\le n$; by the faithfulness of $\uH:\sM^A_H\to \Vec_K$, where $K$ is the coefficent field of $H$,  the morphism $\Lambda^{n-l}$ in $\sM^A_H$ is inverse to $L^{n-l}: H^{l}_A(X) \to H^{2n-l}_A(X)(n-i)$, which is therefore an isomorphism. 
Also, $i^*: H^{l}_A(X)\inj H^{l}_A(Y)$ is split injective for $i:Y\inj X$ a smooth hyperplane section of $X$ by Remark \ref{rk:split}; since  the functor $\uH$ is additive and faithful and $i^*: H^{l}(X)=\uH(H^{l}_A(X))\iso H^{l}(Y)=\uH(H^{l}_A(Y))$ is an isomorphism for $l\le n-2$, the complementary summand is $0$ in this case and we get Weak Lefschetz for $H_A$ as well. For the normalised and Albanese properties, we use Propositions \ref{p5.1} and \ref{c5.1} similarly.
\end{proof}

From Lemma \ref{l8.2} and Theorem \ref{thm2tris} a), we get an induced faithful $\otimes$-functor $\weil_H^+\inj  \sM^A_H$ refining \eqref{eq6.2}. 

\begin{thm}\label{p5.5} This functor is full and becomes essentially surjective after pseudo-abelian completion, hence a $\otimes$-equivalence $\rho_H:(\weil_H^+)^\natural\iso  \sM^A_H$.
\end{thm}

\begin{proof} First, $\weil^+_H(W(X),W(Y))\to \sM^A_H(H_A(X),H_A(Y))$ is surjective for all $X,Y\in \V$ because the morphisms of $\sM^A_H$ are generated by the images of algebraic correspondences and by inverses of the Lefschetz operators (see \cite[p. 14, Def. 1]{A}). 
Finally, fullness implies essential surjectivity as in the proof of Theorem \ref{t6.4}.
\end{proof}

The diagram \eqref{eq+ab} yields the following commutative diagram of $\otimes$-functors
\begin{equation}\label{eq+abnat} 
\begin{gathered}
\xymatrix{
\sM_H\ar[r]\ar[dr]_{}\ar@/^1.8pc/[rr]^{w_H}&\weil_H^\natural\ar[r]^{\iota_H^\natural} \ar[d]_{\epsilon_H^\natural}&\weil_H^\ab\\
&(\weil_H^+)^\natural\ar[r]^{\rho_H}_{\sim}\ar[ur]^{\iota_H^{+,\natural}}& \sM^A_H
}
\end{gathered}
\end{equation}

Write $\theta_H:\sM^A_H\to\weil_H^\ab$ for the composite functor $\iota_H^{+,\natural}\circ \rho_H^{-1}$ and $\nu_H$ for the composite functor $\rho_H\epsilon_H^\natural$. We get a string of $\otimes$-functors
\[\sM_H\to \weil_H^\natural\by{\nu_H}\sM^A_H\by{\theta_H}\weil_H^\ab\]
which are all faithful by  Theorem \ref{thm2tris} a).

\begin{thm} \label{p8.3}
$\theta_H$ is an equivalence if and only if $\sM^A_H$ is abelian. In this case we have $Z (\weil_H^+)=Z(\weil_H^\ab) = Z(\sM^A_H)$.  
\end{thm}

(In characteristic $0$, this gives another proof of Theorem \ref{t6.4}.)

\begin{proof}
Since $\weil_H^\ab$ is abelian by construction, we are left to show that $\sM^A_H$ abelian implies that $\theta_H$ is an equivalence. Since $\weil_H^\ab$ is ab-initial in the sense of Definition  \ref{relhab}, it suffices to show that $\theta_H$ is exact. But this follows from its faithfulness as in the proof of Lemma \ref{p6.2} (i), by using Lemma \ref{l2.1} b). 
\end{proof}

\subsection{Finite fields} Suppose $k$ finite and $H$ classical.  Then \cite{km} shows that $H$ verifies condition C (see Example \ref{exC}), therefore $h_H:\sM_\rat\to \sM_H$ yields a Weil cohomology. Moreover,  $h_H$ yields an equivalence $\sM_H\iso \weil_H^\natural$ (Theorem \ref{wconjC} a)) but we don't know if $\sM^A_H$ is abelian.  

We obtain a faithful $\otimes$-functor 
$$\iota^\natural_H:\sM_H \simeq (\weil_H)^\natural \to \weil_H^\ab$$
which is an equivalence if $\sM_H$ is abelian.
If we push-forward $h_H$ from $\sM_H$ to $\sM_\num$, we obtain that the composition $\weil\to \sM_H\to \sM_\num$ factors canonically through $T(\weil)$.

Note that for any $(X,i)$, the Frobenius endomorphism $F_X$ of $X$ induces an endomorphism $W^i(F_X)$ of $W^i(X)$. Whence a zeta function
\[Z(W^i(F_X),t) =\exp\left(\sum_{n\ge 1} \tr(W^i(F_X)^n) \frac{t^n}{n}\right)\in A[[t]]\]
where $A$ is the $\Q$-algebra $Z(\weil)=\End_{\weil}(\un)$ \cite[Def. 3.1]{zetamot}.

On the other hand, we have the decomposition of the zeta function of $X$
\[Z(X,t)=\prod_{i=0}^{2n} P_i(t)^{(-1)^{i+1}}\]
associated to the Weil cohomology given by $\ell$-adic cohomology; by \cite{weilI}, $P_i(t)\in \Z[t]$ for all $i$. The following is a very weak independence of $\ell$ result:

\begin{prop} One has $Z(W^i(F_X),t)\mapsto P_i(t)$ for all $i$ under the homomorphism $A\to \Q_\ell$ induced by the symmetric monoidal functor given by $\ell$-adic cohomology, for any $\ell\ne p$. Similar statement with crystalline cohomology.
\end{prop}

\begin{proof} One may compute $\tr(W^i(F_X)^n)$ after applying  the said functor.
\end{proof}

\section{Theory over a base}\label{s10}

\subsection{Deninger-Murre correspondences}\label{sB.1} Let $S$ be smooth quasiprojective over $k$; we then have Deninger-Murre's category of  Chow correspondences between smooth projective $S$-schemes, and the corresponding category of Chow motives $\sM(S,\sV)$ modelled on an admissible category $\sV\subseteq \Sm^\proj(S)$ \cite{den-mur}. We note that these definitions still make sense if $S$ is only smooth over a Dedekind domain, thanks to Fulton's corresponding intersection theory \cite[\S 20.2]{fulton}.

We keep the same axioms for (generalised) Weil cohomologies; the theory developed above then extends to this situation without change. We thus get a universal category $\weil(S,\sV)$ provided with an invertible object $\L_W$ and a universal Weil cohomology $W_\sV$ with values in $\weil(S,\sV)$. Similarly, we get a universal abelian Weil cohomology $W_\sV^\ab$ and their Lefschetz variants.

\subsection{Base change}\label{sB.2} Let $S,T$ be as in \S \ref{sB.1}, and let $f:T\to S$ be a $k$-morphism: it yields a $\otimes$-functor $f^*:\sM(S,\sV)\to \sM(T,f^*\sV)$. Any Weil cohomology $H$ over $T$ induces a Weil cohomology $f_* H=H\circ f^*$ over $S$, with same target as $H$, whence a canonical $\otimes$-functor $f^*:\weil(S,\sV)\to \weil(T,f^*\sV)$, and we have the following ``trivial base change'' theorem:

\begin{thm}\label{tB.1} The comparison morphism $W_{f^*\sV}\to (f^*)_*W_\sV$ is an isomorphism.
\end{thm}

\begin{proof} This follows from (the generalisation of) Lemma \ref{l4.2}.
\end{proof}

\newcommand{\Aut}{\operatorname{Aut}}

\begin{exs} a) $T=\Spec E$ for $E$ an extension of $k$. Then $\Aut_k(E)$ acts on $\sM(E,\sV_E)$; for $X\in \sV$, the objects $W_{\sV_E}^ i(X_E)$ are invariant under this action by Theorem \ref{tB.1} and 
$\Aut_k(E)$ acts on them.\\ 
b) $k$ algebraically closed. For clarity, suppose that $\sV=\Sm^\proj(S)$. For $s\in S(k)$, we have $X_s\in \Sm^\proj(k)$ for $X\in \sV$ and $s^*:\weil(S)\to \weil(k)$. By Theorem \ref{tB.1}, $s^*$ maps $W^i(X)$ to $W^i(X_s)$ for any $i\ge 0$. But we don't have any smooth and proper base change at this stage, of course.
\end{exs}

\appendix

\section{Tate triples and gradings}

\subsection{A complement on the $2$-functor $T$}\label{s3.4}
Let $\sC\in \Add^\rig$.  The canonical $\otimes$-functor $\lambda_\sC:\sC\to T(\sC)$ yields a $\otimes$-functor $\lambda_\sC^{(\Z)}:\sC^{(\Z)}\to T(\sC)^{(\Z)}$, hence by universality an exact $\otimes$-functor
\begin{equation}\label{eq2.1}
T(\sC^{(\Z)})\to T(\sC)^{(\Z)}.
\end{equation}

\begin{prop}\label{p2.1} The functor \eqref{eq2.1} is an equivalence of categories.
\end{prop}

\begin{proof} Let $\sA\in \Ex^\rig$. A $\otimes$-functor $F:\sC^{(\Z)}\to \sA$ yields a functor $\tilde F:\sC\to \sA^\Z$ (see Remark \ref{r1} again), hence to $\hat{\sA}^\Z$ by composition with $y_\sA^\Z$. This composition is a strong $\otimes$-functor, which carries any object $C\in \sC$ to a dualisable object. By lemma \ref{l3.3} a), this object is a direct summand of an object of $\sA^{(\Z)}$, which shows that \emph{$\tilde F$ takes values in $\sA^{(\Z)}$}, and defines a strong $\otimes$-functor for the tensor structure of the latter. By the universality of $T$, $\tilde F$ factors uniquely through an exact $\otimes$-functor $T(\sC)\to \sA^{(\Z)}$; composing with the inclusion $\sA^{(\Z)}\inj \sA^{\Z}$ we obtain an exact $\otimes$-functor $T(\sC)^{(\Z)}\to \sA$ whose restriction to $\sC^{(\Z)}$ is $F$. The rest is history.
\end{proof}





\subsection{Ungraded and graded Weil cohomologies}

To express the next theorem, we introduce definitions generalising the notion of cohomology theory from \cite[V.3.1.1 and A1.1.5]{saa}.

\begin{defn} \label{d6.2} a) An \emph{additive Tate triple} is a triple $(\sC,L_\sC,w_\sC)$ where $(\sC,L_\sC)\in \Add^\rig_*$ and $w_\sC:\sC\to \sC^{(\Z)}$ is a weight $\otimes$-grading of $\sC$ as in Definition \ref{d3.5} b), such that $L_\sC$ is of weight $2$. An additive Tate triple is \emph{pseudo-abelian} (\resp \emph{abelian}) if so is $\sC$.\\
b) A \emph{graded Weil cohomology with values in $(\sC,L_\sC,w_\sC)$} is a strong $\otimes$-functor $H:\Corr[\L,\L^{-1}]\to\fake{\sC}$ provided with an isomorphism $\Tr:H^2(\P^1)\iso L_\sC$ and such that $\un\allowbreak\to H^0(X)$ is an isomorphism if $X$ is geometrically connected. (See Notation \ref{n3.1} for $\fake{\sC}$.)
\end{defn}

\begin{prop}\label{p6.4} Let $(\sC,L_\sC,w_\sC)$ be an additive Tate triple. Then the weight structure $w_\sC$ induces a weight structure $w_{T(\sC)}$ on $T(\sC)$ such that $\lambda_\sC:\sC\to T(\sC)$ induces a morphism of Tate triples $(\sC,L_\sC,w_\sC)\to (T(\sC),T(L_\sC),w_{T(\sC)})$.
\end{prop}

\begin{proof} The weight functor $w_\sC$ induces an exact $\otimes$-functor $T(w_\sC):T(\sC)\to  T(\sC^{(\Z)})$; composing with  \eqref{eq2.1} yields $w_{T(\sC)}$. By functoriality, its composition with $T(\bigoplus)$ is the identity; composition of $w_{T(\sC)}$ with the sum functor $T(\sC)^{(\Z)}\to T(\sC)$ is then the identity by Proposition \ref{p2.1}.
\end{proof}

The following definitions are parallel to Definitions \ref{d5.1} and \ref{d5.2}.

\begin{defn}\label{d5.1w} Let $\Add^\otimes_{*,w}$ be the $2$-category whose 
\begin{itemize}
\item objects are Tate triples;
\item $1$-morphisms $(\sC,L_\sC,w_\sC)\to (\sD,L_\sD,w_\sD)$ are pairs $(F,u)\in \break\Add^\otimes_*((\sC,L_\sC),(\sD,L_\sD))$ such that $w_\sD \circ u=u^{(\Z)}\circ w_\sC$.
\item $2$-morphisms $\theta:(F,u)\Rightarrow (F',u')$ in $\Add^\otimes_{*,w}$ are $2$-morphisms $\theta:(F,u)\Rightarrow (F',u')$ in $\Add^\otimes_*$ such that $w_\sD*\theta=\theta^{(\Z)}*w_\sC$.
\end{itemize}
We define $\Add^\rig_{*,w}$ and $\Ex^\rig_{*,w}$ similarly.
\end{defn}

\begin{defn}\label{d5.2w} Let $(\sC,L_\sC,w_\sC)\in \Add^\otimes_{*,w}$. We denote by \break $\Weil_{\gr}(k,\V; \sC,L_\sC,w_\sC)$ the category whose objects are graded Weil cohomologies $(H,\Tr)$ with values in $(\sC,L_\sC,w_\sC)$.  A morphism $\phi:(H,\Tr)\allowbreak\to ({H'},\Tr')$ in
$\Weil_{\gr}(k,\V; \sC,L_\sC,w_\sC)$ is a natural transformation $\phi:H\Rightarrow {H'}$ such that $\Tr=\Tr'\circ \phi_{\P^1}$. 
\end{defn} 

\begin{prop}\label{p6.5}
The induced $2$-functor $T:\Add^\rig_{*,w}\to \Ex^\rig_{*,w}$ is a $2$-left adjoint to the forgetful $2$-functor in the other direction.
\end{prop}
\begin{proof} Follows from Proposition \ref{p6.4}.
\end{proof}

As in Lemma \ref{l4.2}, we have

\begin{lemma}\label{l4.2w} The category $\Weil_\gr(k,\V; \sC,L_\sC,w_\sC)$ is a groupoid.\qed
\end{lemma}

As in Construction \ref{cons1}, Definition \ref{d5.2w} provides a strict $2$-functor
\begin{equation}\label{eq9.2}
\Weil_\gr(k,\V; -): \Add^\otimes_{*,w}\to \Cat.
\end{equation}

\begin{cons}\label{r6.1}
To a graded Weil cohomology $H$ as in Definition \ref{d6.2} b), we associate the Weil cohomology $H^*$ with values in $(\sC,L_\sC)$ obtained by composing with the twisted weight functor $\fake{w_\sC}:\fake{\sC}\to \sC^{(\Z)}$ of Definition \ref{d3.5} a)  (Koszul rule on the range). This defines a functor
\[\Weil_\gr(k,\V; \sC,L_\sC,w_\sC)\to \Weil(k,\V; \sC,L_\sC)\]
which is $2$-natural in $(\sC,L_\sC,w_\sC)$.
 \end{cons}

\subsection{Gradings and weights} 
Note that in Construction \ref{r6.1}, $H^*$ has weights in the sense of Definition \ref{d3.3}. Conversely:

\begin{prop}\label{p9.1} To a Weil cohomology $H$ with weights and values in $(\sC,L_\sC)$ is canonically  associated a graded Weil cohomology with values in $(\sD,L_\sD,w_\sD)$, where $\sD$ is a full $\otimes$-subcategory of $\sC$ containing the image of $H$ and $L_\sD=L_\sC$.
\end{prop}

\begin{proof} Let $\sD$ be the (strictly full) thick subcategory of $\sC$  additively generated by the $H^i(M)$ for $M\in \Corr[\L,\L^{-1}]$. If $M,N\in \Corr[\L,\L^{-1}]$ and $i,j\in \Z$, then $H^i(M)\otimes H^j(N)$ is isomorphic to a direct summand of $H^{i+j}(M\otimes N)$, hence $\sD$ is stable under $\otimes$ and $H^*$ takes its values in $\sD^{(\Z)}$.
Define a $\Z$-$\otimes$-grading 
\[w_\sD:\sD\to \sD^{(\Z)}\]
by sending $H^i(M)$ to $H^i(M)[i]$ and $\phi\in \sD(H^i(M),H^i(N))$ to $\phi[i]$ (see Lemma \ref{l3.3} b) for the notation). Since $H$ has weights, this does define a functorial section of the direct sum functor, which is symmetric monoidal, and $H=\bigoplus H^i:\Corr[\L,\L^{-1}]\to\fake{\sD}$ is the desired graded Weil cohomology.
\end{proof}

On the other hand, 

\begin{prop}\label{p9.2} Let $(\sC,H)$ be a Weil cohomology. Then the category of functors $F:\sC\to \sC'$ in $\Add^\otimes$ such that $F_*H$ has weights is not empty and has a initial object $F_w:\sC\to\sC_w$. If $H$ is tight (Definition \ref{d5.2b}), so is $(F_w)_*H$.
\end{prop}

\begin{proof} Define $\sC_w$ as the additive quotient of $\sC$ by the $\otimes$-ideal generated by morphisms in $\sC(H^i(M),H^j(N))$ for $M,N\in \Corr[\L,\L^{-1}]$ and $i\ne j$. The claim on tightness is Lemma \ref{l8.1}.
\end{proof}

\subsection{An adjunction} Let $\Weil_\gr(k,\V)$ be the $2$-category associated to \eqref{eq9.2} in the same way as $\Weil(k,\V)$ is associated to \eqref{eq5.10} in Remark \ref{r4.1}. Construction \ref{r6.1} provides a ``forgetful''  $2$-functor
\[U:\Weil_\gr(k,\V)\to \Weil(k,\V).\]

\begin{prop} This $2$-functor has a  $2$-left adjoint $\gr$.
\end{prop}

\begin{proof} Let $(\sC,H)$ be a Weil cohomology. Define a graded Weil cohomology $(\sD,w_\sD,H^\gr)$ by composing the constructions of Propositions \ref{p9.2} and \ref{p9.1}. Namely, the underlying category $\sD$ is the full subcategory of the category $\sC_w$ of Proposition \ref{p9.2} described in the proof of Proposition \ref{p9.1}. If $(\sC_1,w_{\sC_1},H_1)$ is a graded Weil cohomology, a morphism $F:(\sC,H^*)\to U(\sC_1,w_{\sC_1},H_1)$ has an underlying functor $F:\sC\to \sC_1$ which factors through $\sC_w$ by the weight property of $U(H_1)$ and then restricts to $\sD$, and the construction of $w_\sD$ shows that the composition $\sD\to \sC_1\by{w_{\sC_1}}\sC_1^{(\Z)}$ factors uniquely through $w_\sD$.
\end{proof}

\subsection{Universal graded Weil cohomology}

\begin{thm}\label{thm2w} The $2$-functor \eqref{eq9.2} is strongly $2$-representable.
\end{thm}

\begin{proof} By definition, $(\weil,W)$ is $2$-initial in $\Weil(k,\V)$, hence it is formal that $\gr(\weil,W)$ is $2$-initial in $\Weil_\gr(k,\V)$.
\end{proof}

Similarly to Corollary \ref{cor:thm2} we get an abelian variant of this theorem by making use of Proposition \ref{p6.5}. Since the Propositions \ref{p9.2} and \ref{p9.1} preserve tightness we also get the graded analougue of 
Theorem \ref{thm2m}.


\begin{thebibliography}{II}
\bibitem{A} Y. Andr\'e: Pour une th\'eorie inconditionnelle des motifs, {\it Publ. Math. IH\'ES} {\bf 83} (1996) 5--49.
\bibitem{A2} Y. André: Déformation et spécialisation de cycles motivés, \emph{J. Inst. Math. Jussieu} {\bf 5} (2006), 563 --603.
\bibitem{andremotifs} Y. André: Une introduction aux motifs: motifs purs, motifs mixtes, périodes, Panoramas \& synthèses {\bf 17}, SMF, 2004.
\bibitem{AK2} Y. André \& B. Kahn: Nilpotence, radicaux et structures monoïdales (with an appendix by P. O'Sullivan), {\it Rend. Sem. Mat. Univ. Padova} {\bf 108} (2002), 107--291.
\bibitem{AK} Y. Andr\'e \& B. Kahn: Construction inconditionnelle de groupes de Galois motiviques, {\it C. R. Acad. Sci. Paris} {\bf 334} (2002) 989--994.
\bibitem{AK3} Y. Andr\'e \& B. Kahn: Erratum: nilpotence, radicaux et structures monoïdales, 
{\it Rend. Sem. Mat. Univ. Padova} {\bf 113} (2005), 125--128.  
\bibitem{AyW} J. Ayoub: Nouvelles cohomologies de Weil en caract\'eristique positive, \emph{Algebra \& Number Theory} {\bf 14} (2020) 1747--1790.
\bibitem{AyW2} J. Ayoub: Weil cohomology theories and their motivic Hopf algebroids, \emph{Indagationes Mathematicae} Special issue dedicated to the memory of Jacob Murre. \url{https://doi.org/10.1016/j.indag.2024.09.009}.
\bibitem{BVHP} L. Barbieri-Viale, A. Huber \& M. Prest: Tensor structure for Nori motives {\it Pacific Journal of Mathematics} {\bf 306} No. 1 (2020) 1-30
\bibitem{BVK} L. Barbieri-Viale \& B. Kahn: A universal rigid abelian tensor category, \emph{Doc. Math.} {\bf 27} (2022), 699--717.
\bibitem{BKP} R. Blackwell, G.M. Kelly and A.J. Power: Two-dimensional monad theory, {\it J. pure appl. Algebra} {\bf 59} (1989) 1--41.
\bibitem{bo} S. Bloch, A. Ogus: Gersten's conjecture and the homology of schemes, {\it Ann. Sci. \'Ec. Norm. Sup.} {\bf 7} (1974), 81--201. 
\bibitem{bbki} N. Bourbaki Groupes et algèbres de Lie, Ch. 8: \emph{Algèbres de Lie semi-simples déployées}, CCLS, 1975.
\bibitem{CisDeg} D.C. Cisinski \& F. D\'eglise: Mixed weil cohomologies, \emph{Adv. in Math.} {\bf 230} (2012), 55-130. 
\bibitem{CEOP} K. Coulembier, P. Etingof, V. Ostrik, Victor \& B. Pauwels: Monoidal abelian envelopes with a quotient property, \emph{J. reine angew. Math.} {\bf 794} (2023) 179--214.
\bibitem{criteres} P. Deligne: Théorèmes de Lefschetz et critères de dégénérescence de suites spectrale, \emph{Publ. Math. IHÉS} {\bf 35} (1968), 107--126.
\bibitem{weilI} P. Deligne: la conjecture de Weil, I, \emph{Publ. Math. IHÉS} {\bf 43} (1974), 273--307.
\bibitem{weilII} P. Deligne: la conjecture de Weil, iI, \emph{Publ. Math. IHÉS} {\bf 52} (1980), 137--252.
\bibitem{delmil} P. Deligne \& J.S. Milne: Tannakian Categories, in \emph{Hodge Cycles, Motives, and Shimura Varieties} LNM {\bf 900}, 1982, pp. 101-228.
\bibitem{den-mur} C. Deninger \& J. P. Murre: Motivic decomposition of abelian schemes and the Fourier transform, 
\emph{J. reine angew. Math.} {\bf 422} (1991), 201--219.
\bibitem{fulton} W. Fulton: \emph{Intersection theory} (2nd ed.), Springer, 1998.
\bibitem{gz} P. Gabriel \& M. Zisman: \emph{Calculus of fractions and homotopy theory}, Springer, 1967.
\bibitem{SGA1} A. Grothendieck: Catégories fibrées et descente, Exposé VI {\it in} {\it Revêtements étales et groupe fondamental (SGA 1)}, revised ed., SMF, 2002, 
\bibitem{gro-standard} A. Grothendieck: Standard conjectures on algebraic cycles, {\it Algebraic Geometry} (Internat. Colloq., Tata Inst. Fund. Res., Bombay, 1968) 193--199, Oxford Univ. Press, 1969.
\bibitem{hms} A. Huber, S. Müller-Stach: \emph{Periods and Nori motives} (with contributions by Benjamin Friedrich and Jonas von Wangenheim), Springer, Cham, 2017.
\bibitem{jannsen3} U. Jannsen: Motives, numerical equivalence and semi-simplicity, \emph{Invent. Math.} {\bf 107} (1992), 447--452.
\bibitem{zetamot} B. Kahn: Zeta functions and motives, \emph{Pure appl. math. quarterly} {\bf 5} (2009), 507--570.
\bibitem{ss} B. Kahn: On the semi-simplicity of Galois actions, \emph{Rend. Sem. Mat. Univ. Padova} {\bf 112} (2004), 97--102.
\bibitem{zetaL} B. Kahn: \emph{Zeta and $L$-functions of varieties and motives}, LMS Lect. Notes Series {\bf 462}, Cambridge Univ. Press, 2020. Engl. translation of: Fonctions zêta et $L$ de variétés et de motifs, \emph{Coll. Nano} {\bf 103} Calvage \& Mounet, 2018.    
\bibitem{tunis} B. Kahn: Albanese kernels and Griffiths groups (with an appendix by Yves André, \emph{Tunisian J. Math.} {\bf 3} (2021), 589-656.
\bibitem{standard-schur} B. Kahn: Universal rigid abelian tensor categories and Schur finiteness, preprint, 2022, \url{https://arxiv.org/abs/2203.03572}.
\bibitem{birat-pure} B. Kahn, R. Sujatha: Birational motives, I: pure birational motives, \emph{Annals of $K$-theory} {\bf 1} (2016), 379--440.
\bibitem{km} N. Katz, W. Messing {\it Some consequences of the Riemann hypothesis for varieties over finite fields}, Invent. Math. {\bf 23} (1974), 73--77.  
\bibitem{kdix} S. Kleiman: Algebraic cycles and the Weil conjectures,  in \emph{Dix exposés sur la cohomologie des schémas}, Masson-North-Holland, 1968. 
\bibitem{kst} S. Kleiman:  The standard conjectures, in {\it Motives} (Seattle, WA, 1991), 3–20, Proc. Sympos. Pure Math., {\bf 55} (1), AMS, 1994.
\bibitem{kimura} S-i. Kimura: Chow groups are finite-dimensional, in some sense,  {\it Math. Ann.} {\bf 331} (2005), 173--201.
\bibitem{krause} H. Krause \emph{Homological Theory of Representations}, Cambridge Univ. press, 2022.
\bibitem{ku} K. Künnemann:  A Lefschetz decomposition for Chow motives of abelian schemes, \emph{Invent. Math.} {\bf 113} (1993), 85--102. 
\bibitem{levine} M. Levine: {\it Mixed Motives}, Math. Surveys and Monographs {\bf 57} AMS, Prov. 1998.
\bibitem{mcl} S. Mac Lane: Categories for the working mathematician (2nd ed.), Springer, 1998.
\bibitem{murre} J. P. Murre: On the motive of an algebraic surface,  \emph{J. reine angew. Math.} {\bf 409} (1990), 190--204.
\bibitem{nori} M. V. Nori: Cycles on the generic abelian threefold, \emph{Proc. Indian Acad. Sci. Math. Sci.} {\bf 99:3} (1989), 191--196.
\bibitem{os} P. O'Sullivan: Algebraic cycles on an abelian variety, \emph{J. reine angew. Math.} {\bf 654} (2011), 1--81.
\bibitem{os2} P. O'Sullivan: Super Tannakian hulls, preprint (2020), \url{https://arxiv.org/abs/2012.15703}.
\bibitem{saa} N. Saavedra Rivano, \emph{Cat\'egories tannakiennes}, LNM {\bf 265}, Springer, 1972.
\bibitem{scholl} A. Scholl: Classical motives,  in {\it Motives} (Seattle, WA, 1991), Proc. Sympos. Pure Math., {\bf 55} (1), AMS, 1994, 163–187.
\bibitem{scholze} P. Scholze: $p$-adic geometry in {\it Proceedings of the ICM} Rio de Janeiro 2018. Vol. I. Plenary lectures, 899–933, World Sci. Publ., Hackensack, NJ, 2018.
\bibitem{serre-motifs} J.-P. Serre: Motifs, \emph{Astérisque} {\bf 198-199-200} (1991), 333--349.
\bibitem{sm} O. Smirnov: Graded associative algebras and Grothendieck standard conjectures, \emph{Invent. Math.} {\bf 128} (1997), 201--206.
\bibitem{woodshole} J. Tate: Algebraic cycles and poles of zeta functions, \emph{Arithmetical Algebraic Geometry} (Proc. Conf. Purdue Univ., 1963), Harper \& Row, New York, 1965, 93–110.
\bibitem{tate} J. Tate: Conjectures on algebraic cycles in $\ell$-adic cohomology,   in {\it Motives} (Seattle, WA, 1991), Proc. Sympos. Pure Math., {\bf 55} (1), AMS, 1994, 71--83.
\bibitem{voe} V. Voevodsky: A nilpotence theorem for cycles algebraically equivalent to zero, \emph{IMRN} {\bf 4} (1995), 187--199.
\bibitem{voeicm} V. Voevodsky: $\A^1$--homotopy theory, \emph{Proc. ICM} {\bf I} (Berlin, 1998). Doc. Math. 1998, Extra Vol. I, 579--604.
\end{thebibliography}
\end{document}